\documentclass[Threeside,11pt]{article}
\usepackage{etoolbox}
\makeatletter
\patchcmd{\thebibliography}
{\list}
{\small          
	\setstretch{1.2}
	\list}
{}{}
\makeatother
\usepackage{setspace}
\usepackage{tikz}
\usepackage{float}
\definecolor{myPurple}{RGB}{109,90,207}
\definecolor{myGreen}{RGB}{4,130,67}
\definecolor{myGold}{RGB}{253,177,71}
\definecolor{myBurgundy}{RGB}{63,1,44}
\definecolor{myTeal}{RGB}{8,119,127}
\definecolor{myCream}{RGB}{255,216,177}
\definecolor{myOrange}{RGB}{225,119,1}
\usetikzlibrary{arrows.meta,bending,calc}
\tikzset{
	ball3Donly/.style={
		circle, minimum size=1.4cm,
		shading=ball,
		ball color=myPurple,
		draw=myPurple
	}
}
\pgfmathsetmacro{\wNear}{2.0pt}
\pgfmathsetmacro{\wMid}{1.3pt}
\pgfmathsetmacro{\wFar}{0.7pt}
\pgfmathsetmacro{\bend}{6mm}
\usepackage[latin1]{inputenc}
\usepackage[english]{babel}
\usepackage{amsmath}
\usepackage{amsthm}
\IfFileExists{algorithm.sty}{\usepackage{algorithm}}{}
\usepackage{cite}
\usepackage{amsfonts}
\usepackage{amssymb}
\usepackage{mathrsfs}
\usepackage{graphicx}
\usepackage{colortbl,dcolumn}
\usepackage{marvosym}
\IfFileExists{ifsym.sty}{\usepackage{ifsym}}{}
\usepackage{paralist}
\usepackage{xcolor}
\usepackage{subcaption}
\usepackage{adjustbox}
\usepackage{enumitem}
\usepackage{booktabs,tabularx}
\usepackage[colorlinks,citecolor=blue,
linkcolor=blue,hypertexnames=false]{hyperref} 
\allowdisplaybreaks
\IfFileExists{mathx10.tfm}{
\DeclareFontFamily{U}{mathx}{\hyphenchar\font45}
\DeclareFontShape{U}{mathx}{m}{n}{
	<5> <6> <7> <8> <9> <10>
	<10.95> <12> <14.4> <17.28> <20.74> <24.88>
	mathx10
}{}
\DeclareSymbolFont{mathx}{U}{mathx}{m}{n}
\DeclareMathAccent{\widecheck}{0}{mathx}{"71}
}{\let\widecheck\check}
\usepackage{graphicx}
\usepackage{caption}
\usepackage{amsthm}

\newtheorem{theorem}{Theorem}[section]
\newtheorem{proposition}[theorem]{Proposition}
\newtheorem{lemma}[theorem]{Lemma}
\newtheorem{corollary}[theorem]{Corollary}
\newtheorem{fact}{Fact}

\theoremstyle{definition}
\newtheorem{definition}[theorem]{Definition}
\theoremstyle{remark}
\newtheorem{remark}[theorem]{Remark}

\numberwithin{equation}{section}

\newcommand{\R}{\mathbb R}
\newcommand{\C}{\mathbb C}

\allowdisplaybreaks
\theoremstyle{remark}

\newcommand{\Rp}{\mathbb{R}_{>0}}
\newcommand{\ord}{\operatorname{ord}}
\newcommand{\code}[1]{\begingroup\def\_{\textunderscore\allowbreak}\texttt{#1}\endgroup}
\newenvironment{romanlist}{\begin{enumerate}%
  }{\end{enumerate}}
\newenvironment{alphalist}{\begin{enumerate}%
  }{\end{enumerate}}
\hypersetup{pdftitle={Finiteness of planar central configurations of five bodies for all positive masses: On Smale's sixth problem},pdfauthor={Zhe Chen and Zhicheng Tong},pdfsubject={Central configurations and Smale's sixth problem}}

\begin{document}
\title{{\sl Finiteness of planar central configurations of five
bodies for all positive masses: On Smale's sixth problem}}
\author{Zhe Chen\thanks{School of Wind Energy, Technical University of Denmark,
2800 Kongens Lyngby, Lyngby-Taarb\ae k Municipality, Copenhagen, Denmark.
Email: \url{hitchenzhe@gmail.com}}
\and
Zhicheng Tong\thanks{School of Mathematics, Jilin University,
Changchun 130012, P. R. China. Email: \url{tongzc25@jlu.edu.cn}}}
\date{}
\maketitle

\begin{abstract}
{
Combining local boundary analysis, algebraic curves and the computational
statements specified in the appendix, we prove that the planar Newtonian
five-body problem has finitely many central configurations up to rotation,
translation and scaling for every choice of positive masses. This gives the
five-body statement of Smale's sixth problem. We use the singular-sequence
classification of Albouy and Kaloshin and resolve the corresponding boundary
equations through weighted blow-up charts, rank calculations and local
compatibility conditions. For the two remaining diagrams, we prove that the
product $\Delta$ of all squared mutual distances vanishes. On a hypothetical
solution curve containing a positive arc, a pole of a four-product therefore
gives a zero of $\Delta$. A pole of $\Delta$ elsewhere on the completed curve,
together with $\Delta^4=\prod_{i<j}p_{ij}$, then contradicts the same local
vanishing result. The computational statements for Diagrams~2, 3 and~6 involve
exact algebra, interval arithmetic and ball arithmetic.
}
\\
        \\
        {\bf Keywords:} {Central configurations, planar five-body problem,
        Smale's sixth problem, singular sequences, algebraic curves, computer-assisted proof}
        \vspace{2mm}
        \\
        {\bf2020 Mathematics Subject Classification:} {70F10, 
        70F15, 
        70G55, 
        14H05 
}
\end{abstract}

\tableofcontents

\section{Introduction}\label{sec:intro}

The finiteness of central configurations asks whether fixed positive masses admit
only finitely many shapes generating relative equilibria of the Newtonian
$n$-body problem. We study the planar five-body case. Our argument combines the
singular-sequence classification of Albouy and Kaloshin with local analysis of
the normalized equations and a product identity on an algebraic curve.

For masses $m_1,\dots,m_n>0$, a collision-free configuration
$q=(q_1,\dots,q_n)\in(\R^2)^n$ is a \emph{central configuration} if
\[
 \lambda(q_k-c)=\sum_{l\ne k}m_l\frac{q_l-q_k}{|q_l-q_k|^3},
 \qquad k=1,\dots,n,\qquad
 c:=\frac{\sum_{j=1}^n m_jq_j}{\sum_{j=1}^n m_j},
\]
for some $\lambda\in\R$. We identify configurations related by translation,
rotation and positive scaling. The question is whether the resulting set is
finite for every fixed positive mass vector.

\subsection*{Background and development of the problem}
Central configurations determine motions of constant shape: rigid rotation
gives relative equilibria, while a time-dependent scale and rotation give
homographic motions. Their classification therefore connects the geometry of
the configuration space with the dynamics of the Newtonian equations; see
Wintner~\cite{Wintner}. This connection also underlies the stability analysis of
elliptic relative equilibria, for example for the $(1+n)$-gon configurations
studied by Hu, Long and Ou~\cite{HLO20}.

The classical solutions of Euler and Lagrange give the five classes for three
bodies. Moulton~\cite{Moulton} established the exact number $n!/2$ of collinear
classes for arbitrary positive masses. The general finiteness question appears
in Wintner~\cite{Wintner} and became Problem~6 in Smale's list of mathematical
problems for the twenty-first century~\cite{Smale}. Hampton and
Moeckel~\cite{HM} proved finiteness for four positive masses. Albouy and
Kaloshin~\cite{AK} subsequently introduced an analysis of singular sequences
that gives another proof for four bodies and a five-body theorem with an
algebraic exceptional set. The four-body theory was further extended to
arbitrary nonzero real masses by Yu and Zhu~\cite{YZ}.

After fixing the moment of inertia, Shub~\cite{Shub71} proved that central
configurations for fixed positive masses stay away from collisions.
Using analytic continuation, Xia~\cite{Xia91} obtained exact counts for
certain open sets of positive mass vectors. In a complementary direction,
Hampton and Jensen~\cite{HJ15} proved that a fixed planar configuration of
$n-1$ masses admits only finitely many extensions to a central configuration
after a prescribed additional mass is added. Further questions on central
configurations, symmetry, stability and singularities are collected
in~\cite{ACS12}.

For five bodies, the theorem of Albouy and Kaloshin states that there is a
closed algebraic subset $A$ of codimension two in the mass space such that
\[
 m\in\Rp^5\setminus A
 \quad\Longrightarrow\quad
 \#\{\text{positive normalized central configurations at }m\}<\infty.
\]
Their method associates one of sixteen two-coloured diagrams with each singular
sequence of complex normalized solutions. The leading equations impose mass
relations, and two singular limits on the same algebraic curve yield the
codimension-two restriction. Jensen and Leykin~\cite{JL} developed a tropical
geometric method and completed the computation for $n\le5$, recovering generic
finiteness. Chang and Chen~\cite{CC1,CC2} organized the singular-sequence method
into symbolic procedures for determining diagrams, orders and mass relations,
with applications to the planar six-body problem.

In Part~I of their study of the planar six-body problem, the
authors~\cite{CT26} establish a reduction from an infinite family of
positive normalized central configurations to a singular germ with an
unbounded seven-product. Their first-limit diagram analysis confines
the mass vector to a proper algebraic exceptional set. This work
extends the curve reduction and local singular-limit methods used
here to the six-body setting.

Complementary approaches describe particular mass vectors and geometric
classes. Lee and Santoprete~\cite{LS} classified the isolated central
configurations of five equal masses using homotopy continuation and interval
verification. Moczurad and Zgliczy\'nski~\cite{MZ19} obtained complete
computer-assisted lists for $n=5,6,7$ equal masses; their later work~\cite{MZ26}
gives exact counts for selected five-body mass vectors on exceptional loci.
Chen and Hsiao~\cite{CH18} derived geometric restrictions on strictly convex
five-body central configurations, while Deng and Hampton~\cite{DH} gave a
computer-assisted proof of finiteness within the equilateral pentagonal class. These results develop the
mass-specific and geometric aspects of the same classification problem.

The continuing interest in central configurations also encompasses their
geometry, symmetry and uniqueness. Albouy, Fu and Sun~\cite{AFS08} proved
that a convex four-body central configuration is symmetric across a diagonal
if and only if the two masses on the other diagonal are equal.
Santoprete~\cite{San21a,San21b} proved that each of the
co-circular and trapezoidal four-body classes contains at most one central
configuration for each cyclic ordering of the masses. Hampton~\cite{Ham05}
constructed stacked planar five-body central configurations containing a
Lagrange equilateral three-body subconfiguration. Chen and Hsiao~\cite{CH12}
studied convex configurations that are not strictly convex. Sun, Xie and
You~\cite{SXY23} gave a computer-assisted proof of uniqueness for convex
four-body configurations with a prescribed order on a closed domain of
positive mass vectors. Recent work examines four-body kite
configurations~\cite{RobK25}, symmetric concave four-body configurations
with two pairs of equal masses~\cite{LX26}, rhomboidal five-body
configurations~\cite{CCRh26}, and convex pentagons with three right interior
angles~\cite{FGMS26}. Tang and Zhu~\cite{TZ24} derive nonexistence criteria
for centered co-circular configurations for several classes of unequal
masses under power-law potentials. These geometric and symmetry questions
describe how central configurations are organized within the solution set.

The local structure of this set is another active direction. Hampton~\cite{Hamp24}
introduced twist-vector coordinates on the tangent space to planar
configurations and applied them to four-body central configurations;
Albouy and Sun~\cite{AS25} study the eigenvalues associated with a central
configuration, and
Sun, Xie and You~\cite{SXY25} investigate degeneracy of planar central
configurations. Natrup, Wang and Wang~\cite{NWW25} derive upper and lower
bounds for the number of nondegenerate planar central configurations for
homogeneous potentials.
Together, these studies connect the defining equations and their linearization
with the geometry and finiteness of their solutions.

The dynamics generated by central configurations provides a further source of
interest. Quantitative linear stability of planar elliptic relative equilibria
is studied by Hu, Ou and Sun~\cite{HOS25}, while the $(1+7)$-gon case and
spectral instability of regular $n$-gon elliptic relative equilibria are treated
in~\cite{OS23,OW25}. Hu, Ou and Tang~\cite{HOT23} use index theory for spatial
elliptic relative equilibria. Related variational questions include the Morse
index of homothetic motions~\cite{OP25} and index theory for collision,
parabolic and hyperbolic solutions~\cite{HOY21}. These developments connect
central configurations with stability questions and variational methods in
celestial mechanics.

Recent preprints of Tomar~\cite{FiveT} and Liu and Liu~\cite{LL26planar}
address finiteness for all positive masses in the planar five-body problem.
Both use the singular-sequence framework of Albouy and Kaloshin. Tomar combines
valuation balance with exact rational certificates and separate counts for
special mass vectors. Liu and Liu combine a lower bound for the potential at
physical configurations with asymptotic estimates and local analytic
reductions. For central configurations of affine dimension $n-2$,
Moeckel~\cite{Moe01} proved generic finiteness. In the spatial five-body
problem, Hampton and Jensen~\cite{HJ} obtained finiteness outside explicitly
given exceptional mass conditions. For configurations of exact affine dimension
three, Liu and Liu~\cite{LL26dziobek} use the Dziobek equations and local analysis of zeros and
poles on a compact complex curve.

\subsection*{Main result}
We obtain the following statement. Positive normalized configurations are
defined by the polynomial system in Section~\ref{sec:ak}.

\begin{theorem}\label{thm:main}
Assume the computational statements listed in Appendix~\ref{app:verification}.
For every $(m_1,\dots,m_5)\in\Rp^5$ there are finitely many positive normalized
central configurations of the planar five-body problem. Equivalently, the
planar five-body problem has finitely many central configurations up to
rotation, translation and scaling.
\end{theorem}

The central feature of our proof is the use of the product of all squared
mutual distances. Writing $r_{kl}^2:=z_{kl}w_{kl}$ in the complex coordinates of
Section~\ref{sec:ak}, we set
\[
 \Delta:=\prod_{k<l}r_{kl}^2,\qquad
 p_{ij}:=r_{ij}^2r_{kl}^2r_{lh}^2r_{hk}^2,
 \qquad \{k,l,h\}=\{1,\dots,5\}\setminus\{i,j\}.
\]
We use the four-products of Albouy and Kaloshin, recalled in Section~\ref{sec:ak}.
The product identity converts the local distance estimates into a global
contradiction on the completed solution curve. We now explain this step and
the local analysis that makes it possible.

\subsection*{Technical difficulties and contributions}
The proof has three main components.
\begin{enumerate}[label=\textup{(\arabic*)}]
\item \textbf{Representing every singular branch in local coordinates.}
The diagrams record leading orders, while the equations also couple distances
that vanish or diverge at different rates. We construct weighted blow-up
charts for Diagrams~1--4 and~6--8 and prove that every corresponding Puiseux
germ has a finite boundary limit in its chart. The rules and estimates of
\cite[\S3]{AK} give the required order inequalities. This converts the analysis
of singular sequences into equations and rank conditions at finite boundary
points. We also derive the mass relations used in the local arguments.

\item \textbf{Resolving degenerate boundary points.}
At boundary points where the Jacobian loses rank, the first nonzero
compatibility condition determines whether a solution branch can enter the
interior. We compute these conditions through local expansions and
Lyapunov--Schmidt reduction. The resulting arguments exclude Diagrams~1, 2, 4,
6 and~8 for complex germs. For Diagram~7, the inertia identity combines with
an unbounded four-product to give $I\to0$, whereas a curve through a real
configuration satisfies $U=I=c>0$. For Diagram~3, analytic continuation gives
\[
 \delta_{15}+2\delta_{13}=0
\]
on the irreducible curve, which excludes a positive point. Thus identities on
the whole curve connect the complex boundary analysis with the real
configurations counted in Theorem~\ref{thm:main}.

\item \textbf{Closing the argument with the product of distances.}
At a boundary point $P$ carrying a germ of Diagram~5 or~9, we prove
\[
 \ord_P\Delta>0.
\]
For Diagram~5 this follows from force-order estimates for the shrinking
triangle and the two outer bodies, together with the nondegeneracy transverse
to the triangle's critical manifold. For Diagram~9 it follows from the
distance estimates of \cite[\S8.9]{AK}. On a hypothetical irreducible curve
containing a positive arc, some four-product has a pole at a point $P$ of the
smooth projective completion. The diagram exclusions leave Diagrams~5 and~9,
so $\ord_P\Delta>0$. Since $\Delta$ is nonzero on the affine curve, it is then
nonconstant and has a pole at another point $Q$. The identity
\[
 \Delta^4=\prod_{i<j}p_{ij}
 \quad\Longrightarrow\quad
 4\ord_Q\Delta=\sum_{i<j}\ord_Qp_{ij}<0
\]
forces a four-product to have a pole at $Q$. Applying the same local results
gives $\ord_Q\Delta>0$, a contradiction. This argument uses the vanishing in
Diagrams~5 and~9 directly to complete the proof for the exceptional masses.
\end{enumerate}

The computational statements enter the analysis of Diagrams~2, 3 and~6 through
exact polynomial identities, Taylor coefficients, and interval and ball
arithmetic over the parameter sets. Appendix~\ref{app:verification} records
the statements and the associated program modules. The local arguments for
the other diagrams are given explicitly in the text. We use the
singular-sequence classification of \cite[Fig.~11]{AK}, its rules and estimates
in \cite[\S3]{AK}, and its four-product boundedness argument for
Diagrams~10--16. Section~\ref{sec:ak} states these inputs, and
Section~\ref{sec:exclusion} treats Diagram~9 through its distance estimates.

\subsection*{Organization}
Section~\ref{sec:ak} introduces the normalized equations and the
singular-sequence framework. Section~\ref{sec:reduction} passes from an infinite
set of positive configurations to an algebraic curve and a pole of a
four-product. Section~\ref{sec:exclusion} states the local boundary results,
proves the Diagram~9 estimate and deduces Theorem~\ref{thm:main} from the product
identity. Sections~\ref{sec:d1}--\ref{sec:d8} establish the diagram statements.
Appendix~\ref{app:verification} collects the computational statements used in
the proof.

\section{The framework of Albouy and Kaloshin}\label{sec:ak}

We express normalized central configurations as a complex algebraic set and record the
quantities preserved along its curves.  The singular-sequence classification then reduces
behaviour at infinity to finitely many diagrams, which we analyse in the following sections.

\subsection{Normalized central configurations}
Fix masses $m=(m_1,\dots,m_5)\in\Rp^5$.  Following \cite[\S2]{AK} we use the complex coordinates
$z_k=x_k+iy_k$, $w_k=x_k-iy_k$ and inverse distances $\delta_{kl}=\delta_{lk}$, and write
$z_{kl}=z_l-z_k$, $w_{kl}=w_l-w_k$.  The normalized system (system (4) of \cite{AK}, written as
system (10) there) is
\begin{equation}\label{eq:system}
\begin{gathered}
  z_k = \sum_{l\ne k} m_l\,\delta_{kl}^3\,(z_k-z_l),\qquad
  w_k = \sum_{l\ne k} m_l\,\delta_{kl}^3\,(w_k-w_l)\qquad (k=1,\dots,5),\\
  \delta_{kl}^2\, z_{kl}\,w_{kl} = 1\quad (1\le k<l\le 5),\qquad z_{12} = w_{12}.
\end{gathered}
\end{equation}
Its zero set $A_m\subset\C^{20}$ is the set of \emph{normalized central configurations}.  A
solution is \emph{real} if $w_k=\bar z_k$ for all $k$ and \emph{positive} if moreover
$\delta_{kl}>0$ for all $k<l$ \cite[Def.~2]{AK}.  Positive normalized central configurations are
in two-to-one correspondence with central configurations modulo rotation, translation and
scaling, so Theorem~\ref{thm:main} asserts that $A_m$ has finitely many positive points.
The map $\sigma(z,w,\delta)=(\bar w,\bar z,\bar\delta)$ preserves \eqref{eq:system} and fixes
{the real solutions.  System \eqref{eq:system} is symmetric under $z\leftrightarrow w$.  A permutation of the bodies,
with the masses permuted, preserves the equations apart from the choice of gauge.  After
relabelling, write the coordinates again as $(z,w,\delta)$.  A rotation
$(z,w)\mapsto(\lambda z,\lambda^{-1}w)$ with $\lambda^2=w_{12}/z_{12}$ restores the gauge
$z_{12}=w_{12}$.  It preserves $U$, $I$, every $r_{kl}^2$, and the diagram.  Along a germ,
if $\ord\lambda=a$, it adds $a$ to all $z$-orders and subtracts $a$ from all $w$-orders.  Thus
\[
 \ord(\lambda z_{ij})-\ord(\lambda z_{kl})
 =\ord z_{ij}-\ord z_{kl},\qquad
 (\lambda z_{ij})(\lambda^{-1}w_{ij})=r_{ij}^2.
\]
We implement these symmetries on finite algebraic covers, with a finite ramification of
the germ parameter when required.

More precisely, on a relabelled irreducible curve $\Gamma$ consider
$\widetilde\Gamma=\{(q,\lambda)\in\Gamma\times\C^*:
\lambda^2z_{12}(q)=w_{12}(q)\}$.  The distance constraints make $z_{12}$ and $w_{12}$ invertible regular
functions, so this is a finite unramified cover, and every irreducible component maps onto
$\Gamma$.  Any chosen germ lifts after a finite ramification.  Its component maps by the
rotation to an irreducible algebraic curve after taking the Zariski closure: some squared
distance is nonconstant by Lemma~\ref{lem:distances}, and remains nonconstant under the
cover and the rotation.  If $\Gamma$ contains a positive arc, each component contains a
local lift of that arc; there $|\lambda|=1$, so its rotated image is again a nonconstant
positive arc.  The same argument preserves the existence of a real point.  All squared
distances, four-products, $U$ and $I$ are unchanged, and a pole of a four-product remains
a pole on a lifted germ.}

\subsection{Potential and inertia}
The potential and inertia are
\[
 U:=\sum_{k<l}m_km_l\delta_{kl},\qquad I:=\sum_km_kz_kw_k.
\]
We show that both have the same constant value on each irreducible curve of solutions.
This converts the local limits computed below into restrictions on the entire curve.

\begin{lemma}\label{lem:U-constant}
Every solution of \eqref{eq:system} satisfies $U=I$, and $U$ is constant on every irreducible
curve $\Gamma\subset A_m$.
\end{lemma}

\begin{proof}
Let $B\subset\C^{20}$ be the set where $\delta_{kl}^2z_{kl}w_{kl}=1$ for all $k<l$, and let
$f=2U+I$.  On $B$ no $z_{kl}w_{kl}$ vanishes, so $B$ is smooth, $(z,w)$ are local coordinates
on $B$, and $d\delta_{kl}=-\frac12\delta_{kl}^3(w_{kl}\,dz_{kl}+z_{kl}\,dw_{kl})$.  Hence, on
$B$,
\[
  \frac{\partial f}{\partial z_k}=m_k\Bigl(w_k-\sum_{l\ne k}m_l\delta_{kl}^3(w_k-w_l)\Bigr),\qquad
  \frac{\partial f}{\partial w_k}=m_k\Bigl(z_k-\sum_{l\ne k}m_l\delta_{kl}^3(z_k-z_l)\Bigr),
\]
so every point of $A_m$ is a critical point of $f|_B$.  The scaling
$(z,w,\delta)\mapsto(sz,sw,s^{-1}\delta)$ preserves $B$, and hence
\[
 0=\left.\frac{d}{ds}\bigl(2s^{-1}U+s^2I\bigr)\right|_{s=1}
 =-2U+2I
\]
at every such point.  Therefore $U=I$ on $A_m$.  Let $\Gamma\subset A_m$ be an irreducible curve.  At a smooth point of
$\Gamma$ the tangent line lies in the tangent space of $B$, on which $df$ vanishes; so $f$ is
locally constant on the smooth locus of $\Gamma$, which is connected, and hence constant on
$\Gamma$.  As $f=3U$ on $A_m$, $U$ is constant on $\Gamma$.
\end{proof}
The identity $U=I$ also appears in \cite[p.~541]{AK}; in fact,
$U(A_m)$ is finite by \cite[Lemma~2]{AK}.  We will also use the following consequence of
constructibility to produce singular sequences.

\begin{fact}[{\cite[Lemma~1]{AK}}]\label{fact:lemma1}
If $X\subset\C^N$ is a closed algebraic subset and $f$ a polynomial, then $f(X)$ is finite or
the complement of a finite set (``$f$ is dominating'').
\end{fact}

\subsection{Singular sequences and diagrams}
We recall how the dominant coordinate and force terms determine the diagram of a singular
sequence.  This description will let us read the diagrams from the orders of Puiseux series.
Let $\mathcal Z=(z_1,\dots,z_5,Z_{12},\dots,Z_{45})$ and $\mathcal W=(w_1,\dots,w_5,W_{12},
\dots,W_{45})$ with $Z_{kl}=z_{kl}\,\delta_{kl}^3$ and $W_{kl}=w_{kl}\,\delta_{kl}^3$; these are,
up to mass factors, the terms of the equations \eqref{eq:system}.  Let $Z=\max_p|\mathcal Z_p|$
and $W=\max_q|\mathcal W_q|$.  A \emph{singular sequence} \cite[Def.~3]{AK} is a sequence of
normalized central configurations with $Z$ or $W$ unbounded, extracted so that the maximal
components are fixed and $Z^{-1}\mathcal Z$, $W^{-1}\mathcal W$ converge.  Its \emph{diagram}
has a $z$-circle at $k$ if $\lim z_k/Z\ne0$, a $z$-stroke between $k$ and $l$ if
$\lim Z_{kl}/Z\ne 0$, and likewise in the second colour \cite[\S3]{AK}.  After multiplying $z$
by $a>0$ and $w$ by $a^{-1}$ one may assume $Z=W=\epsilon^{-2}$ with $\epsilon\to0$
\cite[p.~547]{AK}.  We write $a\prec b$, $a\preccurlyeq b$, $a\approx b$, $a\sim b$ as in
\cite[\S3]{AK}.

\begin{fact}[{\cite[\S3]{AK}}]\label{fact:rules}
Rules 1a--1e and 2a--2e and Estimates 1--2 of \cite[\S3]{AK} hold for every singular sequence.
\end{fact}

\begin{fact}[{\cite[Prop.~1]{AK}}]\label{fact:prop1}
Proposition 1 of \cite{AK} (localization of the origin in a cluster) holds.
\end{fact}

\begin{fact}[{\cite[\S7]{AK}}]\label{fact:fig11}
Every singular sequence of the five-body problem approaches, after renumbering the bodies and
possibly exchanging $z$ and $w$, one of the sixteen diagrams of \cite[Fig.~11]{AK}.
\end{fact}

For pairwise distinct $i,j,k,l,m$ with $\{i,j,k,l,m\}=\{1,\dots,5\}$ the \emph{four-product} is
$p_{ij}=r_{ij}^2r_{kl}^2r_{lm}^2r_{mk}^2$, where $r_{kl}^2=z_{kl}w_{kl}$.

\begin{fact}[{\cite[Prop.~8]{AK}}]\label{fact:prop8}
Along a singular sequence approaching one of Diagrams 10--16 every four-product is bounded.
\end{fact}

Facts~\ref{fact:fig11} and~\ref{fact:prop8} leave Diagrams~1--9 whenever a four-product is
unbounded.  Lemma~\ref{lem:d9-product} gives the distance estimates needed for Diagram~9.
For Diagrams~1--6 and~8, Propositions~\ref{prop:d1-first}, \ref{prop:d2-first},
\ref{prop:d3-first}, \ref{prop:d45-first}, \ref{prop:d6-first} and \ref{prop:d8-first}
derive the mass conditions from the corresponding limiting equations.

\section{Reduction to germs with an unbounded four-product}\label{sec:reduction}

To apply the diagram classification, we seek an algebraic curve containing a positive arc
and a branch at which a four-product has a pole.  We first prove that fixed squared
distances determine only finitely many normalized solutions.  A monomial identity then
ensures that every curve has a nonconstant four-product.

\begin{lemma}\label{lem:distances}
Let $s\in(\C^*)^{10}$.  Only finitely many points of $A_m$ satisfy $z_{kl}w_{kl}=s_{kl}$ for all
$k<l$.
\end{lemma}

\begin{proof}
For any such point, $\delta_{kl}^2=s_{kl}^{-1}$, so $\delta$ takes finitely many values; fix one.
Multiplying the $z$-equation of body $k$ by $m_k$ and summing gives $\sum_km_kz_k=0$ by
antisymmetry, and likewise $\sum_km_kw_k=0$.

Put $v_k=(z_k-z_1,w_k-w_1)\in\C^2$ and let $B$ be
the symmetric bilinear form with $B((a,b),(a,b))=ab$.  Then $B(v_k-v_l,v_k-v_l)=s_{kl}$, so the
Gram matrix $B(v_k,v_l)=\frac12(s_{1k}+s_{1l}-s_{kl})$ is fixed.  The span $V$ of the $v_k$ is
nondegenerate for $B$: either $V=\C^2$, or $V=\C v_2$ with $B(v_2,v_2)=s_{12}\ne0$.

If
$(v'_k)$ is a second family with the same Gram matrix, its span $V'$ is nondegenerate by the
same argument, and $\sum_kc_kv_k=0$ holds if and only if $\sum_kc_kB(v_k,v_l)=0$ for all $l$,
and likewise for the $v'_k$; so the $v_k$ and the $v'_k$ satisfy the same linear relations, and
$v_k\mapsto v'_k$ is an isometry of nondegenerate subspaces; by Witt's theorem it extends to an
element of the orthogonal group of $B$, which consists of the maps $(a,b)\mapsto(\lambda a,
\lambda^{-1}b)$ and $(a,b)\mapsto(\lambda b,\lambda^{-1}a)$, $\lambda\in\C^*$.

The centre-of-mass identities determine $(z_1,w_1)$ from $(v_k)$:
\[
 (z_1,w_1)=-\frac{1}{\sum_km_k}\sum_km_kv_k.
\]  For a fixed reference family $(v^0_k)$ with first
components $a^0_2$ and second components $b^0_2$ of $v^0_2$ (so $a^0_2b^0_2=s_{12}\ne0$), the gauge
$z_{12}=w_{12}$ leaves $\lambda^2=b^0_2/a^0_2$ or $\lambda^2=a^0_2/b^0_2$: at most four values.
Hence there are finitely many solutions for each $\delta$.
\end{proof}

We now use semialgebraic curve selection to retain positivity while passing from an infinite
set of configurations to an algebraic curve.

\begin{lemma}[Arc lemma]\label{lem:arc}
If $A_m$ has infinitely many positive points, there is an irreducible algebraic curve
$\Gamma\subset A_m$ such that
\begin{romanlist}
  \item $\Gamma$ contains a nonconstant real-analytic arc of positive points;
  \item $\sigma(\Gamma)=\Gamma$;
  \item $U=I$ is constant on $\Gamma$, with value $c\in(0,\infty)$.
\end{romanlist}
\end{lemma}

\begin{proof}
In the real coordinates $(x,y,\delta)\in\R^{20}$ the positive points form a semialgebraic set.
It has finitely many connected components, so one of them is infinite, hence of positive
dimension, and the Nash curve selection lemma \cite[Prop.~8.1.13]{BCR} gives a nonconstant Nash arc in
it.  Its Zariski closure $\Gamma\subset A_m$ (in the complexified coordinates; passing to $(z,w)$ is linear) is an
irreducible curve, which proves (i).  Since \eqref{eq:system} has real coefficients, $\sigma(
\Gamma)$ is an irreducible curve containing the arc, so $\sigma(\Gamma)=\Gamma$.  By
Lemma~\ref{lem:U-constant}, $U=I$ is constant on $\Gamma$, and on the arc
$U=\sum m_km_l\delta_{kl}>0$.
\end{proof}

On the projective normalization of this curve, a singular sequence approaches a branch
at infinity.  The next lemma expresses that branch in a parameter suitable for comparing
the orders of the force terms.

\begin{lemma}\label{lem:germ}
Let $\Gamma\subset A_m$ be an irreducible curve and $(q_n)$ a singular sequence on $\Gamma$.  There
are a Puiseux parametrization $\gamma$ of a branch of $\Gamma$ at infinity, $\gamma(t)\to\infty$ as
$t\to 0$, and $t_n\to0$ with $\gamma(t_n)$ a subsequence of $(q_n)$.  Every coordinate and every
component of $\mathcal Z$, $\mathcal W$ has a Puiseux expansion $c\,t^{\alpha}(1+o(1))$ along
$\gamma$ (or vanishes identically), and the diagram of every singular sequence extracted from
$(\gamma(t_n))$ is determined by these orders and leading coefficients: the circles and strokes
of colour $z$ are the components of $\mathcal Z$ of minimal order, and the relations $\prec$,
$\preccurlyeq$, $\approx$, $\sim$ of \cite[\S3]{AK} become relations between orders and leading
coefficients.  In particular a four-product unbounded along $(q_n)$ has negative order along
$\gamma$.
\end{lemma}

\begin{proof}
The closure of $\Gamma$ in a projective compactification has finitely many points at infinity,
each with finitely many analytic branches; a sequence tending to infinity has a subsequence on
one branch, and the normalization of the branch gives the parameter $t$.  Polynomial functions
of the coordinates are meromorphic in $t$, which gives the expansions.

If $\alpha$ is the least
order among the components of $\mathcal Z$, then $Z=\max_p|\mathcal Z_p|$ has order $\alpha$, and
$\mathcal Z_p/Z\to0$ exactly when $\ord\mathcal Z_p>\alpha$; this does not depend on the
subsequence.

The renormalization of \cite[p.~547]{AK} multiplies all $z$-quantities by one
positive factor $a(t)\approx|t|^{(\beta-\alpha)/2}$ and all $w$-quantities by $a^{-1}$, where
$\beta$ is the least order in $\mathcal W$.  This preserves every relation within one colour and
every product $z_{kl}w_{kl}$.  Replacing $a$ by the Puiseux series $\lambda(t)=(\mathcal W_q(t)/
\mathcal Z_p(t))^{1/2}$ built from maximal components, with $|\lambda|/a$ bounded above and
below, gives an analytic renormalization with the same relations; after reparametrizing
$t=s^N$, $\epsilon$ is comparable to a power of $s$.  A meromorphic function that is unbounded
along $t_n\to0$ has a pole.
\end{proof}

\begin{definition}\label{def:germ}
A \emph{Diagram-$r$ germ} is a Puiseux branch $\gamma$ of a curve in $A_m$ whose diagram, in the
sense of Lemma~\ref{lem:germ}, is Diagram $r$ of \cite[Fig.~11]{AK} after renumbering the bodies
and possibly exchanging $z$ and $w$.  It has an \emph{unbounded four-product} if some $p_{ij}$ has
negative order along $\gamma$.
\end{definition}

The following monomial relations transfer variation of the squared distances to variation
of the four-products.

\begin{lemma}\label{lem:petersen}
On every irreducible curve $\Gamma\subset A_m$ some four-product is not constant.
\end{lemma}
\begin{proof}
Index the squared distances and the four-products by the ten pairs.  The pairs disjoint from
$ij$ are the sides of the triangle complementary to $ij$, so the exponent vector of $p_{ij}$ in
the variables $r_{ab}^2$ is the column $ij$ of $E=I+P$, where $P$ is the adjacency matrix of the
Petersen graph (two pairs are adjacent if they are disjoint).  The eigenvalues of $E$ are $4$,
$2$ (five times) and $-1$ (four times), so $\det E=128$, and $N=128E^{-1}$ is an integer
matrix.  Since $r_{ab}^2\ne0$ on $A_m$, we obtain
\[
 (r_{ab}^2)^{128}=\prod_{ij}p_{ij}^{N_{ab,ij}}.
\]
If every $p_{ij}$ were constant on $\Gamma$, this identity would make every $r_{ab}^2$
constant on the irreducible curve $\Gamma$.  Lemma~\ref{lem:distances} would then make
$\Gamma$ finite, a contradiction.
\end{proof}
For the global argument, we also need a relation that couples all four-products to a
single meromorphic function.

{
\begin{lemma}[The total product]\label{lem:total-product}
Put $\Delta=\prod_{k<l}r_{kl}^2$.  On $A_m$, the function $\Delta$ is regular and nowhere zero,
and
\begin{equation}\label{eq:total-product}
  \prod_{i<j}p_{ij}=\Delta^4.
\end{equation}
\end{lemma}
\begin{proof}
The distance constraints give $r_{kl}^2=z_{kl}w_{kl}\ne0$ on $A_m$.  Each squared distance occurs in its own
four-product and in the three four-products indexed by a pair disjoint from its endpoints.
It therefore occurs exactly four times in the product on the left of \eqref{eq:total-product}.
\end{proof}
}

\begin{proposition}[Reduction]\label{prop:reduction}
Let $m\in\Rp^5$.  If $A_m$ has infinitely many positive points, there are $r\in\{1,\dots,9\}$, a
curve $\Gamma$ as in Lemma~\ref{lem:arc}, and a Diagram-$r$ germ on $\Gamma$ with an unbounded
four-product.
\end{proposition}

\begin{proof}
Let $\Gamma$ be given by Lemma~\ref{lem:arc}.  By Lemma~\ref{lem:petersen} some four-product
$p_{ij}$ is not constant on $\Gamma$, and by Fact~\ref{fact:lemma1} it is dominating there.  A
sequence on $\Gamma$ with $p_{ij}\to\infty$ is unbounded, and a singular sequence extracted from
it approaches one of Diagrams 1--9 by Facts~\ref{fact:fig11} and~\ref{fact:prop8}.
Lemma~\ref{lem:germ} turns it into a germ.
\end{proof}

\begin{remark}\label{rem:levels}
On a positive arc with $U=I=c>0$, write $M=\sum_km_k$.  The identities
\[
 c=\sum_{k<l}\frac{m_km_l}{r_{kl}}
   =\frac1M\sum_{k<l}m_km_lr_{kl}^2
\]
give, for every pair,
\[
 \frac{m_km_l}{c}\le r_{kl}\le\sqrt{\frac{Mc}{m_km_l}}.
\]
The branches at infinity are therefore reached by continuation on the complex curve.
We use identities on that curve to connect their asymptotic behaviour with the positive arc:
$U=I=c$ for Diagram~7 and an inverse-distance identity for Diagram~3.  For
Diagrams~5 and~9, the local conclusion $\ord_\gamma\Delta>0$ supplies the input to the
global argument below.
\end{remark}

{
\section{Boundary alternatives and the global argument}\label{sec:exclusion}

We combine the diagram analysis with the meromorphic identity
$\Delta^4=\prod_{i<j}p_{ij}$.  On a curve containing a positive arc, the local alternatives
force $\Delta$ to vanish at every pole of a four-product.  A zero and a pole of $\Delta$
on the completed curve then yield the contradiction proving finiteness.

\begin{theorem}\label{thm:exclusion}
Let $m\in\Rp^5$, $r\in\{1,\dots,9\}$ and let $\gamma$ be a Diagram-$r$ germ with an unbounded
four-product on an irreducible curve $\Gamma\subset A_m$.
\begin{alphalist}
  \item If $r\in\{1,2,4,6,8\}$, no such germ exists.
  \item If $r=7$, then $\Gamma$ contains no real point.
  \item If $r=3$, then $\Gamma$ contains no arc of positive points.
  \item If $r\in\{5,9\}$, then $\ord_\gamma\Delta>0$, where $\Delta=\prod_{k<l}r_{kl}^2$.
\end{alphalist}
\end{theorem}

\begin{lemma}[Diagram 9]\label{lem:d9-product}
Let $m\in\Rp^5$ and let $\gamma$ be a Diagram~9 germ on an irreducible curve
$\Gamma\subset A_m$.  With the numbering of \cite[\S\S7.5.1.2 and~8.9]{AK}, the
$zw$-triangle is $123$ and the two $w$-edges are $14$ and $25$.  Then
\begin{equation}\label{eq:d9-distances}
\begin{gathered}
 r_{12},r_{13},r_{23}\approx\epsilon,\qquad
 r_{14},r_{25},r_{45}\approx1,\\
 r_{15},r_{24},r_{34},r_{35}\approx\epsilon^{-1/2}.
\end{gathered}
\end{equation}
Consequently $\Delta\approx\epsilon^2\to0$ and $\ord_\gamma\Delta>0$.
\end{lemma}

\begin{proof}
We normalize the germ as in Lemma~\ref{lem:germ} and take $r_{ij}=\delta_{ij}^{-1}$,
so that the signs of the square roots are those of the given solution.
We first exclude the clustering in \cite[\S8.9.1]{AK}.  In that case,
$r_{14}\approx\epsilon^{1/4}$, and equation~(41) of \cite{AK}, together with the
estimates following it, gives
\[
 \frac{m_1m_2}{r_{12}}+\frac{m_1m_3}{r_{13}}+
 \frac{m_2m_3}{r_{23}}=O(\epsilon^2).
\]
Every term of $U$ outside this triangle, apart from $m_1m_4/r_{14}$, is bounded.  Hence
\[
 U=\frac{m_1m_4}{r_{14}}+O(1),\qquad
 \left|\frac{m_1m_4}{r_{14}}\right|\longrightarrow\infty,
\]
contrary to Lemma~\ref{lem:U-constant}.  The same argument applies after the
simultaneous exchanges $1\leftrightarrow2$ and $4\leftrightarrow5$.

We are therefore in the case of \cite[\S8.9.2]{AK}, where both $w$-edges are
maximal.  The estimates for the $zw$-triangle and the maximal $w$-edges give
\[
 z_{12},z_{13},z_{23},w_{12},w_{13},w_{23}\approx\epsilon,
 \qquad z_{14},z_{25}\approx\epsilon^2,
 \qquad w_{14},w_{25}\approx\epsilon^{-2}.
\]
Taking differences gives $z_{45}\approx\epsilon$ and, for
$ij\in\{15,24,34,35\}$,
\[
 z_{ij}\approx\epsilon,\qquad w_{ij}\approx\epsilon^{-2},
 \qquad r_{ij}\approx\epsilon^{-1/2}.
\]
Thus all estimates in \eqref{eq:d9-distances}, except that for $r_{45}$, follow.
To obtain the remaining estimate, subtract the $z$-equations of bodies $4$ and $5$.
The terms from pairs $14$ and $25$ are $O(\epsilon^2)$, and those from pairs
$15,24,34,35$ are $O(\epsilon^{5/2})$.  Therefore
\[
 (z_4-z_5)\bigl(1-(m_4+m_5)r_{45}^{-3}\bigr)=O(\epsilon^2).
\]
Since $z_4-z_5\approx\epsilon$ and $m_4+m_5>0$, we obtain
$r_{45}^3\to m_4+m_5$, and hence $r_{45}\approx1$.

Multiplying the ten squared-distance estimates now gives
\[
 \Delta\approx(\epsilon^2)^3\,1^3\,(\epsilon^{-1})^4
 =\epsilon^2\longrightarrow0.
\]
The meromorphic function $\Delta$ thus has positive order along $\gamma$.
\end{proof}

In the numbering of Lemma~\ref{lem:d9-product}, the same estimates give
\[
 p_{15}=r_{15}^2r_{23}^2r_{34}^2r_{24}^2\approx\epsilon^{-1},\qquad
 p_{24}=r_{24}^2r_{13}^2r_{35}^2r_{15}^2\approx\epsilon^{-1}.
\]
Thus Diagram~9 contributes unbounded four-products together with a vanishing total product.
This is the alternative used in the global argument, while Fact~\ref{fact:prop8}
provides boundedness for Diagrams~10--16.

\begin{proof}[Proof of Theorem~\ref{thm:main}]
Suppose that $A_m$ has infinitely many positive points.  Proposition~\ref{prop:reduction}
gives an irreducible curve $\Gamma$ containing a positive arc and a germ $\gamma_0$ on which
some four-product has a pole.  Theorem~\ref{thm:exclusion}(a)--(c) excludes every diagram
except Diagrams~5 and~9.  By part~(d),
\[
 \ord_{\gamma_0}\Delta>0.
\]

Let $\overline\Gamma$ be the smooth projective completion of the normalization of $\Gamma$.
All squared distances, four-products and $\Delta$ pull back to meromorphic functions on
$\overline\Gamma$.  By Lemma~\ref{lem:total-product}, $\Delta$ is nowhere zero on the affine
curve; its zero at $\gamma_0$ therefore implies that it is nonconstant.  A nonconstant
meromorphic function on the compact curve $\overline\Gamma$ has a pole, so choose a pole
of $\Delta$ and let $\gamma_1$ be the corresponding germ.  Taking orders in
\eqref{eq:total-product} gives
\[
 \sum_{i<j}\ord_{\gamma_1}p_{ij}=4\ord_{\gamma_1}\Delta<0,
\]
so $\ord_{\gamma_1}p_{ij}<0$ for at least one pair $ij$.
As in the proof of Proposition~\ref{prop:reduction}, Facts~\ref{fact:fig11} and
\ref{fact:prop8} and Lemma~\ref{lem:germ} put this germ in one of Diagrams~1--9.
It lies on the same curve $\Gamma$ through the positive arc, so
Theorem~\ref{thm:exclusion}(a)--(c) again leaves only Diagrams~5 and~9.  Part~(d) now gives
$\ord_{\gamma_1}\Delta>0$, contradicting the choice of a pole.
The finite covers and rotations used to restore the gauge after relabelling preserve
$\Delta$, the four-products and the existence of the positive arc, as explained in
Section~\ref{sec:ak}.
\end{proof}

The statements of Theorem~\ref{thm:exclusion} for Diagrams~1--8 are proved in
Sections~\ref{sec:d1}--\ref{sec:d8}, and the Diagram~9 statement is Lemma~\ref{lem:d9-product}.
Theorem~\ref{thm:d5} and Lemma~\ref{lem:d9-product} establish
$\ord_\gamma\Delta>0$ for every germ of Diagram~5 or~9.
Table~\ref{tab:exclusions} records the mass conditions, assumptions and conclusions for each
diagram.

\begin{table}[ht]

\centering
\small
\begin{tabular}{cllp{5.4cm}}
\hline
$r$ & mass condition & configurations & method and conclusion \\
\hline
1 & (23) on the triangle & complex & triangle virial and Morse--Bott estimate are incompatible \\
2 & $K_j$, $L_j$, $P_j$ & complex & isolation of every boundary label at orders 1--3,
with folds of corank one and two \\
3 & (29) $m_1m_3=m_2m_4$ & positive & orders 3--6; identity $\delta_{15}+2\delta_{13}=0$ \\
4 & (23) & complex & incompatibility of the force balances at orders one and three \\
5 & (23) & complex & direct force-order estimates give $\ord\Delta>0$; used in the global pole argument \\
6 & (32) $m_1=m_2$, $m_3=m_4$ & complex & finite family of coordinate charts \\
7 & none needed & real & inertia identity forces $I\to0$ \\
8 & (40) & complex & order-six and order-nine pairings; elimination on one curve of masses \\
9 & none needed & complex & the distance estimates give $\Delta\approx\epsilon^2\to0$ \\
\hline
\end{tabular}
\caption{The nine boundary arguments.  For Diagrams~5 and~9, the product of squared distances vanishes;
this conclusion is used in the global argument.  Numbers in parentheses refer to the equations of \cite[\S8]{AK}.}
\label{tab:exclusions}
\end{table}

}
Throughout, a statement about a diagram is meant for every renumbering of the bodies (with the
masses permuted) and for both colourings; both operations are symmetries of
\eqref{eq:system}.

\section{Diagram 1}\label{sec:d1}

We exclude Diagram~1 by comparing two orders of vanishing of the triangle potential.
The limiting triangle has a nondegenerate critical manifold, which gives a quadratic
bound in the forces; the virial identity gives an incompatible leading term.

{Diagram~1 is one of the two disconnected diagrams remaining in \cite[\S7.1, p.~565]{AK}.
It consists of the $z$-stroke $12$, with $z$-circles at $1$ and $2$, and the $w$-triangle
with strokes $34,35,45$ and no circle.  There are no other strokes or circles.
We use the numbering of \cite[\S8.1]{AK}: $P=\{1,2\}$ is the pair and $T=\{3,4,5\}$ is the
triangle.  Indices $i$ refer to bodies of $P$, and $k,l$ to bodies of $T$.  Put
$M_P=m_1+m_2$, $M_T=m_3+m_4+m_5$, $Z_P=M_P^{-1}(m_1z_1+m_2z_2)$,
$Z_T=M_T^{-1}\sum_km_kz_k$, and define $W_P,W_T$ in the same way.  The mass-centre identities
$\sum_jm_jz_j=\sum_jm_jw_j=0$ on $A_m$ were proved in Lemma~\ref{lem:distances}; they give
$M_PZ_P+M_TZ_T=0=M_PW_P+M_TW_T$.  Write the coordinate equations of \eqref{eq:system} as
$E^z_j=E^w_j=0$, where}
\[
  E^z_j=z_j-\sum_{l\ne j}m_l\delta_{jl}^3(z_j-z_l),\qquad
  E^w_j=w_j-\sum_{l\ne j}m_l\delta_{jl}^3(w_j-w_l).
\]
Along a germ, $a\prec b$, $a\preccurlyeq b$ and $a\approx b$ in the sense of \cite[\S3]{AK} mean
$\ord a>\ord b$, $\ord a\ge\ord b$ and $\ord a=\ord b$ (Lemma~\ref{lem:germ}).

{The first step is to control a triangle body that might approach the $w$-centre of the triangle
more closely than the other two.  The torque on the pair rules this out
(Lemma~\ref{lem:d1-central}), allowing every germ to be represented in one weighted chart
(Proposition~\ref{prop:d1-cover}).  The limiting triangle then has zero multiplier and its
masses satisfy relation~(23) of \cite[\S8.1]{AK} (Proposition~\ref{prop:d1-first}).
Near this limiting triangle, the potential vanishes quadratically in the forces.
The virial identity gives a lower vanishing order along the germ, and the two estimates
yield the required contradiction.}

\subsection{The chart}
{Let $\gamma$ be a Diagram~1 germ at $m\in\Rp^5$.  The symmetries of
\eqref{eq:system} allow us to renumber the bodies and exchange $z$ and $w$ so that its
diagram is as described above.  Since the coordinate equations are homogeneous of degree
one in $z$ and in $w$ separately, and the distance constraints involve only $z_{jl}w_{jl}$, the transformation
$(z,w)\mapsto(cz,c^{-1}w)$ preserves both the equations and the diagram
\cite[p.~547]{AK}.
By Rule~2e \cite[p.~552]{AK}, the three $z$-separations of the triangle have the same order,
as do its three $w$-separations.  Choosing $c$ to be a power of the parameter $t$, we normalize
the $w$-separations of the triangle to have order $0$.  This choice of rotation fixes the chart normalization.  Define
$\lambda,\mu,\rho,\sigma$ by}
\[
  \ord z_{12}=-\lambda,\qquad \ord z_{kl}=\mu,\qquad \ord Z_P=\mu+\rho,\qquad \ord W_P=\sigma,
\]
with $\rho=+\infty$ (resp.\ $\sigma=+\infty$) if $Z_P$ (resp.\ $W_P$) vanishes identically, in
which case the terms carrying $\hat Z$ (resp.\ $\hat W$) below are absent.
After a finite ramification of $t$ we write
\begin{equation}\label{eq:d1-chart}
\begin{gathered}
  z_i=t^{\mu+\rho}\hat Z+t^{-\lambda}A_i,\qquad w_i=t^{\sigma}\hat W+t^{\lambda}B_i
  \qquad(i\in P),\\
  z_k=-\tfrac{M_P}{M_T}\,t^{\mu+\rho}\hat Z+t^{\mu}\zeta_k,\qquad
  w_k=-\tfrac{M_P}{M_T}\,t^{\sigma}\hat W+\eta_k\qquad(k\in T),\\
  \delta_{12}=D_{12},\qquad\delta_{ik}=t^{\lambda/2}D_{ik},\qquad\delta_{kl}=t^{-\mu/2}D_{kl},
\end{gathered}
\end{equation}
where $(A_1,A_2)=(m_2,-m_1)\,\pi$, $(B_1,B_2)=(m_2,-m_1)\,\varpi$ and
$\sum_km_k\zeta_k=\sum_km_k\eta_k=0$.  {The chart coordinates are
$\hat Z,\hat W,\pi,\varpi$, the mass-centred triangle offsets $\zeta=(\zeta_k)$,
$\eta=(\eta_k)$, and the ten $D_{jl}$.  To obtain finite limiting equations, we multiply
$E^z_i$, $E^w_i$, $E^z_k$ and $E^w_k$ by $t^{\lambda}$, $t^{-\lambda}$, $t^{\mu/2}$ and
$t^{3\mu/2}$, respectively; the ten distance constraints retain weight $1$.
The exponents will belong to the open cone}
\begin{equation}\label{eq:d1-cone}
  C=\{\lambda>0,\ \mu>0,\ \sigma-\lambda>0,\ \rho+\tfrac32\mu>0,\ \rho+\lambda+\mu>0\}.
\end{equation}
For the triangle block put, with sums over $T$,
\[
  \phi_k=\sum_{l\ne k}m_lD_{kl}^3(\zeta_k-\zeta_l),\qquad
  \psi_k=\sum_{l\ne k}m_lD_{kl}^3(\eta_k-\eta_l),\qquad
  \chi_{kl}=D_{kl}^2(\zeta_l-\zeta_k)(\eta_l-\eta_k)-1,
\]
$U_T=\sum_{k<l}m_km_lD_{kl}$ and $I_T=\sum_km_k\zeta_k\eta_k$; note
$\sum_km_k\phi_k=\sum_km_k\psi_k=0$ identically.  The crossed forces are
$X^z_k=\sum_{i\in P}m_i\delta_{ik}^3(z_k-z_i)$ and $X^w_k=\sum_{i\in P}m_i\delta_{ik}^3(w_k-w_i)$.

{\begin{lemma}\label{lem:d1-chart}
After substitution of \eqref{eq:d1-chart} and multiplication by the prescribed weights,
each of the ten coordinate equations and ten distance constraints is a finite sum $\sum_ec_e\,t^e$.
The exponents $e$ are linear forms in $(\lambda,\mu,\rho,\sigma)$, $e=0$ occurs, and each
$c_e$ is a polynomial in the chart coordinates.  Every exponent $e\ne0$ is a nonnegative
rational combination of the five forms in \eqref{eq:d1-cone}, with at least one positive
coefficient.  The constant coefficients $c_0$ are independent of $\hat Z$ and $\hat W$.
For $k\in T$,
\[
  t^{\mu/2}E^z_k=t^{\mu/2}(z_k-X^z_k)-\phi_k,\qquad
  t^{3\mu/2}E^w_k=t^{3\mu/2}(w_k-X^w_k)-\psi_k,
\]
and the exponents occurring in $t^{\mu/2}(z_k-X^z_k)$, respectively $t^{3\mu/2}(w_k-X^w_k)$, are
\[
\begin{gathered}
  \tfrac32\mu,\ \tfrac12(\lambda+\mu),\ \tfrac32(\lambda+\mu),\ \tfrac32\mu+\rho,\
  \tfrac32(\lambda+\mu)+\rho,\\
  \text{respectively}\\
  \tfrac32\mu,\ \tfrac32(\lambda+\mu),\ \tfrac52\lambda+\tfrac32\mu,\ \sigma+\tfrac32\mu,\
  \tfrac32(\lambda+\mu)+\sigma .
\end{gathered}
\]
\end{lemma}}
{\begin{proof}
Put $h=M_P/M_T$ and, for $i\in P$, $k\in T$, set
\[
 a_{ik}=A_i-t^{\lambda+\mu}\zeta_k
             +(1+h)t^{\lambda+\mu+\rho}\hat Z,\qquad
 b_{ik}=\eta_k-t^\lambda B_i-(1+h)t^\sigma\hat W.
\]
Then $z_i-z_k=t^{-\lambda}a_{ik}$ and $w_i-w_k=-b_{ik}$.
The weighted coordinate equations are
\begin{align*}
 t^\lambda E^z_i
 &=A_i(1-M_PD_{12}^3)+t^{\lambda+\mu+\rho}\hat Z
       -t^{3\lambda/2}\sum_km_kD_{ik}^3a_{ik},\\
 t^{-\lambda}E^w_i
 &=B_i(1-M_PD_{12}^3)+t^{\sigma-\lambda}\hat W
       +t^{\lambda/2}\sum_km_kD_{ik}^3b_{ik},\\
 t^{\mu/2}E^z_k
 &=-\phi_k+t^{3\mu/2}\zeta_k-h t^{3\mu/2+\rho}\hat Z
       +t^{(\lambda+\mu)/2}\sum_im_iD_{ik}^3a_{ik},\\
 t^{3\mu/2}E^w_k
 &=-\psi_k+t^{3\mu/2}\eta_k-h t^{\sigma+3\mu/2}\hat W
       -t^{3(\lambda+\mu)/2}\sum_im_iD_{ik}^3b_{ik}.
\end{align*}
The distance constraints are, respectively,
$M_P^2D_{12}^2\pi\varpi-1$, $-D_{ik}^2a_{ik}b_{ik}-1$ and $\chi_{kl}$.
Every nonconstant power is a nonnegative rational combination of the five forms
defining $C$, with at least one positive coefficient.  Indeed,
$\sigma=\lambda+(\sigma-\lambda)$, and the terms involving $\rho$ are generated by
$\lambda+\mu+\rho$ and $3\mu/2+\rho$.  Expanding $a_{ik}$ and $b_{ik}$ gives the two
lists of exponents in the statement.  The constant terms give the limiting equations
described below.
\end{proof}

}{The constant terms $c_0$ are $A_i(1-M_PD_{12}^3)$ and
$B_i(1-M_PD_{12}^3)$ for the pair equations, $-\phi_k$ and $-\psi_k$ for the triangle
equations, and $M_P^2D_{12}^2\pi\varpi-1$, $-D_{ik}^2A_i\eta_k-1$ and $\chi_{kl}$ for the
distance constraints within the pair, between the pair and triangle, and within the triangle.
We call their common zero set, in the coordinates other than $\hat Z,\hat W$, the
\emph{boundary fibre} $\mathcal B_1$.  On $\mathcal B_1$,
$D_{12}^3=M_P^{-1}$ and $D_{ik}^2A_i\eta_k=-1$, while
$(\zeta,\eta,D_{34},D_{35},D_{45})$ satisfies the \emph{triangle system with zero multiplier}
$\phi=\psi=0$, $\chi=0$.}

\subsection{Every germ lies in the chart}
{We now prove that every Diagram~1 germ has a finite limit in the chart
\eqref{eq:d1-chart}.  We use Rules 1e and 2e and Estimates 1 and 2 of \cite[\S3]{AK}
(Fact~\ref{fact:rules}), together with \cite[Prop.~1]{AK} (Fact~\ref{fact:prop1}).
The following identities express the forces on each component and the torque on the pair.}

\begin{lemma}\label{lem:d1-syzygies}
{The following polynomial identities hold:}
\begin{gather*}
  m_1E^w_1+m_2E^w_2=m_1w_1+m_2w_2-\sum_{i\in P,\,k\in T}m_im_k\delta_{ik}^3(w_i-w_k),\\
  \sum_{k\in T}m_kE^z_k=\sum_{k\in T}m_kz_k-\sum_{i\in P,\,k\in T}m_im_k\delta_{ik}^3(z_k-z_i),\\
  T_P:=\sum_{i\in P,\,k\in T}m_im_k\delta_{ik}^3(z_iw_k-w_iz_k)
  =-\sum_{i\in P}m_i\big(w_iE^z_i-z_iE^w_i\big).
\end{gather*}
In particular the torque $T_P$ of the triangle on the pair vanishes on $A_m$.
\end{lemma}

\begin{proof}
{The internal force terms cancel in pairs within $P$ and within $T$.
The same cancellation in the weighted coordinate equations gives the torque identity.}
\end{proof}

\begin{lemma}\label{lem:d1-orders}
Let $\gamma$ be a Diagram~1 germ, normalized as above.
\begin{romanlist}
  \item $\ord\delta_{12}=0$ and $\ord w_{12}=\lambda>0$; the least order of the components of
  $\mathcal Z$ is $\ord z_1=\ord z_2=\ord Z_{12}=-\lambda$.
  \item $\mu>0$, $\ord\delta_{kl}=\ord Z_{kl}=-\mu/2$, and $\lambda>\mu/2$.
  \item $\ord z_{ik}=-\lambda$, $\ord w_{ik}<\lambda$ and $\ord w_i\ge\lambda$; hence
  $\ord W_T=\ord W_P\ge\lambda$.
  \item For $k\in T$ let $o_k$ be the coefficient of $t^0$ in $w_k-W_T$.  At most one $o_k$
  vanishes.  If $o_k\ne0$, then $\ord w_{ik}=0$ and $\ord\delta_{ik}=\lambda/2$.  If $o_k=0$, then
  $f=\ord(w_k-W_T)$ satisfies $0<f<\lambda$, and $\ord w_{ik}=f$, $\ord\delta_{ik}=(\lambda-f)/2$.
\end{romanlist}
\end{lemma}

\begin{proof}
(i) The isolated component $\{1,2\}$ of the $z$-diagram has circled vertices, so it has a
maximal $z$-stroke by Rule~1e \cite[p.~546]{AK}, which can only be $12$.  Estimate~2
\cite[p.~548]{AK} gives $r_{12}\approx1$, so $\ord\delta_{12}=0$ and $\ord w_{12}=-\ord
z_{12}=\lambda$.  Estimate~1 \cite[p.~548]{AK} with $z$ and $w$ exchanged gives
$w_{12}\approx\epsilon^2$, and Estimate~2 with $z$ and $w$ exchanged gives $w_{34}\succcurlyeq\epsilon$
for the $w$-stroke $34$.  Hence $w_{12}\prec w_{34}$; a comparison within one colour is not
affected by the normalizations, so $\lambda>\ord w_{34}=0$.  Finally $1,2$ are $z$-circled and
$12$ is a $z$-stroke, and $\ord Z_{12}=\ord z_{12}+3\ord\delta_{12}=-\lambda$.

(ii) Bodies $1,2$ carry no $w$-stroke, so $T$ is an isolated component of the $w$-diagram
without $w$-circles.  By Rule~1e it has no maximal $w$-stroke.  Estimate~2 with $z$ and $w$
exchanged gives $r_{kl}\preccurlyeq1$, with $r_{kl}\approx1$ only for a maximal $w$-edge, so
$r_{kl}\prec1$; thus $\mu=\ord(z_{kl}w_{kl})=2\ord r_{kl}>0$, $\ord\delta_{kl}=
-\mu/2$ and $\ord Z_{kl}=\mu-\frac32\mu=-\frac12\mu$.  Diagram~1 has no $z$-stroke inside $T$, so
$\ord Z_{kl}>-\lambda$, that is, $\lambda>\mu/2$.

(iii) The bodies of $T$ are not $z$-circled, so $\ord z_k>-\lambda=\ord z_i$ and $\ord
z_{ik}=-\lambda$.  Since $Z_{ik}^2=(z_{ik}w_{ik}^3)^{-1}$, $\ord Z_{ik}=\frac12(\lambda-3\ord
w_{ik})$, and the absence of a $z$-stroke $ik$ gives $\ord Z_{ik}>-\lambda$, i.e.\ $\ord
w_{ik}<\lambda$.  By Estimate~1 ($z$ and $w$ exchanged) $w_{2k}\succ\epsilon^2\approx w_{12}$,
because $12$ is the only maximal $z$-edge, and $z_{2k}\approx z_{12}$.  So
\cite[Prop.~1, p.~559]{AK} with $z$ and $w$ exchanged, applied to bodies $1,2$ as in the proof of
\cite[Prop.~4, p.~564]{AK}, gives $w_1,w_2\preccurlyeq w_{12}$, i.e.\ $\ord w_i\ge\lambda$.  This
also follows from Lemma~\ref{lem:d1-syzygies}: every term of $M_PW_P=\sum m_im_k\delta_{ik}^3
(w_i-w_k)$ has order $\frac12(3\lambda-\ord w_{ik})>\lambda$, and $w_i-W_P$ has order
$\ord w_{12}=\lambda$.

(iv) Since $\sum_km_ko_k=0$, and the $o_k$ do not all vanish because $\ord w_{kl}=0$, at most one
$o_k$ vanishes.  Now $w_k-w_i=(w_k-W_T)+(W_T-w_i)$ with $\ord(W_T-w_i)\ge\lambda>0$ by (iii).  If
$o_k\ne0$ this gives $\ord w_{ik}=0$.  If $o_k=0$ then $f>0$; $f\ge\lambda$ would give $\ord
w_{ik}\ge\lambda$, against (iii), so $f<\lambda$ and $\ord w_{ik}=f$.  In both cases
$\ord\delta_{ik}=-\frac12(\ord z_{ik}+\ord w_{ik})$.
\end{proof}

{\begin{lemma}\label{lem:d1-torque}
Let $\lambda,\mu,f>0$ with $f<\lambda$ and let $\rho,\sigma$ be real.  Modify
\eqref{eq:d1-chart} by $(\eta_3,\eta_4,\eta_5)=(t^f\omega,\ m_5\vartheta,\
-m_3m_5^{-1}t^f\omega-m_4\vartheta)$ and $\delta_{i3}=t^{(\lambda-f)/2}d_i$ $(i=1,2)$, all other
entries unchanged.  Then
\[
  T_P=t^{(\lambda-f)/2}\Big(m_3\sum_{i=1,2}m_id_i\,\big(d_i^2A_i\omega\big)+\sum_{g\in G}c_g\,t^g\Big),
\]
where the $c_g$ are polynomials in the chart coordinates and $G$ consists of the thirteen
exponents
\[
\begin{gathered}
  \tfrac12f,\ \tfrac32f,\ \sigma-f,\ \sigma+\tfrac12f,\ 2\lambda+\mu-f,\
  2\lambda+\mu+\tfrac12f,\ \lambda+\mu+\rho,\ \lambda+\mu+\sigma-f,\ 2\lambda+\mu+\rho-f,\\
  \lambda+\mu+\rho+\tfrac12f,\ \lambda+\mu+\sigma+\tfrac12f,\ 2\lambda+\mu+\rho+\tfrac12f,\
  \lambda+\mu+\rho+\tfrac32f,
\end{gathered}
\]
each a combination of $\lambda,\ \mu,\ \sigma-\lambda,\ \rho+\lambda+\mu,\ f,\ \lambda-f$ with
nonnegative rational coefficients, at least one of them positive.  Moreover, if
$d_i^2A_i\omega=-1$ for $i=1,2$, then $(m_1d_1)^2-(m_2d_2)^2=-(m_1^3+m_2^3)/(m_1A_1\omega)$.
\end{lemma}}
{\begin{proof}
Put $h=M_P/M_T$.  The determinant in each summand of the torque is
\begin{align*}
 z_iw_k-w_iz_k
 ={}&t^{-\lambda}A_i\eta_k-h t^{\sigma-\lambda}A_i\hat W
       +t^{\mu+\rho}\hat Z\eta_k-t^{\mu+\sigma}\hat W\zeta_k\\
   &+h t^{\lambda+\mu+\rho}B_i\hat Z
       -t^{\lambda+\mu}B_i\zeta_k.
\end{align*}
For $k=3$, multiply this expression by
$t^{3(\lambda-f)/2}d_i^3$ and use $\eta_3=t^f\omega$.
After removing the common factor $t^{(\lambda-f)/2}$, the first term
is $d_i^3A_i\omega$, and the other five exponents are
\[
 \sigma-f,\quad\lambda+\mu+\rho,\quad
 \lambda+\mu+\sigma-f,\quad2\lambda+\mu+\rho-f,\quad2\lambda+\mu-f.
\]
For $k=4,5$, multiply instead by $t^{3\lambda/2}D_{ik}^3$;
the constant and $t^f$ parts of $\eta_k$ give exactly the remaining
eight exponents in $G$.  This proves the displayed expansion.
Each exponent has the asserted nonnegative representation: the only
subtractions needing attention are
$\sigma-f=(\sigma-\lambda)+(\lambda-f)$,
$2\lambda+\mu-f=\lambda+\mu+(\lambda-f)$ and
$2\lambda+\mu+\rho-f=(\rho+\lambda+\mu)+(\lambda-f)$;
the remaining cases follow by adding nonnegative multiples of
$\lambda,\mu,f$ to these or to $\rho+\lambda+\mu$.
Finally $A_2=-m_1A_1/m_2$, so the distance constraints give directly
\[
 (m_1d_1)^2-(m_2d_2)^2
 =-\frac{m_1^2}{A_1\omega}+\frac{m_2^2}{A_2\omega}
 =-\frac{m_1^3+m_2^3}{m_1A_1\omega}.
\]
\end{proof}

}\begin{lemma}\label{lem:d1-central}
Let $m\in\Rp^5$.  Along every Diagram~1 germ all three leading triangle $w$-offsets $o_k$ are
nonzero.
\end{lemma}

\begin{proof}
Suppose not.  By Lemma~\ref{lem:d1-orders}(iv), after renumbering $T$, $o_3=0\ne o_4,o_5$,
$0<f<\lambda$, $\ord\delta_{i3}=\frac12(\lambda-f)$ and $\ord\delta_{ik}=\frac12\lambda$ for $k=4,5$.
In the first two identities of Lemma~\ref{lem:d1-syzygies} the crossed terms have orders
$\frac12(3\lambda-\ord w_{ik})$ and $\frac12(\lambda-3\ord w_{ik})$, so
$\sigma\ge\frac12(3\lambda-f)$ and $\mu+\rho\ge\frac12(\lambda-3f)$, that is,
$\sigma-\lambda\ge\frac12(\lambda-f)>0$ and $\rho+\lambda+\mu\ge\frac32(\lambda-f)>0$.
Hence $\gamma$ is described by the chart of Lemma~\ref{lem:d1-torque}, with chart coordinates of
order $\ge0$, of order $0$ for $d_i,A_i,\omega$ (so these are nonzero at $t=0$), and every
exponent in $G$ is positive.
As $\ord(Z_T-Z_P)=\mu+\rho>-\lambda$ and $\sigma>\lambda>f$, we have $z_3-z_i=-t^{-\lambda}A_i(1+o(1))$
and $w_3-w_i=t^f\omega(1+o(1))$, so the distance constraints $\delta_{i3}^2(z_3-z_i)(w_3-w_i)=1$ give
$d_i^2A_i\omega=-1$ at $t=0$.
By Lemma~\ref{lem:d1-torque}, $t^{-(\lambda-f)/2}T_P$ tends along $\gamma$ to
$-m_3(m_1d_1+m_2d_2)$, evaluated at $t=0$, and this is nonzero because
$(m_1d_1)^2-(m_2d_2)^2=-(m_1^3+m_2^3)/(m_1A_1\omega)\ne0$.  But $T_P=0$ along $\gamma$ by
Lemma~\ref{lem:d1-syzygies}.
\end{proof}

\begin{proposition}[Representation in the chart]\label{prop:d1-cover}
Let $m\in\Rp^5$ and let $\gamma$ be a Diagram~1 germ, normalized as above.
\begin{romanlist}
  \item $\lambda>\mu/2>0$.
  \item Every leading triangle $w$-offset $o_k$ is nonzero; hence $\ord w_{ik}=0$ and
  $\ord\delta_{ik}=\lambda/2$ for $i\in P$, $k\in T$.
  \item $\sigma\ge\frac32\lambda$ and $\mu+\rho\ge\frac12\lambda$.  In particular
  $(\lambda,\mu,\rho,\sigma)\in C$, with $\sigma-\lambda\ge\frac12\lambda$,
  $\rho+\frac32\mu\ge\frac12(\lambda+\mu)$ and $\rho+\lambda+\mu\ge\frac32\lambda$.
  \item Along $\gamma$ every chart coordinate of \eqref{eq:d1-chart} converges as $t\to0$; the
  limits of $\pi$, $\varpi$, of the $D_{jl}$, of the $\eta_k$ and of the differences
  $\zeta_k-\zeta_l$, $\eta_k-\eta_l$ $(k\ne l)$ are nonzero; and the limit point lies in the
  boundary fibre $\mathcal B_1$.
\end{romanlist}
\end{proposition}

\begin{proof}
(i) is Lemma~\ref{lem:d1-orders}(i),(ii), and (ii) follows from Lemmas~\ref{lem:d1-central}
and~\ref{lem:d1-orders}(iv).  (iii) On $\gamma$, Lemma~\ref{lem:d1-syzygies} gives
$M_PW_P=\sum m_im_k\delta_{ik}^3(w_i-w_k)$ and $M_TZ_T=\sum m_im_k\delta_{ik}^3(z_k-z_i)$.  By (ii)
and Lemma~\ref{lem:d1-orders}(iii) the terms of the first sum have order $\frac32\lambda$ and
those of the second $\frac32\lambda-\lambda=\frac12\lambda$; as $\ord Z_T=\ord Z_P=\mu+\rho$ the
two bounds follow.  {Substitution gives the three remaining inequalities.}
(iv) By the definitions of $\lambda,\mu,\rho,\sigma$, by Lemma~\ref{lem:d1-orders} and by (ii),
every chart coordinate is a Puiseux series of order $\ge0$, of order exactly $0$ for the listed
ones; so all converge.  By Lemma~\ref{lem:d1-chart} and (iii), each weighted equation equals
$c_0+\sum_{e>0}c_et^e$ along $\gamma$ with bounded $c_e$; it vanishes identically, so $c_0$
vanishes at the limit point.
\end{proof}

\begin{remark}\label{rem:d1-ak}
{Lemma~\ref{lem:d1-orders}} and Proposition~\ref{prop:d1-cover} give the clustering scheme
$z:1\ldots345\ldots2$, $w:3\ldots12\ldots5\ldots4$, which
\cite[\S8.1, p.~572]{AK} derives from Rule~2e, Estimate~2 and \cite[Prop.~4]{AK}.
Lemma~\ref{lem:d1-central} is the torque argument of the proof of \cite[Prop.~4]{AK}, carried out
in the chart with every competing term controlled.
\end{remark}

\subsection{The limiting triangle and its masses}

The chart reduces the limiting equations to a triangle with zero multiplier.  We classify
its solutions, identify the mass relation, and compute the rank transverse to the scaling
and rotation orbit.  This rank calculation will control the potential by the squared forces.

{\begin{lemma}\label{lem:d1-shape}
Let $m_3,m_4,m_5>0$ and let $(\zeta,\eta,D_{34},D_{35},D_{45})$, with
$\sum_km_k\zeta_k=\sum_km_k\eta_k=0$, solve the triangle system with zero multiplier.  Put
$\alpha_{kl}=\zeta_l-\zeta_k$ and $\beta_{kl}=\eta_l-\eta_k$ for $k<l$ in $T$.  Then
\begin{alphalist}
  \item $\beta_{kl}=\Lambda^{-1}\alpha_{kl}$ for one $\Lambda\in\C^*$, and $(D_{kl}\alpha_{kl})^2=
  \Lambda$ for all $k<l$;
  \item $(D_{34}\alpha_{34},D_{35}\alpha_{35},D_{45}\alpha_{45})=s\,(-1,1,-1)$ with $s^2=\Lambda$;
  the other six sign patterns admit no solution;
  \item $m_5\alpha_{34}^2=m_3\alpha_{45}^2=m_4\alpha_{35}^2$ and $Q(m_3,m_4,m_5)=0$, where
  \[
    Q=2m_3m_4m_5(m_3+m_4+m_5)-m_3^2m_4^2-m_4^2m_5^2-m_5^2m_3^2 .
  \]
\end{alphalist}
\end{lemma}}
{\begin{proof}
Let $M=\operatorname{diag}(m_3,m_4,m_5)$ and let $B$ be the symmetric matrix
\[
 B_{kl}=-m_km_lD_{kl}^3\quad(k\ne l),\qquad
 B_{kk}=\sum_{l\ne k}m_km_lD_{kl}^3.
\]
Then $B\mathbf1=0$, $B\zeta=M\phi=0$ and $B\eta=M\psi=0$.
All the $D_{kl}$ are nonzero by the distance constraints, so $B\ne0$.
The mass-centred vectors $\zeta,\eta$ cannot be independent.
If they were, $\mathbf1,\zeta,\eta$ would form a basis of $\C^3$, forcing $B$ to vanish.
Thus $\eta=\Lambda^{-1}\zeta$ with $\Lambda\ne0$, and the distance constraints give (a).

Choose $s^2=\Lambda$ and write $D_{kl}\alpha_{kl}=s\varepsilon_{kl}$,
where $\varepsilon_{kl}\in\{1,-1\}$.  Put
$u=\alpha_{45}/\alpha_{34}$ and $v=\alpha_{35}/\alpha_{34}=1+u$.
The equations $\phi_4=\phi_5=0$ give
\[
 u^2=\frac{m_5}{m_3}\frac{\varepsilon_{45}}{\varepsilon_{34}},
 \qquad
 v^2=-\frac{m_5}{m_4}\frac{\varepsilon_{35}}{\varepsilon_{34}}.
\]
Both squares are nonzero real numbers.  Thus $u$ is real or purely imaginary.
The latter possibility is excluded because $(1+u)^2$ would then have nonzero
imaginary part.  Hence $u,v$ are nonzero real numbers.  Positivity of the masses now implies
$\varepsilon_{45}=\varepsilon_{34}=-\varepsilon_{35}$.  Changing the sign
of $s$ if necessary gives (b), and the two displayed equations give the
equalities in (c).

With $x_k=m_k^{-1/2}>0$, the equalities just obtained and
$\alpha_{35}=\alpha_{34}+\alpha_{45}$ imply that one of $x_3,x_4,x_5$
is the sum of the other two.  The factorization
\[
 Q=(m_3m_4m_5)^2(x_3+x_4+x_5)(-x_3+x_4+x_5)
       (x_3-x_4+x_5)(x_3+x_4-x_5),
\]
therefore gives $Q=0$.  For positive masses, the same factorization shows that
$Q=0$ is equivalent to relation (23).
\end{proof}

}The torus $\mathbb T=\{(\theta_1,\theta_2,\theta_3)\in(\C^*)^3:\theta_3^2=\theta_1\theta_2\}$ acts by
$(\zeta,\eta,D)\mapsto(\theta_1\zeta,\theta_2\eta,\theta_3^{-1}D)$.  It multiplies $\phi$ by
$\theta_1\theta_3^{-3}$, $\psi$ by $\theta_2\theta_3^{-3}$ and $I_T$ by $\theta_1\theta_2$, and it
fixes $\chi$.

{\begin{lemma}\label{lem:d1-mb}
Let $b,c>0$, $(m_3,m_4,m_5)=(b^2c^2,\,b^2(b+c)^2,\,c^2(b+c)^2)$ and $\kappa>0$.  Let $y^*$ be the
point with $\eta=\zeta$, $\sum_km_k\zeta_k=0$, $\zeta_4-\zeta_3=\kappa m_5^{-1/2}$,
$\zeta_5-\zeta_3=-\kappa m_4^{-1/2}$, $D_{34}=m_5^{1/2}/\kappa$, $D_{35}=m_4^{1/2}/\kappa$ and
$D_{45}=-m_3^{1/2}/\kappa$.  In the coordinates $y=(\zeta_3,\zeta_4,\eta_3,\eta_4,D_{34},D_{35},
D_{45})$ let $F=(\phi_3,\phi_4,\psi_3,\psi_4,\chi_{34},\chi_{35},\chi_{45})$.  Then $F(y^*)=0$; the
Jacobian matrix $F'(y^*)$ has rank $5$, and its kernel is spanned by the generators
$(\zeta_3,\zeta_4,0,0,-\frac12D)$ and $(0,0,\eta_3,\eta_4,-\frac12D)$ of the $\mathbb T$-action; and
\[
  I_T(y^*)=\frac{\kappa^2}{m_3+m_4+m_5}\Big(\frac{m_3m_4}{m_5}+\frac{m_3m_5}{m_4}
  +\frac{m_4m_5}{m_3}\Big)>0 .
\]
\end{lemma}}
{\begin{proof}
Write $q=\zeta=\eta$ at $y^*$.  The mass relation is
$m_3^{-1/2}=m_4^{-1/2}+m_5^{-1/2}$, and hence
$q_5-q_4=-\kappa m_3^{-1/2}$.  All three distance constraints hold and
$D_{34}(q_4-q_3)=1$, $D_{35}(q_5-q_3)=-1$,
$D_{45}(q_5-q_4)=1$.
Substitution in the two force identities used in the proof of
Lemma~\ref{lem:d1-shape} gives $\phi_4=\phi_5=0$; the mass-centre identity
gives $\phi_3=0$, and $\psi=\phi$.  Thus $F(y^*)=0$.

Use the matrix $B$ from that proof, put $M=\operatorname{diag}(m_3,m_4,m_5)$,
and set $L=M^{-1}B$.
Since $B\mathbf1=Bq=0$, the vectors $\mathbf1,q$ are independent and
$B\ne0$, we have $\operatorname{rank}B=1$ and
$\ker B=\operatorname{span}\{\mathbf1,q\}$.
On the mass-centred plane
$H=\{x:\sum_km_kx_k=0\}$, the map $L$ has rank $1$ and kernel $\C q$.

The three constraint differentials are independent, since
$\partial\chi_{kl}/\partial D_{kl}=2/D_{kl}\ne0$ and the other two
constraint rows have zero entries in this column.  Their zero set is therefore
a smooth four-dimensional manifold, with local coordinates $(\zeta,\eta)\in H^2$.
On its tangent space at $y^*$,
\[
 dD_{kl}=-\frac{D_{kl}}{2(q_l-q_k)}
   \bigl(d\zeta_l-d\zeta_k+d\eta_l-d\eta_k\bigr),
\]
and differentiation of the forces gives
\[
 \begin{pmatrix}d\phi\\d\psi\end{pmatrix}
 =-\frac12
 \begin{pmatrix}L&3L\\3L&L\end{pmatrix}
 \begin{pmatrix}d\zeta\\d\eta\end{pmatrix}.
\]
The scalar $2\times2$ matrix here is invertible, so this map has rank $2$
and kernel $\C(q,0)\oplus\C(0,q)$.  Since either two components determine
each mass-centred force vector, the full Jacobian $F'(y^*)$ has rank
$3+2=5$.  The formula for $dD$ identifies its two kernel generators with
the stated infinitesimal torus actions.

The elementary mass-centre identity
\[
 \sum_km_kq_k^2=\frac1{m_3+m_4+m_5}
       \sum_{k<l}m_km_l(q_l-q_k)^2
\]
gives the displayed value of $I_T(y^*)$, which is positive because $q$ is
a nonzero real vector.  Conversely every positive solution of the mass
relation has the stated $(b,c)$ parametrization: take
$b/c=\sqrt{m_4/m_5}$ and $bc=\sqrt{m_3}$.
\end{proof}
By equivariance, the same rank and kernel statements hold at every point of
$O=\mathbb T\cdot y^*$, and $I_T\ne0$ on $O$.

}\begin{proposition}[The limiting mass condition]\label{prop:d1-first}
Let $m\in\Rp^5$ and let $\gamma$ be a Diagram~1 germ.  Then $Q(m_3,m_4,m_5)=0$, that is, the
triangle masses satisfy relation (23) of \cite[\S8.1]{AK}.  After renumbering $T$ so that
$m_3^{-1/2}=m_4^{-1/2}+m_5^{-1/2}$, the limit $y_0$ of
$y(t)=(\zeta_3,\zeta_4,\eta_3,\eta_4,D_{34},D_{35},D_{45})(t)$ lies on the orbit $O$ of
Lemma~\ref{lem:d1-mb}.
\end{proposition}

\begin{proof}
By Proposition~\ref{prop:d1-cover}(iv) the limit lies in $\mathcal B_1$, so its triangle part
solves the triangle system with zero multiplier.  In the notation of Lemma~\ref{lem:d1-shape},
$\phi_4=m_3\alpha_{34}D_{34}^3-m_5\alpha_{45}D_{45}^3$ and
$\phi_5=m_4\alpha_{45}D_{45}^3+m_3\alpha_{35}D_{35}^3$, and $m_3\phi_3=-m_4\phi_4-m_5\phi_5$;
likewise for $\psi$.  So $\phi=0$ and $\psi=0$ are relations (22) and (21) of \cite[\S8.1]{AK} in
cross-multiplied form, and Lemma~\ref{lem:d1-shape} gives $Q=0$.
After the renumbering, Lemma~\ref{lem:d1-shape}(c) and $\alpha_{35}=\alpha_{34}+\alpha_{45}$
leave only $(\alpha_{34},\alpha_{45},\alpha_{35})=\kappa(x_5,-x_3,-x_4)$ with $\kappa\in\C^*$, and
Lemma~\ref{lem:d1-shape}(a),(b) then determine $\eta$ and $D$ from $(\kappa,s)$.  The torus
$\mathbb T$ acts transitively on these pairs $(\kappa,s)$, and $y^*$ is the solution with
$\kappa>0$, $s=-1$.  Hence the solutions form the single orbit $O$, and $y_0\in O$.
\end{proof}

\subsection{Excluding Diagram~1}

We first express the triangle potential and its differential in terms of the forces.
Combined with the preceding rank calculation, these identities yield the quadratic
estimate needed to compare orders along a germ.

{\begin{lemma}\label{lem:d1-virial}
For arbitrary $\zeta,\eta\in\C^3$ and $D_{34},D_{35},D_{45}\in\C$,
\begin{gather*}
  \sum_{k\in T}m_k\zeta_k\psi_k=U_T+\sum_{k<l}m_km_lD_{kl}\chi_{kl},\\
  2\,dU_T+\sum_{k\in T}m_k\big(\psi_k\,d\zeta_k+\phi_k\,d\eta_k\big)
  =\sum_{k<l}m_km_l\big(D_{kl}\,d\chi_{kl}-2\chi_{kl}\,dD_{kl}\big).
\end{gather*}
\end{lemma}}
{\begin{proof}
For a single edge put $\alpha=\zeta_l-\zeta_k$,
$\beta=\eta_l-\eta_k$, $D=D_{kl}$ and $\chi=D^2\alpha\beta-1$.
Pairing the $k$ and $l$ terms in the force sum gives
$m_km_lD^3\alpha\beta=m_km_lD(1+\chi)$, which proves the first identity
after summing over the three edges.  For the second use the exact
single-edge identity
\[
 D\,d\chi-2\chi\,dD
       =2\,dD+D^3(\beta\,d\alpha+\alpha\,d\beta).
\]
Multiplication by $m_km_l$ and summation gives the result.
\end{proof}

}{On the constraint manifold $\{\chi=0\}$, these identities reduce to
$U_T=\sum_km_k\zeta_k\psi_k$ and
$dU_T=-\frac12\sum_km_k(\psi_k\,d\zeta_k+\phi_k\,d\eta_k)$.
Thus both the potential and its differential vanish on $O$.  Together with the rank
calculation in Lemma~\ref{lem:d1-mb}, this gives the quadratic estimate below.}

\begin{lemma}\label{lem:d1-estimate}
Every $y_0\in O$ has a neighbourhood $V\subset\C^7$ with a constant $c_V$ such that
$|U_T(y)|\le c_V\,|(\phi_3,\phi_4,\psi_3,\psi_4)(y)|^2$ for every $y\in V$ with $\chi(y)=0$.
\end{lemma}

\begin{proof}
{By Lemma~\ref{lem:d1-mb} and equivariance, five components $F_5$ of $F$
have independent differentials at $y_0$.  Their common zero set $F_5^{-1}(0)$ is therefore
a $2$-dimensional manifold near $y_0$.  The orbit map $\theta\mapsto\theta y_0$ immerses
a neighbourhood of $1\in\mathbb T$ onto a $2$-dimensional piece of $O$ contained in this
manifold.  Shrinking $V$, we obtain $F_5^{-1}(0)\cap V\subset O$.
The implicit function theorem then gives
$\operatorname{dist}(y,O)\le c\,|F_5(y)|\le c\,|(\phi_3,\phi_4,\psi_3,\psi_4)(y)|$
for $y\in V$ with $\chi(y)=0$.

Near $y_0$, the constraint manifold $\{\chi=0\}$ is the graph of $D$ over
$(\zeta_3,\zeta_4,\eta_3,\eta_4)$, because
$\partial\chi_{kl}/\partial D_{kl}=2D_{kl}^{-1}(\chi_{kl}+1)\ne0$.
In these local coordinates, both $U_T$ and its differential vanish on $O$
(Lemma~\ref{lem:d1-virial}).  Taylor's formula at a nearest point of $O$ gives
$|U_T(y)|\le c'\operatorname{dist}(y,O)^2$.  Combining this with the preceding distance
estimate proves the assertion.}
\end{proof}

\begin{theorem}\label{thm:d1}
Let $m\in\Rp^5$.  No curve in $A_m$ has a Diagram~1 germ (Definition~\ref{def:germ}).  In
particular there is no Diagram~1 germ with an unbounded four-product, which is the case $r=1$ of
Theorem~\ref{thm:exclusion}\,(a).
\end{theorem}

\begin{proof}
{Suppose that $\gamma$ is a Diagram~1 germ.  Normalize it as above and
renumber $T$ as in Proposition~\ref{prop:d1-first}.  Along $\gamma$, the triangle distance constraints give $\chi(y(t))=0$
identically, and $y(t)\to y_0\in O$.  We compare the order of the triangle potential obtained
from the force estimate with the order given by the virial identity.

Since $E^z_k=E^w_k=0$ along $\gamma$, Lemma~\ref{lem:d1-chart} gives
$\phi_k=t^{\mu/2}(z_k-X^z_k)$ and $\psi_k=t^{3\mu/2}(w_k-X^w_k)$, with the exponents listed
there and bounded coefficients.  By Proposition~\ref{prop:d1-cover}(iii),
$\frac32\mu+\rho\ge\frac12(\lambda+\mu)$.  Hence $|\phi(y(t))|=O(|t|^r)$ with
$r=\min(\frac32\mu,\frac12(\lambda+\mu))>0$, and $|\psi(y(t))|=O(|t|^{3\mu/2})$.
Lemma~\ref{lem:d1-estimate} therefore gives
\[
 |U_T(y(t))|=O(|t|^{2r}),\qquad
 2r=\min(3\mu,\lambda+\mu).
\]

On the other hand, Lemma~\ref{lem:d1-virial} and $\chi=0$ give
$U_T(y(t))=\sum_km_k\zeta_k\psi_k=t^{3\mu/2}\sum_km_k\zeta_k(w_k-X^w_k)$.
Because $w_k=W_T+\eta_k$ and $\sum_km_k\zeta_k=0$, we have
$\sum_km_k\zeta_kw_k=I_T(y(t))$.  Moreover,
$\ord X^w_k\ge\frac32\lambda$ by Proposition~\ref{prop:d1-cover}(ii).  Thus
\[
 U_T(y(t))=t^{3\mu/2}\bigl(I_T(y_0)+o(1)\bigr),\qquad I_T(y_0)\ne0.
\]
Comparison with the quadratic estimate gives
\[
 \frac32\mu\ge\min(3\mu,\lambda+\mu)
 \quad\Longrightarrow\quad
 \lambda+\mu\le\frac32\mu
 \quad\Longrightarrow\quad
 \lambda\le\frac12\mu,
\]
whereas Proposition~\ref{prop:d1-cover}(i) gives $\lambda>\mu/2$.}
\end{proof}

\begin{remark}\label{rem:d1-wall}
{In the chart family \eqref{eq:d1-chart}, the outer bodies are $z$-circled
exactly when $\lambda\ge\mu/2$. At $\lambda=\mu/2$, a rotation gives a Diagram~4 chart;
for $\lambda<\mu/2$, the diagram is Diagram~5. These two cases are treated in
Section~\ref{sec:d45}.}
\end{remark}

\section{Diagram 2}\label{sec:d2}

We exclude Diagram~2 by resolving its pair and triangle in a blow-up chart.
The torque identity restricts the boundary points to explicit mass loci; local
analytic reduction then turns continuation from the boundary into scalar Taylor
conditions. The coefficients of these reduced equations yield the exclusion.

{Diagram~2 in \cite[Fig.~11]{AK} consists of the $z$-segment $12$ and the $w$-triangle $345$.
Every vertex carries a circle of the same colour as its edges \cite[\S8.2, p.~572]{AK}.
Since \eqref{eq:system} is invariant under exchanging $z$ and $w$ and renumbering the bodies,
we may assume that a Diagram~2 germ has this form.
Throughout this section, $A,A'$ denote the bodies in the pair, $\{A,A'\}=\{1,2\}$, and $B,B'$
denote bodies in the triangle $\{3,4,5\}$.  We write $M=m_1+m_2$, $M_T=m_3+m_4+m_5$ and
$m'=(m_3,m_4,m_5)$.  We reserve $t$ for a chart coordinate and use $s$ for the germ parameter
called $t$ in Lemma~\ref{lem:germ}; thus $\ord$ denotes order in $s$.
The notation $x,y,u$ agrees with \cite[\S8.2]{AK}.}

{Proposition~\ref{prop:d2-cover} places every Diagram~2 germ at a finite point of the
boundary fibre $t=0$. The torque identity gives $\Phi=0$ there and recovers the
mass loci $L_j$, $K_j$, $P_j$ of \cite[\S8.2]{AK}
(Proposition~\ref{prop:d2-first}). At a boundary limit with $\Phi=0$, if twenty
selected equations define an analytic graph over $t$, the remaining equations reduce to
\[
 tq(t)=0,\qquad q(0)=\Phi,\qquad q'(0)=c_\sigma^2Q_T(\tilde v).
\]
The quadratic form $Q_T$ is determined by the triangle
(Lemmas~\ref{lem:d2-isolation} and~\ref{lem:d2-second}). At fold points we use
one- or two-parameter reductions. Section~\ref{sec:d2-cases} applies these reductions
to each mass locus. At a regular boundary point, $q'(0)\ne0$ gives isolation directly.}

\subsection{\texorpdfstring{The blow-up chart and its boundary}{The blow-up chart and its boundary}}\label{sec:d2-chart}

{We use coordinates $(t,Z,W,D)\in\C^{21}$, where $Z=(Z_1,\dots,Z_5)$,
$W=(W_1,\dots,W_5)$ and $D=(D_{kl})_{k<l}$.  Their relation to the coordinates of
\eqref{eq:system} is}
\begin{equation}\label{eq:d2-chart}
\begin{gathered}
  z_A=t^{-1}Z_A,\quad w_A=tW_A,\quad z_B=tZ_B,\quad w_B=t^{-1}W_B,\\
  \delta_{12}=D_{12},\quad \delta_{BB'}=D_{BB'},\quad \delta_{AB}=tD_{AB}.
\end{gathered}
\end{equation}
{Multiplying the coordinate equations of \eqref{eq:system} for
$z_A,w_A,z_B,w_B$ by $t,t^{-1},t^{-1},t$, respectively, gives the polynomial rows}
\begin{equation}\label{eq:d2-rows}
\begin{aligned}
  F^Z_A&=Z_A-m_{A'}D_{12}^3(Z_A-Z_{A'})-t^3\textstyle\sum_Bm_BD_{AB}^3(Z_A-t^2Z_B),\\
  F^W_A&=W_A-m_{A'}D_{12}^3(W_A-W_{A'})+t\textstyle\sum_Bm_BD_{AB}^3(W_B-t^2W_A),\\
  F^Z_B&=Z_B-\textstyle\sum_{B'\ne B}m_{B'}D_{BB'}^3(Z_B-Z_{B'})
         +t\textstyle\sum_Am_AD_{AB}^3(Z_A-t^2Z_B),\\
  F^W_B&=W_B-\textstyle\sum_{B'\ne B}m_{B'}D_{BB'}^3(W_B-W_{B'})
         -t^3\textstyle\sum_Am_AD_{AB}^3(W_B-t^2W_A),
\end{aligned}
\end{equation}
{The distance equations become $T_{12}=D_{12}^2(Z_2-Z_1)(W_2-W_1)-1$,
$T_{BB'}=D_{BB'}^2(Z_{B'}-Z_B)(W_{B'}-W_B)-1$ ($B<B'$) and
$T_{AB}=D_{AB}^2(t^2Z_B-Z_A)(W_B-t^2W_A)-1$.
We replace the gauge in \eqref{eq:system} by $G_1=Z_2-Z_1-1$ and $G_2=W_5-1$.
The common zero set of these $22$ polynomials is denoted by $\mathcal X_m\subset\C^{21}$,
and its boundary fibre by $\mathcal B_m=\mathcal X_m\cap\{t=0\}$.
A point of $\mathcal B_m$ specifies the limiting pair, triangle and crossed inverse
distances in this chart.}

\begin{definition}[Shape factors]\label{def:d2-shapes}
For $\varsigma\in\{\pm1\}^3$ let $N_\varsigma(w_3,w_4,w_5)$ be the numerator of
$w_3/(\varsigma_3(w_5-w_4)^2)+w_4/(\varsigma_4(w_3-w_5)^2)+w_5/(\varsigma_5(w_4-w_3)^2)$.  The
quintics $S_1,S_2,S_3$ are $N_\varsigma$ for $\varsigma=(1,-1,-1),(1,-1,1),(1,1,-1)$, and
$N_{(1,1,1)}$ is a constant multiple of $S_4S_5$, where
\begin{gather*}
  S_4=w_3^3+w_4^3+w_5^3-w_3^2w_4-w_3w_4^2-w_3^2w_5-w_3w_5^2-w_4^2w_5-w_4w_5^2+3w_3w_4w_5,\\
  S_5=w_3^2+w_4^2+w_5^2-w_3w_4-w_3w_5-w_4w_5 .
\end{gather*}
The quartics $S_6,S_7,S_8$ are the factors of
$\operatorname{Res}_p\bigl((w_4-w_3)^3-p^3(w_5-w_4)^3,\ (w_4-w_3)^3-(-1-p)^3(w_3-w_5)^3\bigr)$
(which is $S_5^3S_6S_7S_8$ up to a constant) whose coefficients of $w_3^3w_4$ are $-2$, $1$ and
$-5$.  These are the factors of the complex three-body central configuration locus of
\cite[Appendix, pp.~585--586]{AK}.
\end{definition}

\begin{lemma}[computer-assisted]\label{lem:d2-fibre}
Let $m\in\Rp^5$.  A point $p=(0,Z,W,D)$ is a boundary point if and only if
\begin{alphalist}
  \item $Z_1=-m_2/M$, $Z_2=m_1/M$, $D_{12}^3=1/M$, $W_1=-m_2D_{12}$ and $W_2=m_1D_{12}$;
  \item $F^Z_B=F^W_B=T_{BB'}=0$ at $t=0$ (system \eqref{eq:system} for the bodies $3,4,5$,
  without its gauge) and $W_5=1$;
  \item $D_{AB}^2Z_AW_B=-1$ for all $A,B$.
\end{alphalist}
At a boundary point $\sum_Bm_BZ_B=\sum_Bm_BW_B=0$, every $W_B\ne0$, the $W_B$ are pairwise distinct, and
the \emph{triangle} $(Z_B,W_B,D_{BB'})$ is of one of two kinds.
\begin{romanlist}
  \item \emph{Lagrange:} $D_{BB'}=M_T^{-1/3}\omega^{a_{BB'}}$, $\omega=e^{2\pi i/3}$,
  $a_{BB'}\in\{0,1,2\}$ ($27$ classes), and $\kappa=(W_4-W_3)/(W_5-W_4)$ is a root of
  $\omega^{a_{45}}\kappa^2+(\omega^{a_{34}}+\omega^{a_{45}}-\omega^{a_{35}})\kappa+\omega^{a_{34}}$.
  The root is double exactly for the six classes with $a_{34},a_{35},a_{45}$ distinct; these are
  the \emph{fold} triangles.  $(W_3,W_4,W_5)$ is a zero of $S_5$ if $a_{34}=a_{35}=a_{45}$ or the
  $a_{BB'}$ are distinct, of $S_6$ if $a_{34}=a_{35}\ne a_{45}$, of $S_7$ if
  $a_{34}=a_{45}\ne a_{35}$, and of $S_8$ if $a_{35}=a_{45}\ne a_{34}$.
  \item \emph{Euler:} $Z_B=c^{-2}W_B$ for some $c\in\C^*$, $D_{BB'}=e_{BB'}c/(W_{B'}-W_B)$ for
  $B<B'$ with $e_{BB'}=\pm1$, and $(W_3,W_4,W_5)$ is a zero of $S_1$, $S_2$, $S_3$ or $S_4S_5$
  according to the class of $(e_{34},e_{35},e_{45})$ modulo a global sign.  An Euler triangle on
  $S_5$ is a Lagrange fold triangle.
\end{romanlist}
\end{lemma}

\begin{proof}
{At $t=0$, the crossed terms in \eqref{eq:d2-rows} vanish and
$T_{AB}=-D_{AB}^2Z_AW_B-1$, giving (b) and (c).  Together with $G_1$, the remaining pair
equations are $Z_A=m_{A'}D_{12}^3(Z_A-Z_{A'})$,
$W_A=m_{A'}D_{12}^3(W_A-W_{A'})$ and $T_{12}=0$, whose solutions are those in (a).
Summing $m_BF^Z_B$ and $m_BF^W_B$ gives the two centre-of-mass relations.
The distance equations also ensure that the triangle positions are nonzero and distinct:
if $W_B=0$, then $T_{1B}=-1$, and if $W_B=W_{B'}$, then $T_{BB'}=-1$.}

{Suppose first that $(1)_B,(Z_B)_B,(W_B)_B$ are linearly independent.
We claim that $D_{BB'}^3=1/M_T$ for all three edges.  The linear map
$X\mapsto(\sum_{B'\ne B}m_{B'}D_{BB'}^3(X_B-X_{B'}))_B$ on $\C^3$ fixes $(Z_B)$ and $(W_B)$
by (b), and annihilates $(1)_B$.  By the centre-of-mass relations,
$X\mapsto X-M_T^{-1}(\sum_Bm_BX_B)(1)_B$ has the same action on this basis.
The two maps therefore agree.  Comparing their entries at $(B,B')$, $B\ne B'$, gives
$m_{B'}D_{BB'}^3=m_{B'}/M_T$.
The coordinate equations now reduce to the centre-of-mass relations, and the distance
equations give $Z_{B'}-Z_B=M_T^{2/3}\omega^{a_{BB'}}/(W_{B'}-W_B)$.
Substituting these differences into $Z_5-Z_3=(Z_4-Z_3)+(Z_5-Z_4)$ gives the quadratic for
$\kappa$.}

{If the three vectors are dependent, then $Z=\alpha W+\beta$.
The centre-of-mass relation gives $\beta=0$, and the distance equations give $\alpha\ne0$.
Writing $\alpha=c^{-2}$ yields the asserted formula for $D_{BB'}$.
To identify the double roots, write $c_{BB'}=\omega^{a_{BB'}}$.
When two of the $c_{BB'}$ agree, the discriminant is $b(b-4a)\ne0$, with $|a|=|b|=1$,
where $a$ is the repeated value and $b$ the remaining one.
When all three are distinct, $c_{34}+c_{45}=-c_{35}$ and $c_{34}c_{45}=c_{35}^2$, so the
discriminant vanishes.  The identification of the corresponding shape factors is
computer-assisted.}
\end{proof}

{This identification uses exact arithmetic over $\mathbb Q(\omega)$.
For each of the $27$ Lagrange classes, the corresponding factor $S_j$, expressed as a
polynomial in $\kappa$, is reduced modulo the quadratic for $\kappa$; for fold classes it is
evaluated at the double root.  For each of the eight sign vectors $(e_{34},e_{35},e_{45})$,
the corresponding $N_\varsigma$ is identified with one of $\pm S_1$, $\pm S_2$, $\pm S_3$,
$\pm S_4S_5$.  Finally, for Euler triangles on $S_5$, the three quantities $D_{BB'}^3$ are
checked to coincide modulo $S_5$ (Appendix~\ref{app:verification}).}

\begin{definition}[Mass polynomials]\label{def:d2-mass}
For positions $W_3,W_4,W_5$ put $f_B=m_B^2/W_B$,
$\Lambda=f_3^2+f_4^2+f_5^2-2f_3f_4-2f_3f_5-2f_4f_5$ and
$\Theta=3(f_3^2+f_4^2+f_5^2)+2f_3f_5+2f_4f_5-6f_3f_4$.  Equations (25) and (28) of
\cite[\S8.2]{AK} are
\begin{equation*}
  \Lambda=0,\qquad
  \Lambda\,y^2+2(f_3-f_4+f_5)(f_3-f_4-f_5)\,y+\tfrac12(\Theta-\Lambda)=0 .
\end{equation*}
For a rational expression $H$ in $W_3,W_4,W_5$ whose denominator is a monomial, let $H^\flat$ be
the polynomial in $W_4$ obtained by clearing the position denominators, substituting $W_5=1$ and
$W_3=-(m_4W_4+m_5)/m_3$ (centre of mass), and clearing the resulting power of $m_3$.  Define
$L_j(m')$ and $P_j(y)$ by
\[
  \operatorname{Res}_{W_4}(S_j^\flat,\Lambda^\flat)=c_j\,L_j,\qquad
  \operatorname{Res}_{W_4}\bigl(S_j^\flat,(28)^\flat\bigr)=m_3^{2\deg S_j}P_j(y),
\]
where $c_j$ has no zero on $\Rp^3$; for $j=5$ the first resultant is
$m_3^4M_T^2L_5L_5'$.  Write $P_j(y)=E_j(y^2)+yO_j(y^2)$ and $K_j=P_j(0)=E_j(0)$.
\end{definition}

{These definitions recover the factors $L_1,\dots,L_8,L_5'$ and the
polynomials $P_1,\dots,P_8$ in \cite[pp.~573--575]{AK} from the shape factors in its appendix.
They agree with the coefficients and symmetries displayed in \cite{AK}.
In particular, $2L_5=(m_3-m_4)^2+(m_4-m_5)^2+(m_5-m_3)^2$, while $L_5'$ has degree $12$
and $91$ positive coefficients.
Permuting $(m_3,m_4,m_5)$ permutes $\{L_1,L_2,L_3\}$ and $\{L_6,L_7,L_8\}$ up to nonzero
constants.  Exchanging bodies $3$ and $4$ sends $S_1$ to $-S_2$, exchanges
$P_1\leftrightarrow P_2$ and $P_6\leftrightarrow P_7$, and fixes $P_3,P_4,P_5,P_8$.}

{The verification is exact over $\mathbb Q$, treating the masses and
$y$ as indeterminates.  The resultants defining $L_j$ and $L_5'$ are computed and factored.
Those defining $P_j$ are determined by evaluations on unisolvent grids modulo three primes,
with a priori bounds on the degrees and coefficients.  Expansion and substitution then
verify the decompositions $P_j=E_j+yO_j$, the leading coefficients and the symmetries
(Appendix~\ref{app:verification}).}
\subsection{\texorpdfstring{Limits of Diagram~2 germs in the chart}{Limits of Diagram 2 germs in the chart}}\label{sec:d2-cover}

We now relate the polynomial chart to the singular sequences of Section~\ref{sec:ak}.
The following order estimates give a finite boundary limit for every Diagram~2 germ,
so the local analysis of $\mathcal B_m$ covers all such germs.

\begin{proposition}[Existence of a boundary limit]\label{prop:d2-cover}
Let $m\in\Rp^5$ and let $\gamma$ be a Diagram~2 germ on a curve in $A_m$.  Along $\gamma$ put
$t=(z_{12}w_5)^{-1/2}$ \textup{(}either branch\textup{)} and $\lambda=(tz_{12})^{-1}$, and let
$(Z,W,D)$ be given by \eqref{eq:d2-chart} applied to the rotated configuration
$(\lambda z,\lambda^{-1}w,\delta)$.  Then $t,Z,W,D$ are Puiseux series in $s$ with $\ord t>0$ and
nonnegative orders, $(t,Z,W,D)\in\mathcal X_m$ for $s\ne0$, and the limit $p$ as $s\to0$ is a boundary point at
which $Z_1,Z_2$, $W_3,W_4,W_5$ and all $D_{kl}$ are nonzero.
\end{proposition}

\begin{proof}
{The equations in \eqref{eq:system}, apart from the gauge, are invariant
under $(z,w)\mapsto(\lambda z,\lambda^{-1}w)$.  The chart equations are obtained from them by
multiplying by powers of $t$, and the chosen normalization gives $G_1=\lambda tz_{12}-1=0$
and $G_2=\lambda^{-1}tw_5-1=t^2z_{12}w_5-1=0$.
Thus the chart point belongs to $\mathcal X_m$.

To prove convergence, we estimate the order of each coordinate.  A rotation of $\gamma$
before this construction changes neither $t$ nor the resulting chart point.
We may therefore use the normalization of Lemma~\ref{lem:germ} and \cite[p.~547]{AK}, in
which the maximal components of $\mathcal Z$ and $\mathcal W$ have the same order.
We measure orders in units for which this common order is $-2$.
Fact~\ref{fact:rules} then gives the following estimates.}
\begin{enumerate}
  \item The pair is an isolated component of the $z$-diagram with $z$-circles, so its $z$-stroke
  is maximal (Rule~1e): $\ord z_{12}=-2$; by Estimate~2, $r_{12}\approx1$, so $\ord w_{12}=2$ and
  $\ord\delta_{12}=0$.
  \item The triangle is an isolated $w$-component with $w$-circles; it has a maximal $w$-stroke
  (Rule~1e), and by Rule~2e all three are maximal: $\ord w_{BB'}=-2$, $\ord z_{BB'}=2$,
  $\ord\delta_{BB'}=0$.
  \item The circles give $\ord z_A=\ord w_B=-2$.  The bodies $3,4,5$ carry no $z$-stroke and
  $1,2$ no $w$-stroke, hence no circle of that colour (Rule~1a): $\ord z_B>-2$, $\ord w_A>-2$.
  \item Hence $\ord z_{AB}=\ord w_{AB}=-2$ and $\ord\delta_{AB}=2$.
  \item In $w_A=\sum_{l\ne A}m_l\delta_{Al}^3(w_A-w_l)$ the pair term has order $2$ and the
  crossed terms order $4$, so $\ord w_A\ge2$; likewise $\ord z_B\ge2$.
  \item So $\ord t=2$ and $\ord\lambda=0$; $Z_A=\lambda tz_A$, $W_B=\lambda^{-1}tw_B$, $D_{12}$,
  $D_{BB'}$ and $D_{AB}=\delta_{AB}/t$ have order exactly $0$, and $W_A=\lambda^{-1}w_A/t$,
  $Z_B=\lambda z_B/t$ have order $\ge0$.
\end{enumerate}
{All chart coordinates consequently converge, with $t\to0$.
Those of order exactly $0$ have nonzero limits.
Since $\mathcal X_m$ is closed, the limit belongs to it and satisfies $t=0$.}
\end{proof}

\subsection{\texorpdfstring{The torque identity and the mass conditions}{The torque identity and the mass conditions}}\label{sec:d2-first}

The pair torque gives a scalar equation that must vanish at every boundary limit.
We first derive this equation, then eliminate the triangle coordinates to obtain
conditions on the masses.

\begin{lemma}[Syzygies]\label{lem:d2-syzygy}
Identically in $(t,Z,W,D)$,
\begin{gather*}
  \textstyle\sum_{k=1}^5m_k(W_kF^Z_k-Z_kF^W_k)=0,\qquad
  \sum_Am_A(W_AF^Z_A-Z_AF^W_A)=t\,Q,\\
  Q=\textstyle\sum_{A,B}m_Am_BD_{AB}^3(t^4W_AZ_B-Z_AW_B).
\end{gather*}
At $t=0$ the only rows with a nonzero $t$-derivative are $F^W_A$ and $F^Z_B$, with
$\partial_tF^W_A=\sum_Bm_BD_{AB}^3W_B$ and $\partial_tF^Z_B=\sum_Am_AD_{AB}^3Z_A$.  At a boundary point,
$Q=\Phi:=\sum_Bm_B\varphi_B$ with $\varphi_B=m_1D_{1B}+m_2D_{2B}$, and
$\partial_tF^Z_B=-\varphi_B/W_B$.
\end{lemma}

\begin{proof}
Since $W_AF^Z_A=w_AF^z_A$, $Z_AF^W_A=z_AF^w_A$, $W_BF^Z_B=w_BF^z_B$ and $Z_BF^W_B=z_BF^w_B$ for
the coordinate equations $F^z_k,F^w_k$ of \eqref{eq:system}, the first identity is the torque
identity $\sum_km_k(w_kF^z_k-z_kF^w_k)=0$, which holds by antisymmetry.  In the second sum the
terms with $D_{12}$ cancel by antisymmetry and the crossed terms give $tQ$.  The derivatives are
read off from \eqref{eq:d2-rows}, and Lemma~\ref{lem:d2-fibre}(c) gives
$D_{AB}^3Z_AW_B=-D_{AB}$ and $D_{AB}^3Z_A=-D_{AB}/W_B$.
\end{proof}

{To express the torque condition in terms of the triangle, we first describe
the crossed inverse distances and their signs.}  At a boundary point fix $v_B$ with $v_B^2=1/W_B$.  By Lemma~\ref{lem:d2-fibre}(a),(c),
$D_{1B}^2=M/(m_2W_B)$ and $D_{2B}^2=-M/(m_1W_B)$, so $D_{1B}=s_B\sqrt{M/m_2}\,v_B$ and
$D_{2B}=is'_B\sqrt{M/m_1}\,v_B$ with $s_B,s'_B\in\{\pm1\}$ (positive square roots); these are the
$64$ crossed-sign assignments.  Put $\sigma_B=s_Bs'_B$,
$c_\pm=\sqrt M\,(m_1/\sqrt{m_2}\pm i\,m_2/\sqrt{m_1})$ and $x=i(m_2/m_1)^{3/2}$.  Then
$c_+c_-=M(m_1^3+m_2^3)/(m_1m_2)>0$, $\varphi_B=c_{\sigma_B}s_Bv_B$, and
$m_2D_{2B}/(m_1D_{1B})=\sigma_Bx$ with $x^2=-m_2^3/m_1^3$.  There are two relative-sign cases:
\[
 \sigma_3=\sigma_4=\sigma_5=\sigma,\qquad
 \sigma_k=-\sigma,\quad \sigma_B=\sigma\quad(B\ne k).
\]
We refer to these as the equal-sign case and the opposite-sign case, respectively;
they comprise two and six of the eight sign patterns. In the opposite-sign case,
$k$ is the body with the opposite sign.  In both cases put $\tilde v_B=\varphi_B/c_\sigma$, so that $\tilde v_B=s_Bv_B$ if
$\sigma_B=\sigma$ and $\tilde v_k=\rho\,s_kv_k$ for the body with the opposite sign $k$, where
$\rho=c_{-\sigma}/c_\sigma$ satisfies $|\rho|=1$, $\rho\ne\pm1$.  Thus
$\Phi=c_\sigma\sum_Bm_B\tilde v_B$.
\begin{proposition}[Necessary mass conditions]\label{prop:d2-first}
Let $m\in\Rp^5$ and let $p$ be the boundary limit of a Diagram~2 germ
\textup{(}Proposition~\ref{prop:d2-cover}\textup{)}.  Then $\Phi(p)=0$, and:
\begin{alphalist}
  \item in the equal-sign case, its triangle satisfies \textup{(25)}, $\Lambda=0$, and
  $L_j(m')=0$ for every $j\ne5$ with $S_j(W_3,W_4,W_5)=0$ \textup{(}$L_5L_5'=0$ if $j=5$\textup{)};
  \item in the opposite-sign case with $k=5$ and common sign $\sigma=\sigma_3=\sigma_4$, its triangle satisfies
  \textup{(28)} with $y=\sigma(x+x^{-1})/2$, and $u:=y^2=-(T-1)^2/(4T)$, $T=(m_1/m_2)^3$.  Hence
  $P_j(y)=0$ for every $j$ with $S_j(W_3,W_4,W_5)=0$: if $m_1=m_2$ then $K_j(m')=0$, and if
  $m_1\ne m_2$ then $u<0$ and $E_j(u)=O_j(u)=0$.
\end{alphalist}
\end{proposition}

\begin{proof}
{All equations vanish along the germ, and $t\not\equiv0$.
Lemma~\ref{lem:d2-syzygy} therefore gives $Q\equiv0$ along it, and continuity implies
$\Phi(p)=Q(p)=0$, including at fold points.}

{In case (a), $\Phi=c_\sigma\sum_Bm_Bs_Bv_B$ with $c_\sigma\ne0$.
The product of $m_3v_3\pm m_4v_4\pm m_5v_5$ over the four sign choices equals $\Lambda$,
because $(m_Bv_B)^2=f_B$.  Hence the vanishing of the torque gives (25).
In case (b), put $r_B=m_1m_BD_{1B}$, so that $r_B^2=(m_1^2M/m_2)f_B$, and $x'=\sigma x$.
The torque becomes $\Phi=(r_3+r_4)(1+x')+r_5(1-x')$, as in (26) of \cite[\S8.2]{AK}.
Its product over the four sign changes of $r_3,r_4$ is $4x'^2(m_1^2M/m_2)^2$ times the
left-hand side of (28), with $y=(x'+1/x')/2$.}

{Thus $W_4$ is a common root of $S_j^\flat$ and $\Lambda^\flat$ in
case (a), or of the shape polynomial and $(28)^\flat$ in case (b).
The resultants in Definition~\ref{def:d2-mass} vanish, and their nonvanishing prefactors
$c_j\ne0$, $m_3\ne0$ give the asserted mass conditions.
Finally, $y=iv$ with $v=\sigma\bigl((m_2/m_1)^{3/2}-(m_1/m_2)^{3/2}\bigr)/2\in\R$,
which vanishes if and only if $m_1=m_2$.
Since $E_j,O_j$ have real coefficients, separating the real and imaginary parts of
$P_j(iv)=E_j(-v^2)+ivO_j(-v^2)=0$ proves the last assertion.}
\end{proof}

\subsection{\texorpdfstring{Local exclusion of boundary limits}{Local exclusion of boundary limits}}\label{sec:d2-isolation}

We turn the torque equation into a local criterion for excluding a boundary point.
An invertible coordinate Jacobian reduces the problem to one analytic function;
a corank-two reduction instead yields three equations in two parameters.

Let $F_*$ denote the twenty rows other than $F^Z_2$ and $F^Z_5$.  A boundary point $p$ is
\emph{excluded} if no holomorphic curve germ $c\colon(\C,0)\to(\mathcal X_m,p)$ has
$t\circ c\not\equiv0$; an isolated point of $\mathcal X_m$ is excluded.

\begin{lemma}[Local analytic lemma]\label{lem:d2-isolation}
Let $m\in\Rp^5$ and $p\in\mathcal B_m$.
\begin{alphalist}
  \item If $p$ is excluded, it is not the boundary limit of a Diagram~2 germ.
  \item Suppose $dF_*(p)$ has rank $20$, so that near $p$ the set
  $\{F_*=0\}$ is a smooth curve $C$.  Then near $p$ the set $\mathcal X_m$ consists of the
  points of $C$ at which $tQ=0$, and $p$ is excluded if $Q|_C\not\equiv0$.  If moreover
  $\partial F_*/\partial(Z,W,D)(p)$ is invertible, then $C$ is the graph of an
  analytic map $t\mapsto\phi(t)$, and $p$ is an isolated point of $\mathcal X_m$ as soon as
  $q(t)=Q(t,\phi(t))\not\equiv0$, in particular if $q(0)=\Phi(p)\ne0$ or $q'(0)\ne0$.
  \item Suppose some twenty of the twenty-two polynomials have a differential of rank $20$ at
  $p$, so that near $p$ their common zeros form a smooth curve $C'$.  If one of the other two
  polynomials does not vanish identically on $C'$, then $p$ is an isolated point of $\mathcal X_m$.
\end{alphalist}
\end{lemma}

\begin{proof}
(a) By Proposition~\ref{prop:d2-cover}, a Diagram~2 germ with limit $p$ gives, after the
substitution $s=\tau^N$ that clears the denominators of the Puiseux exponents, a holomorphic germ
$\tau\mapsto(t,Z,W,D)$ in $\mathcal X_m$ at $p$ with $t\not\equiv0$.
(b) On $C$ all rows except $F^Z_2,F^Z_5$ vanish and $W_5=1$, so Lemma~\ref{lem:d2-syzygy} gives
$m_2W_2F^Z_2=tQ=-m_5F^Z_5$ on $C$; and $W_2(p)=m_1D_{12}(p)\ne0$.  This describes $\mathcal X_m$
near $p$.  A germ $c$ in $\mathcal X_m$ at $p$ lies in $C$, and if $t\circ c\not\equiv0$ then
$Q\circ c\equiv0$; when
$Q|_C\not\equiv0$ its zeros on the curve $C$ are isolated, so $c$ is constant, which contradicts
$t\circ c\not\equiv0$.  In the graph case $\mathcal X_m$ is, near $p$, the set of points $(t,\phi(t))$ with
$tq(t)=0$, which is $\{p\}$ when $q\not\equiv0$.
(c) Near $p$, $\mathcal X_m$ lies in $C'$ and in the zero set of the other polynomial, whose restriction
to the smooth curve $C'$ is analytic and not identically zero, hence has an isolated zero at $p$.
\end{proof}

{At the boundary points where the Jacobian of the twenty-two equations has rank
$19$, the implicit function theorem leaves two local parameters.  The following lemma
reduces exclusion to the Taylor coefficients of the three remaining equations.
Its proof follows the possible tangent directions of a formal arc by a Newton--Puiseux
argument after Lyapunov--Schmidt reduction.}

\begin{lemma}[Corank-two isolation]\label{lem:d2-corank-two}
Let $m\in\Rp^5$ and $p\in\mathcal B_m$.  Let $F_R$ be nineteen of the twenty-two chart polynomials,
$F_{R'}$ the other three, and $y\ne t$ a chart coordinate such that the derivative of
$(F_R,t,y)\colon\C^{21}\to\C^{21}$ at $p$ is invertible.  By the implicit function theorem,
near $p$ the set $\{F_R=0\}$ is the image of an injective analytic map $\phi$, defined near
$0\in\C^2$, with $\phi(0)=p$ and $(y,t)\circ\phi(a,t)=(y(p)+a,t)$, and near $p$
\[
  \mathcal X_m=\phi\bigl(g^{-1}(0)\bigr),\qquad g=F_{R'}\circ\phi\colon(\C^2,0)\to(\C^3,0).
\]
Let $g^{(j)}$ be the homogeneous part of degree $j$ of $g$, and for $\xi\in\C^2$ let
$\Pi_\xi=Dg^{(2)}(\xi)$, a $3\times2$ matrix.  Suppose:
\begin{alphalist}
  \item $g^{(1)}=0$; equivalently, the Jacobian at $p$ of all twenty-two polynomials has
  rank~$19$;
  \item the components of $g^{(2)}$ have only finitely many common zeros $[\xi_1],\dots,[\xi_r]$
  in $\mathbb P^1$, and $\Pi_{\xi_j}\ne0$ for every $j$;
  \item for every $j$ there is $N\ge3$ such that the vector $r_N(\xi_j)$ defined below exists
  and does not lie in $\Pi_{\xi_j}\C^2$.
\end{alphalist}
Then no nonconstant formal arc $x\in\C[[\tau]]^2$ with $x(0)=0$ satisfies $g(x)=0$.  Hence $p$
is an isolated point of $\mathcal X_m$; in particular $p$ is excluded.  The same holds if
\textup{(a)--(c)} hold for $Kg(L\,\cdot)$, $K\in GL_3(\C)$, $L\in GL_2(\C)$, in place
of $g$.

The vectors $r_N(\xi)$ are defined for a common zero $\xi$ of $g^{(2)}$ and a fixed
$e\in\C^2\setminus\C\xi$.  As $\Pi_\xi\xi=2g^{(2)}(\xi)=0$ and $\Pi_\xi\ne0$, the vector
$\varpi=\Pi_\xi e$ spans the line $\Pi_\xi\C^2$.  Put $\beta_1=0$.  For $n\ge1$, if $\beta_n$ is
defined, let $r_{n+2}(\xi)$ be the coefficient of $\lambda^{n+2}$ in
$g(\lambda\xi+\beta_n(\lambda)e)$; if moreover $r_{n+2}(\xi)=-b_{n+1}\varpi$ with $b_{n+1}\in\C$,
put $\beta_{n+1}(\lambda)=\beta_n(\lambda)+b_{n+1}\lambda^{n+1}$.  Thus
\[
  r_3(\xi)=g^{(3)}(\xi),\qquad
  r_4(\xi)=g^{(4)}(\xi)+b_2\,Dg^{(3)}(\xi)e+b_2^2\,g^{(2)}(e).
\]
Condition (c) with $N=3$ says that the $3\times3$ matrix $\bigl(\Pi_{\xi_j}\,|\,r_3(\xi_j)\bigr)$
has rank two; with $N=4$ it says that $\bigl(\Pi_{\xi_j}\,|\,r_3(\xi_j)\bigr)$ has rank one and
$\bigl(\Pi_{\xi_j}\,|\,r_4(\xi_j)\bigr)$ has rank two.
\end{lemma}

\begin{proof}
The points of $\{F_R=0\}$ near $p$ are the points $\phi(a,t)$, and $\phi(a,t)\in\mathcal X_m$ if
and only if $g(a,t)=0$.  If $p$ were not isolated in $\mathcal X_m$, the curve selection lemma
\cite[Thm.~2.5.5]{BCR}, applied to the real algebraic set $\mathcal X_m\subset\C^{21}=\R^{42}$,
would give a nonconstant semialgebraic arc $\gamma\colon[0,\varepsilon)\to\mathcal X_m$ with
$\gamma(0)=p$.  Its
coordinates are convergent Puiseux series, so $x=(y\circ\gamma-y(p),t\circ\gamma)$ is, after a
substitution $s=\tau^{N'}$, a nonconstant arc in $\C[[\tau]]^2$ with $g(x)=0$.  So the first
assertion implies the second.  The zero set of $Kg(L\,\cdot)$ is $L^{-1}g^{-1}(0)$, which gives
the last assertion.

Let $x$ be a nonconstant formal arc with $x(0)=0$ and $g(x)=0$, and write
$x=\tau^k\xi'+O(\tau^{k+1})$ with $\xi'\ne0$.  By (a),
$g(x)=\tau^{2k}g^{(2)}(\xi')+O(\tau^{2k+1})$, so $\xi'\in\C^*\xi$ for one of the $\xi=\xi_j$
of (b).  Write $x=u\xi+\tilde\theta e$; then $\ord u=k<\ord\tilde\theta$.  Take $\lambda=u(\tau)$ as
the new parameter; then $\tau$ is a power series in $\lambda^{1/k}$ with nonzero linear term,
and $x=\lambda\xi+\theta e$ with $\theta\in\C[[\lambda^{1/k}]]$ and $\ord\theta>1$ ($\theta=0$ is
allowed).

We show by induction on $n\ge1$ that $\beta_n$ is defined, that
$g(\lambda\xi+\beta_ne)=O(\lambda^{n+2})$ and that $o_n=\ord(\theta-\beta_n)>n$.  For $n=N-2$
this contradicts (c).  For $n=1$ the claim holds by (a), $g^{(2)}(\xi)=0$ and $\ord\theta>1$.
Taylor's formula in the direction $e$ gives
\[
  0=g(\lambda\xi+\theta e)
   =g(\lambda\xi+\beta_ne)+(\theta-\beta_n)\bigl(\lambda\varpi+O(\lambda^2)\bigr)
    +O(\lambda^{2o_n}),
\]
because $Dg(\lambda\xi+\beta_ne)e=\lambda\Pi_\xi e+\beta_n\Pi_ee+O(\lambda^2)
=\lambda\varpi+O(\lambda^2)$ by (a) and $\ord\beta_n\ge2$.  If $o_n<n+1$, the term
$(\theta-\beta_n)\lambda\varpi$ has order $o_n+1<n+2$, every other term has larger order
($2o_n>o_n+1$), and its leading coefficient is a nonzero multiple of $\varpi\ne0$; this is
impossible.  Hence $o_n\ge n+1$, and the coefficient of $\lambda^{n+2}$ gives
$r_{n+2}(\xi)+b\varpi=0$, where $b$ is the coefficient of $\lambda^{n+1}$ in $\theta$.  So
$\beta_{n+1}$ is defined, $b_{n+1}=b$ since $\varpi\ne0$, and $o_{n+1}>n+1$.  The same expansion
with $\beta_{n+1}$ in place of $\theta$ gives
$g(\lambda\xi+\beta_{n+1}e)=(r_{n+2}(\xi)+b_{n+1}\varpi)\lambda^{n+2}+O(\lambda^{n+3})
=O(\lambda^{n+3})$.
\end{proof}

\begin{remark}\label{rem:d2-corank-two}
The $\xi_j$ are the points where the strict transform of $g^{-1}(0)$ could meet the exceptional
curve of the blow-up of $0\in\C^2$, and the induction is the Newton--Puiseux algorithm in the
blow-up chart at $\xi_j$.
\end{remark}

{We next compute the first derivative of the reduced torque equation
using the triangle block and the quadratic form $Q_T$.
For a boundary point $p$, let $J_T(p)$ be the
Jacobian at $p$ of the nine equations
$F^Z_3,F^W_3,F^Z_4,F^W_4,F^W_5,T_{34},T_{35},T_{45},G_2$ with respect to the nine triangle
coordinates $Z_B,W_B,D_{BB'}$.  If $J_T(p)$ is invertible and $\xi\in\C^3$, let
$(\zeta,\eta,\vartheta)$ be the solution of $J_T(p)(\zeta,\eta,\vartheta)=r(\xi)$.
Here $r(\xi)$ has entries $\xi_3/W_3$ and $\xi_4/W_4$ in the rows $F^Z_3$, $F^Z_4$ and $0$
elsewhere.  Define}
\[
  Q_T(\xi)=-\tfrac12\textstyle\sum_Bm_B\,\xi_B\,\eta_B/W_B .
\]
{This quadratic form $Q_T$ depends only on $m'$ and the triangle of $p$.}

\begin{lemma}[Second-order coefficient]\label{lem:d2-second}
Let $p$ be a boundary point with $\Phi(p)=0$ and $J_T(p)$ invertible.  Then
$\partial F_*/\partial(Z,W,D)(p)$ is invertible and, in the notation of
Lemma~\ref{lem:d2-isolation} and of the crossed data,
\[
  q(0)=0,\qquad q'(0)=c_\sigma^2\,Q_T(\tilde v).
\]
\end{lemma}

\begin{proof}
{At $t=0$ the pair rows involve only $Z_A,W_A,D_{12}$, the triangle rows only the triangle
coordinates, and $T_{AB}$ involves no crossed coordinate other than $D_{AB}$, with
$\partial T_{AB}/\partial D_{AB}=2/D_{AB}$; hence the Jacobian is block lower triangular.
To check its pair block directly, write
\[
 d=D_{12},\qquad u=Z_2-Z_1,\qquad h=W_2-W_1.
\]
At the boundary point $u=1$, $Md^3=1$ and $h=Md$. The five selected pair rows are
$F^Z_1,F^W_1,F^W_2,T_{12},G_1$, and at $t=0$ they satisfy
\[
\begin{gathered}
 F^Z_1=Z_1+m_2d^3u,\qquad F^W_2-F^W_1=(1-Md^3)h,\\
 m_1F^W_1+m_2F^W_2=m_1W_1+m_2W_2,\qquad T_{12}=d^2uh-1.
\end{gathered}
\]
A vector in their differential kernel first has $\dot u=0$ from the gauge, then
$\dot d=0$ from the difference of the two $W$-rows, whose $\dot d$ coefficient is
$-3Md^2h\ne0$. The distance constraint gives $\dot h=0$, the weighted $W$-sum gives
$\dot W_1=\dot W_2=0$, and the $Z$-row and gauge give
$\dot Z_1=\dot Z_2=0$. Thus the pair block is invertible.
Differentiate $F_*(t,\phi(t))=0$ at $t=0$.
By Lemma~\ref{lem:d2-syzygy}, the pair $Z$-row, distance constraint and gauge have no explicit
$t$-forcing, whereas
\[
 \partial_tF^W_A=-Z_A^{-1}\sum_Bm_BD_{AB},\qquad
 \partial_tF^W_1-\partial_tF^W_2=\frac{M\Phi}{m_1m_2}=0.
\]
The same difference, distance constraint and $Z$-row calculation therefore gives
$D_{12}'(0)=Z_1'(0)=Z_2'(0)=0$.

}The triangle rows give
$J_T(p)(Z',W',D')_T(0)=r(\varphi)$, hence $W_B'(0)=c_\sigma\eta_B(\tilde v)$.  On $C$ the distance constraints
$T_{AB}=0$ give $Q=\sum_{A,B}m_Am_BD_{AB}+O(t^2)$ and
$D_{AB}'=-\frac12D_{AB}(Z_A'/Z_A+W_B'/W_B)$.  Therefore
$q'(0)=\sum_{A,B}m_Am_BD'_{AB}(0)=-\frac12\sum_Bm_B\varphi_BW_B'(0)/W_B=c_\sigma^2Q_T(\tilde v)$.
\end{proof}

\begin{definition}\label{def:d2-psi}
For a boundary point $p$ with $J_T(p)$ invertible, put $\Psi_T(p)=Q_T(\tilde v)$ in the equal-sign case and
$\Psi_1(p)=Q_T(\tilde v)$ in the opposite-sign case.  The boundary point is \emph{isolated at second order}
if $\Phi(p)=0$, $J_T(p)$ is invertible and $Q_T(\tilde v)\ne0$; by
Lemmas~\ref{lem:d2-isolation} and~\ref{lem:d2-second} it is then an isolated point of $\mathcal X_m$, and
the omitted rows $F^Z_2,F^Z_5$ vanish to order exactly two along the curve $C$.
\end{definition}

At a boundary point with $\Phi(p)=0$ the vector $\tilde v$ is determined by the triangle and the signs
$s_B$ (for the body with the opposite sign $k$, $m_k\tilde v_k=-\sum_{B\ne k}m_B\tilde v_B$), so $\Psi_T$ and
$\Psi_1$ do not involve $m_1,m_2$.  For an opposite-sign boundary point with $k=5$ the condition
$\Phi=0$ requires $|m_3\tilde v_3+m_4\tilde v_4|=m_5|v_5|$, i.e.\ the vanishing of the
function $G=|m_3\tilde v_3+m_4\tilde v_4|^2-m_5^2$ (recall $W_5=1$), which again
depends on the triangle only; then $\tilde v_5=\rho s_5v_5$ fixes $\rho$ up to sign, hence
$m_1/m_2$ up to exchanging $m_1$ and $m_2$.

\subsection{\texorpdfstring{Exclusion on the mass loci}{Exclusion on the mass loci}}\label{sec:d2-cases}

{We apply the local criteria to the mass loci determined by $\Phi=0$.
The relative signs of the crossed inverse distances separate the cases; when one
relative sign differs, we also distinguish $m_1=m_2$ from $m_1\ne m_2$.
Table~\ref{tab:d2-cases} displays the corresponding scalar obstructions and
computational methods. The computer-assisted statements below use exact arithmetic,
rational intervals, and IEEE ball arithmetic with proved rounding bounds
(Appendix~\ref{app:verification}).}

\begin{table}[ht]
\centering
\footnotesize
\setlength{\tabcolsep}{3pt}
\begin{tabular}{p{2.5cm}p{2.5cm}p{3.3cm}p{3.4cm}}
\hline
Boundary points & Mass locus & Method & Computation \\
\hline
$\Phi\ne0$ (any) & all $m$ & {$Q(p)=\Phi(p)\ne0$, $tQ=0$} & exact identities \\
equal-sign, $S_4$, or $S_5$ with $L_5'=0$ & $L_4=0$, $L_5'=0$ & no point in $\Rp^3$ & exact \\
equal-sign, $S_5$ & $L_5=0$: $m_3=m_4=m_5$ & order three; corank-two blow-up & exact jets \\
equal-sign, $S_1$--$S_3$, $S_6$--$S_8$ & rays of $L_1$--$L_3$, $L_6$--$L_8$ & second order,
{$\Psi_T\ne0$} & Bernstein subdivision, Krawczyk method; interval arithmetic \\
\hline
opposite-sign, $m_1=m_2$, $S_3,S_6$--$S_8$ & $K_3,K_6,K_7,K_8=0$ & no point in $\Rp^3$ &
Bernstein coefficients \\
opposite-sign, $m_1=m_2$, $S_1,S_2$ & $K_1=0$, $K_2=0$ & second order & exact; interval cover \\
opposite-sign, $m_1=m_2$, $S_4$ & {$K_4=0$: cones over four arcs} & second order & exact \\
opposite-sign, $m_1=m_2$, $S_5$ & $K_5=0$: $\Rp(1,1,\rho_*)$ & orders one to four &
exact; interval arithmetic \\
\hline
opposite-sign, $m_1\ne m_2$, not fold & $E_j(u)=O_j(u)=0$ & second order, $\Psi_1\ne0$: ball
covers, exact lemmas & ball and interval arithmetic; exact \\
opposite-sign, $m_1\ne m_2$, fold & $E_5(u)=O_5(u)=0$ & corank one; $Q|_C$ has a simple zero &
exact \\
\hline
\end{tabular}
\caption{The Diagram~2 cases.  A boundary point with $\Phi=0$ lies on $S_j$ for the $j$ given in the
first column (Lemma~\ref{lem:d2-fibre}); opposite-sign boundary points are taken with $k=5$.}
\label{tab:d2-cases}
\end{table}

\begin{lemma}[computer-assisted]\label{lem:d2-empty}
\textup{(a)} The polynomials $L_4$, $L_5'$, $K_3$, $K_6$, $K_7$ and $K_8$ have no zero in $\Rp^3$.
{More precisely, at positive masses every normalized $W$-shape on $S_4$ is real
\textup{(}the centre-of-mass section of $S_4$ is a cubic with positive discriminant\textup{)}
and satisfies $\Lambda\ne0$. Here the normalized shape is real in the sense
$(W_3,W_4,W_5)\in\R^3$.}
\textup{(b)} If $m_3=m_4=m_5$, no $S_j$ with $j\ne5$ has a root satisfying \textup{(25)}.
\end{lemma}
{The verification uses exact arithmetic over $\mathbb Q$.
For $L_4$, a rational parametrization of $S_4$ gives the obstruction:
at positive masses, a real normalized $W$-shape has positions $W_B$ of different signs.
Then $\Lambda$ is a sum of nonnegative terms, at least one of which is positive.
The coefficients of $L_5'$ are positive.  For $K_3$, $K_6$, $K_7$, $K_8$, nonnegative
Bernstein coefficients are checked on twelve projective charts, three for each
polynomial.  Each chart is split into the two sectors obtained by blowing up a corner
of the cone.  For (b), the verification shows that $S_j^\flat$ and $\Lambda^\flat$ are
coprime when $m_3=m_4=m_5$ and $j\ne5$ (Appendix~\ref{app:verification}).}

\begin{lemma}[computer-assisted]\label{lem:d2-L5}
Let $m\in\Rp^5$ with $m_3=m_4=m_5$.  Every boundary point at $m$ whose triangle lies on $S_5$ \textup{(}an
all-equal Lagrange class or a fold triangle\textup{)} and which satisfies $\Phi=0$ or is a fold
boundary point is an isolated point of $\mathcal X_m$.  At each such $m$ these are $72$ boundary points of the all-equal
classes and $1152$ fold points, $1224$ in all.
\end{lemma}

{We first normalize the masses.  Common mass scaling gives an exact
isomorphism of the gauged chart, so we may take $m_3=m_4=m_5=1$ and hence
$\delta_*:=M_T^{-1/3}=3^{-1/3}$.
Every positive pair $(m_1,m_2)$ has a unique parametrization by
$(\nu,M)\in(0,1)\times\Rp$ of the form}
\[
  m_1=\frac{4M\nu^2}{(1+\nu^2)^2},\qquad m_2=\frac{M(1-\nu^2)^2}{(1+\nu^2)^2},
\]
{Indeed, $M=m_1+m_2$ and $\nu=1/(\chi+\sqrt{1+\chi^2})$, where
$\chi=\sqrt{m_2/m_1}=(1-\nu^2)/(2\nu)$.}

{Consider first the $72$ boundary points in the all-equal classes.
They satisfy $\Phi=0$, have invertible $J_T(p)$, and satisfy $q(0)=q'(0)=0$.
The next coefficient is obtained from $F^Z_5=\kappa_3t^3+O(t^4)$ along $C$, which gives
$q''(0)=-2\kappa_3$.
For the nine boundary points of one sign class with $D_{BB'}=\omega^k\delta_*$,
$(W_3,W_4,W_5)=(\omega,\omega^2,1)$ and $D_{12}=\omega^jM^{-1/3}$, the exact computation gives}
\[
  \kappa_3=\omega^jM^{5/3}R_1(1-i\chi)+3\omega^k\delta_*MR_2(1-i\chi^5),\qquad
  R_1=\frac{18\nu^2(1-\nu^2)}{(1+\nu^2)^3},\quad R_2=\frac{6\nu^2(1+\nu^2)}{(1-\nu^2)^3}.
\]
{To obtain these nine coefficient identities from the case $j=k=0$, first regard
$\nu,b$ as algebraically independent indeterminates, set $M=b^3$, and work over
$\mathbb Q(\omega,i,\delta_*)(\nu,b)$, where $\delta_*^3=1/3$.
The commuting automorphisms
\[
  \delta_*\mapsto\omega^k\delta_*,\qquad b\mapsto\omega^{-j}b
\]
fix $\nu$, $\omega$, $i$ and $b^3$, hence the masses and the chart polynomials.
They carry the representative boundary point and its formal implicit solution to those with
$D_{BB'}=\omega^k\delta_*$ and $D_{12}=\omega^jb^{-1}$.
In particular, $b^5\mapsto\omega^jb^5$, giving the displayed identity for $\kappa_3$.

We now specialize $0<\nu<1$ and $b=M^{1/3}>0$.
The boundary point coordinates are regular throughout this range, and the selected-coordinate
Jacobian is invertible: its triangle block is $J_T$, its pair block is invertible by
Lemma~\ref{lem:d2-second}, and its crossed block has diagonal entries $2/D_{AB}\ne0$.
At each Taylor order, the implicit equations determine the next coefficient by applying
the inverse of this same Jacobian to expressions in the preceding coefficients.
Thus the Taylor coefficients are regular at every such specialization.
Clearing the nonvanishing denominators in the formal identities and then specializing
proves the formula at all positive masses.
The remaining $63$ boundary points are obtained by exchanging bodies $3$ and $4$, conjugating,
and reversing $t$.

It remains to see that the two terms in the displayed coefficient cannot cancel.
We have $R_1,R_2>0$, and $(1-i\chi)/(1-i\chi^5)$ is a positive multiple of
$1+\chi^6+i(\chi^5-\chi)$.
Moreover, $|\chi^5-\chi|\le\chi+\chi^5\le1+\chi^6$, since
$(1+\chi^6)-(\chi+\chi^5)=(1-\chi)(1-\chi^5)\ge0$.
The quotient of the two terms of $\kappa_3$ is therefore $\omega^{j-k}$ times a number
whose argument has absolute value at most $\pi/4$, and cannot equal $-1$.
Thus $\kappa_3\ne0$, and Lemma~\ref{lem:d2-isolation}(b) gives isolatedness.}

{We treat the fold points of rank $20$ first.
Among the fold points, $1080$ have a Jacobian of rank $20$ for the
twenty-two polynomials.  At each such boundary point, twenty equations together with
$h=W_4-W_4(p)$ have an invertible derivative.
For $864$ boundary points these equations are $F_*$; for the remaining $216$ they are
all equations other than $F^Z_2$ and $F^Z_3$.
Along the resulting local curve, an omitted equation has a nonzero coefficient of
$h^3$.  Lemma~\ref{lem:d2-isolation}(c) therefore gives isolatedness for these boundary points.}

{The remaining $72$ fold points have rank $19$ and form one orbit
under the symmetries described below.
We choose a representative $p_*=p_*(\nu,M)$ whose pair coordinates are as in
Lemma~\ref{lem:d2-fibre}(a), with $D_{12}=M^{-1/3}$, and whose fold triangle is}
\[
  (W_3,W_4,W_5)=(\omega,\omega^2,1),\qquad Z_B=-\delta_*W_B,\qquad
  (D_{34},D_{35},D_{45})=\delta_*(1,\omega,\omega^2).
\]
{The six crossed inverse distances of this representative are
\[
\begin{aligned}
 (D_{13},D_{14},D_{15})&=-\frac{1+\nu^2}{1-\nu^2}(\omega,\omega^2,1),\\
 (D_{23},D_{24},D_{25})&=\frac{i(1+\nu^2)}{2\nu}(\omega,\omega^2,1).
\end{aligned}
\]
This sign choice fixes the coordinate $t$ relative to the Taylor coefficients below;
simultaneously reversing all crossed signs would reverse this coordinate.}
{For this representative, we use Lemma~\ref{lem:d2-corank-two} with
$F_{R'}=(F^W_2,F^W_4,F^W_5)$, with $F_R$ consisting of the other nineteen equations, and
with $y=Z_3$.
The two kernel coordinates are the chart parameter $t$ and $a=Z_3-Z_3(p_*)$.
At $t=0$, the derivative of $(F_R,t,Z_3)$ is block lower triangular, as in the proof of
Lemma~\ref{lem:d2-second}.
Its diagonal blocks are $(1)$ for $dt$; the triangle block}
(rows $F^Z_3,F^W_3,F^Z_4,F^Z_5,T_{34},T_{35},T_{45},G_2,dZ_3$, columns $Z_B,W_B,D_{BB'}$), with
determinant $18\delta_*$; the pair block (rows $F^Z_1,F^W_1,F^Z_2,T_{12},G_1$, columns
$Z_A,W_A,D_{12}$), with determinant $-3M^{-1/3}$; and the diagonal crossed block $(2/D_{AB})$,
with determinant $-512i\,\nu^3(1-\nu^2)^3/(1+\nu^2)^6$.
{All these determinants are nonzero, so the required invertibility
holds for every $(\nu,M)$.
The differentials of the three omitted equations vanish on the two-dimensional kernel
of $dF_R(p_*)$, establishing (a).
We write the reduced equations as $g=(\psi_2,\psi_4,\psi_5)$, where
$\psi_k=F^W_k\circ\phi$.}

At $M=1$, exact arithmetic over $\mathbb Q(\nu)(\sqrt3,i,\delta_*)$ gives the
jet of $g$ through degree four. Substitution of the corresponding jet of $\phi$
in the nineteen rows of $F_R$ gives $O((|a|+|t|)^5)$.  In the quadratic part, $\psi_2^{(2)}=0$, there is no $a^2$ term,
and the $3\times3$ coefficient matrix has rank one.  So $g^{(2)}=\psi_4^{(2)}\,n$ for a fixed
vector $n=(0,1,n_5)$, where $\psi_4^{(2)}=t(\gamma_{11}a+\gamma_{02}t)$ and
\[
  \gamma_{11}=\frac{9i\delta_*(\nu-i)(\nu+i)^3\,\Upsilon(\nu)}{4\nu(\nu^2-1)(\nu^2+1)^2},\qquad
  \Upsilon(\nu)=\nu^4-2i\nu^3-6\nu^2+2i\nu+1=\tfrac14\bigl(3(\nu-i)^4+(\nu+i)^4\bigr).
\]
For real $\nu$, $|3(\nu-i)^4|=3(\nu^2+1)^2>|(\nu+i)^4|$, so $\Upsilon(\nu)\ne0$ and
$\gamma_{11}\ne0$ for $0<\nu<1$.  Hence the common zeros of $g^{(2)}$ are two lines: the
\emph{boundary line} $\xi_\partial=(1,0)$, tangent to $\{t=0\}$, and the line transverse to $\{t=0\}$
$\xi_s=(-\gamma_{02},\gamma_{11})$.  Moreover $\Pi_\xi=n\,d\psi_4^{(2)}(\xi)$, with
$d\psi_4^{(2)}(\xi_\partial)=(0,\gamma_{11})$ and
$d\psi_4^{(2)}(\xi_s)=\gamma_{11}(\gamma_{11},\gamma_{02})$, both nonzero.  So (b) holds.
On the transverse line take $N=3$: the matrix $(\Pi_{\xi_s}\,|\,r_3(\xi_s))$ has the $2\times2$
minor
\[
  \mu_3=-\frac{3^{10}(\nu-i)^4(\nu+i)^{16}\,\Upsilon(\nu)^5}
             {2^{12}\delta_*\,\nu^6(\nu^2-1)^6(\nu^2+1)^8}\ne0 .
\]
On the boundary line take $e=(0,1)$, so that $\varpi=\gamma_{11}n$, and $N=4$.  Identically in
$\nu$, $r_3(\xi_\partial)=g^{(3)}(1,0)=-b_2\varpi$, and the matrix
$(\Pi_{\xi_\partial}\,|\,r_4(\xi_\partial))$ has the minor
\[
  \mu_4=-\frac{27i\delta_*(\nu-i)^3(\nu+i)^5\,\Upsilon(\nu)}{64\,\nu^3(\nu^2-1)^3}\ne0 .
\]
So (c) holds.  Every coefficient used is regular for $0<\nu<1$: the jets are polynomials in the
coordinates of $p_*$, the masses and the entries of the inverses of the three diagonal blocks,
and the only further division is by $\gamma_{11}$.  Hence the identities above hold at every
$\nu\in(0,1)$.

For general $M$, $p_*(\nu,M)$ is obtained from $p_*(\nu,1)$ by multiplying $W_1,W_2$ by $M^{2/3}$
and $D_{12}$ by $M^{-1/3}$.  Exact Laurent arithmetic in $M^{1/3}$ gives the reduced map
$\tilde g(a,\tilde t)=\bigl(M^{-2/3}\psi_2,\psi_4,\psi_5\bigr)(a,\tilde t/M)$;
the selected rows vanish along both lines through the orders used.  The quadratic part of $\tilde g$ does not depend on $M$, the compatibility
$r_3(\xi_\partial)\in\C\varpi$ persists, and the two minors become $M^{-3}\mu_3(\nu)$ and
$M^{-5/3}\mu_4(\nu)$.  By Lemma~\ref{lem:d2-corank-two}, applied to $\tilde g$, $p_*(\nu,M)$ is
an isolated point of $\mathcal X_m$ for every $0<\nu<1$ and $M>0$.

{To transfer isolatedness from this representative to the other boundary points,
we use the symmetries of the chart equations.
Because the twenty-two polynomials have real coefficients and $m_3=m_4$, they are
preserved by two commuting involutions: the biholomorphic reversal
$(t,D_{AB})\mapsto(-t,-D_{AB})$ and the antiholomorphic map obtained by conjugating and
exchanging bodies $3$ and $4$.
Both fix $(\nu,M)$ and together carry $p_*$ to the other three sign choices in its orbit.
Multiplying $t$, $Z_B$, $W_A$, $D_{12}$ and $D_{BB'}$ by $\omega^j$ multiplies each
polynomial by a cube root of unity and preserves the gauges.
A permutation of bodies $3,4,5$, followed by the chart-torus action restoring the
gauges, gives a local biholomorphism of $\mathcal X_m$.
The implicit function theorem applies because the transversality determinant is
$2(Z_2-Z_1)W_5=2$.
These $3\cdot6$ maps produce eighteen boundary points for each of the four sign choices.
The resulting $72$ boundary points form precisely one orbit under these discrete symmetries and
are disjoint from the $1080$ boundary points already treated.
Every map used here preserves isolatedness.}

The computations in the proof of Lemma~\ref{lem:d2-L5}, namely the coefficients $\kappa_3$, the
ranks and the coefficients of $h^3$ at the $1080$ fold points, the determinants, jets and minors
at $p_*$, and the orbit of $p_*$, are exact, over number fields containing $\omega$, $i$ and
$\delta_*$ and over rational function fields over them in $\nu$ and $M^{1/3}$
(Appendix~\ref{app:verification}).

\begin{lemma}[computer-assisted]\label{lem:d2-rays}
\textup{(a)} The zero set of $L_8$ in $\Rp^3$ is $\Rp(p_*,q_*,1)\cup\Rp(q_*,p_*,1)$ with
$p_*=0.67987758\ldots$ and $q_*=0.82343785\ldots$.
\textup{(b)} The zero set of $L_1$ in $\Rp^3$ is $\Rp(a_*,b_*,1)\cup\Rp(a_*,1,b_*)$ with
$a_*=1.02227400\ldots$ and $b_*=0.90559635\ldots$; at $(a_*,b_*,1)$ the zeros of $S_1$ that
satisfy \textup{(25)} form exactly one pair of complex conjugate non-real shapes.
\textup{(c)} Consequently the zero sets of $L_1L_2L_3$ and of $L_6L_7L_8$ in $\Rp^3$ are the rays
through the six permutations of $(a_*,b_*,1)$ and of $(1,p_*,q_*)$, respectively.
\end{lemma}
{For (a), exact Bernstein subdivision of the mass simplex locates the
possible zeros, and the Krawczyk method in rational interval arithmetic encloses them.
For (b), subresultants eliminate the shape variable.
Bernstein subdivision and the signs of resultants locate the mass zeros, which are
again enclosed by the Krawczyk method.  The quadratic subresultant has negative
discriminant in these enclosures, giving the conjugate non-real shapes
(Appendix~\ref{app:verification}).
Part (c) follows from (a), (b) and the permutation symmetry following
Definition~\ref{def:d2-mass}.}

\begin{lemma}[computer-assisted]\label{lem:d2-Lray}
Let $m\in\Rp^5$ with $m'$ on one of the rays of Lemma~\ref{lem:d2-rays}(c), and let $p$ be a
equal-sign boundary point at $m$ with $\Phi(p)=0$.  Then $J_T(p)$ is invertible and $\Psi_T(p)\ne0$, so $p$
is isolated at second order.
\end{lemma}

{The enumeration includes all $48$ Lagrange triangles at each
permutation of $(1,p_*,q_*)$ and the six Euler triangles on the conjugate shape pair at
each permutation of $(a_*,b_*,1)$.
At these mass vectors, only the factor defining the ray vanishes among
$L_1,\dots,L_8,L_5'$.  Proposition~\ref{prop:d2-first}(a) therefore excludes any other
triangle from carrying an equal-sign boundary point with $\Phi=0$.
For $36+36$ triangles, there is a sign vector satisfying $\sum_Bm_Bs_Bv_B=0$, unique up
to a simultaneous sign change.  Their $J_T$ are invertible, and complex interval
arithmetic gives $|\Psi_T|>0.36$ ($L_6$--$L_8$) and $|\Psi_T|>0.3$ ($L_1$--$L_3$) at the
normalized representatives.

Common mass scaling preserves isolatedness and $\Psi_T\ne0$, extending these
conclusions to the entire rays. Moreover, $\Psi_T$ depends on $m'$ and the triangle,
so the same conclusion holds for every positive pair $(m_1,m_2)$.
The Krawczyk method encloses the shapes and the square and cube roots; the solutions of
the linear systems are also enclosed rigorously (Appendix~\ref{app:verification}).}

\begin{lemma}[computer-assisted]\label{lem:d2-K}
Let $m\in\Rp^5$ with $m_1=m_2$, let $p$ be an opposite-sign boundary point at $m$ with $k=5$ and
$\Phi(p)=0$, and let $S_j$ carry its triangle, so that $K_j(m')=0$.  Then $p$ is an isolated
point of $\mathcal X_m$.  More precisely:
\begin{alphalist}
  \item $j\notin\{3,6,7,8\}$, by Lemma~\ref{lem:d2-empty};
  \item if $j\in\{1,2\}$: at every zero of $K_1$ in $\Rp^3$ the $S_1$ shapes compatible with
  \textup{(28)} at $y=0$ reduce to one simple common root, which is real and lies on one of two
  explicit real branches; $J_T(p)$ is invertible and $q'(0)\ne0$.  The case $j=2$ is the image of
  $j=1$ under the exchange of bodies $3$ and $4$;
  \item {if $j=4$: $m'$ lies in one of the positive cones over the four arcs
  $k\mapsto\mathrm{II}(k),\mathrm{III}(k)$, $5^{-1/4}<k<1$, and their images under
  $m_3\leftrightarrow m_4$.  Thus $m'=c\,\mathrm{II}(k)$ or $c\,\mathrm{III}(k)$, or an
  exchanged vector, for some $c>0$.  Here}
  $\alpha=4k(1+k^4)$, $\beta=(1-k^2)(1+6k^2+k^4)$, $\gamma=(1+k^2)(6k^2-1-k^4)$,
  $\mathrm{II}=((\gamma^3-\beta^3)\alpha,(\alpha^3-\gamma^3)\beta,(\alpha^3-\beta^3)\gamma)$ and
  $\mathrm{III}=((\gamma^3-\beta^3)\alpha,(\gamma^3+\alpha^3)\beta,(\alpha^3+\beta^3)\gamma)$;
  $J_T(p)$ is invertible and $q'(0)\ne0$;
  \item if $j=5$: $m'\in\Rp(1,1,\rho_*)$, where $\rho_*\in(1.5736409,1.573641)$ is the unique
  positive root of $\rho^8+2\rho^7+3\rho^6-4\rho^5-5\rho^4-6\rho^3-3\rho^2-12\rho-12$, and every
  boundary point at $m$ with a Lagrange triangle is isolated, for all $27$ classes and all $64$
  crossed-sign assignments.
\end{alphalist}
\end{lemma}

{For (b) and (c), $60$ of the $64$ crossed-sign assignments satisfy
$\Phi\ne0$.  The other four form one orbit under conjugation and $t\mapsto-t$ and have
$q'(0)\ne0$.
In (d), the two-equal classes have $\Phi\ne0$.
Among the all-equal classes, $60$ assignments have $\Phi\ne0$ and $4$ have $q'(0)\ne0$.
For the fold classes, an augmented system using $h=W_4-W_4(p)$ gives a nonzero coefficient
of order three for $60$ assignments and of order four for the remaining $4$.
The exact grading makes these computations uniform in the common pair mass $m_1=m_2$.}

{In (b), Bernstein coefficients establish smoothness of the zero set
of $K_1$ in $\Rp^3$.  A linear subresultant over $\mathbb Q$ then determines the common
shape root.  A rational parametrization of $S_1$, together with Sturm sequences,
identifies its two real branches.  An exact reduction and an interval cover of the
branch parameter give $q'(0)\ne0$.

For (c), exact computation on the rational parametrizations of the four projective mass
arcs classifies the real normalized $W$-shapes.  Lemma~\ref{lem:d2-empty}(a) shows that
these include every $S_4$ shape at positive masses.
Common mass scaling, followed by restoration of the chart gauges, is a local analytic
isomorphism.  Isolatedness therefore extends from the arcs to their positive cones.

For (d), exact resultants and a Sturm count determine the zero set of $K_5$ in $\Rp^3$.
Complex rational interval arithmetic classifies the boundary points, using one triangle from
each orbit under conjugation and exchange of bodies $3$ and $4$
(Appendix~\ref{app:verification}).
Renumbering the bodies as at the start of this section gives the same conclusions for
every choice of the pair and of the body with the opposite sign.}

\begin{lemma}[computer-assisted]\label{lem:d2-P}
Let $m\in\Rp^5$ with $m_1\ne m_2$, and let $p$ be an opposite-sign boundary point at $m$ with $k=5$,
$\Phi(p)=0$, whose triangle is not a fold triangle.  Then $J_T(p)$ is invertible and
$\Psi_1(p)\ne0$, so $p$ is isolated at second order.
\end{lemma}

{The verification covers the mass simplex $m_3+m_4+m_5=1$.
For $S_1,S_3,S_4$, it includes every Euler shape, enclosed by the Krawczyk method or
Pellet discs.  For $S_5,S_6,S_8$, it includes every Lagrange shape up to conjugation and
the cube-root action.  Both relative signs $s_3s_4$ are considered.
On each box, either $G$ has a strict sign or ball arithmetic encloses $\Psi_1$ away from
$0$.  On boxes meeting $m_3=0$ or $m_4=0$, where $\Psi_1$ vanishes, the corresponding
transversal derivative is instead enclosed away from zero.}

{The remaining boundary and degeneracy cases require separate arguments.
At the vertices, $G\ne0$.
At the ends $(m_3,m_4)=(0,\frac12)$ for $S_1$ and $S_4$, and $(\frac12,0)$ for $S_4$, a
majority body becomes massless.  An exact edge identity gives interval margins
$0.787$, $0.588$, $0.551$.
At the corner $(\frac12,\frac12,0)$ for $S_3$, an exact weighted blow-up gives weight
two in $m_5$ to the Euler shape parameter.  The resulting determinant and response
numerator are $m_5^{13}$ and $m_5^6$ times factors whose moduli are bounded below by
$6.96\cdot10^8$ and $1.63\cdot10^6$, respectively, in interval arithmetic.
For the six real points of $S_4$ where the Euler boundary point degenerates, and for the corner
$(\frac12,\frac12,0)$ of $S_4$, an exact formula for $\Psi_1$ on the Euler shapes cancels
the degenerating factor.
The remaining factor $S_5$ is treated by exact algebra (Appendix~\ref{app:verification}).}

{On $S_5$, let $g$ be the $W_4$-component of the solution of the
triangle linear system with row $F^Z_5$ omitted and $dW_5=0$, when the right-hand side is
one in row $F^Z_4$ and zero elsewhere.  Put $X=\tilde v_4/W_4-\tilde v_3/W_3$.
Then $\Psi_1=-\frac12m_4gX^2$.
The numerator of $g$ is a nonzero constant times
$m_3(m_3^2+m_3m_4+m_4^2)^2(P_0+\omega P_1)$, where the real polynomials $P_0,P_1$ have
only one common zero in the simplex, at $m_3=m_4=\frac13$.
Also, $X\ne0$ except at this point.
There $m_3=m_4=m_5$, and a boundary point with $G=0$ would have $\rho=-1$, which is impossible.
The cases $S_2$ and $S_7$ follow from $S_1$ and $S_6$ by exchanging bodies $3$ and $4$.}

\begin{lemma}[computer-assisted]\label{lem:d2-fold}
Let $m\in\Rp^5$ and let $p$ be an opposite-sign boundary point at $m$ with $k=5$, $\Phi(p)=0$ and a
fold triangle.  Then $dF_*(p)$ has rank $20$ and $Q|_C$ has a simple zero at $p$; in
particular $p$ is excluded.
\end{lemma}

{The triangle Jacobian $J_T(p)$ has corank exactly one.
Its kernel vector and left null vector $\ell$ are computed explicitly.
The $t$-column has entries $-\varphi_B/W_B$ in the triangle $Z$-rows.
Applying $\ell$ to this column gives $c_\sigma X$, with $X$ as in
Lemma~\ref{lem:d2-P}.
Here $X\ne0$ at every boundary point under consideration: the exceptional point $m_3=m_4=m_5$
has no opposite-sign boundary point with $\Phi=0$ at positive masses.
Thus the additional parameter resolves the corank, and the curve $C$ exists with local
parameter $h=W_4-W_4(p)$.
At $p$, it satisfies $dt/dh=0$ and $dQ/dh=-\frac12c_\sigma m_4X\ne0$, which gives the
required simple zero.

These statements are verified exactly over $\mathbb Q(m_3,m_4)(\omega)$, treating
$m_3,m_4$ as indeterminates.
The computation gives the kernel and left null vectors, an order-eight minor of
$J_T(p)$ that is nonzero throughout the simplex, and the zero set of $X$
(Appendix~\ref{app:verification}).}

\begin{theorem}\label{thm:d2}
{For $m\in\Rp^5$, no curve in $A_m$ has a Diagram~2 germ.
In particular, Theorem~\ref{thm:exclusion}\textup{(a)} holds for $r=2$.}
\end{theorem}

\begin{proof}
{Suppose that $\gamma$ is a Diagram~2 germ on a curve in $A_m$, and
let $p$ be its boundary limit given by Proposition~\ref{prop:d2-cover}.
For each $B$, the quantity $\sigma_Bx$ is the limit of
$m_2\delta_{2B}/(m_1\delta_{1B})$ along $\gamma$.
This limit is independent of the chart, so the crossed-sign pattern of $p$ is
determined by $\gamma$.
Renumbering bodies $3,4,5$ before applying Proposition~\ref{prop:d2-cover} permutes the
pattern. We may therefore take $k=5$ in the opposite-sign case.

By Lemma~\ref{lem:d2-isolation}(a), it suffices to exclude $p$.
Proposition~\ref{prop:d2-first} gives $\Phi(p)=0$.
Choose a shape factor $S_j$ vanishing on the triangle of $p$, as in
Lemma~\ref{lem:d2-fibre}.}

{Suppose first that the relative signs are equal.
Proposition~\ref{prop:d2-first}(a) gives $L_j(m')=0$, or $L_5L_5'=0$ if $j=5$.
Lemma~\ref{lem:d2-empty} rules out $j=4$ and $L_5'=0$.
If $L_5=0$, then $m_3=m_4=m_5$, and Lemma~\ref{lem:d2-empty}(b) gives $j=5$;
Lemma~\ref{lem:d2-L5} excludes these boundary points.
In all other cases, $j\in\{1,2,3,6,7,8\}$ and $m'$ lies on one of the rays in
Lemma~\ref{lem:d2-rays}(c), so Lemma~\ref{lem:d2-Lray} applies.}

{For an opposite-sign boundary point with $m_1=m_2$, exclusion follows from
Lemma~\ref{lem:d2-K}.
When $m_1\ne m_2$, Lemma~\ref{lem:d2-fold} applies to fold triangles and
Lemma~\ref{lem:d2-P} to all other triangles.}

{Thus every possible $p$ is excluded; for those isolated at second
order this follows from Definition~\ref{def:d2-psi}.
This contradicts Lemma~\ref{lem:d2-isolation}(a).}
\end{proof}

\section{Diagram 3}\label{sec:d3}

We classify Diagram~3 germs by reducing the chart equations to two analytic
functions. Their Taylor coefficients determine the masses and the possible germs.
The surviving analytic families satisfy an inverse-distance identity on the whole
algebraic curve, which excludes every curve containing a positive point.

{Diagram~3 of \cite[Fig.~11]{AK} occurs at infinity on Roberts' continuum
\cite[\S7.2]{AK}. We use the numbering of \cite[\S5.2 and \S8.3]{AK}, in which the clusters are
$z:12\ldots34$, $w:14\ldots23$. The diagram is the quadrilateral $1234$, with $z$-edges $14$, $23$ and
$w$-edges $12$, $34$; neither the diagonals nor body $5$ carry a stroke. Exchanging $z$ and $w$
has the same effect as exchanging bodies $2$ and $4$. Thus, after renumbering, every Diagram-3 germ
has the form considered below.}

\subsection{The crossed-cluster chart}
{Apart from the gauge, the equations in \eqref{eq:system} are invariant under
$(z,w,\delta)\mapsto(\mu z,\mu^{-1}w,\delta)$, $\mu\in\C^*$. We therefore impose
$z_{13}=w_{13}$, following \cite[p.~543]{AK} for Roberts' continuum. To resolve the two pairs
of clusters, introduce a parameter $t$ and coordinates $A,B,x,y,\hat z_5,C,D,u,v,\hat w_5$ and $D_{kl}$
by the following substitutions:}
\begin{equation*}
\begin{gathered}
z_1=\tfrac{A}{t},\quad z_2=\tfrac{A}{t}+tx,\quad z_3=\tfrac{B}{t},\quad z_4=\tfrac{B}{t}+ty,
\quad z_5=t^2\hat z_5,\\
w_1=\tfrac{C}{t},\quad w_2=\tfrac{D}{t},\quad w_3=\tfrac{D}{t}+tv,\quad w_4=\tfrac{C}{t}+tu,
\quad w_5=t^2\hat w_5,\\
\delta_{kl}=D_{kl}\ \ (kl=12,14,23,34),\qquad
\delta_{kl}=t\,D_{kl}\ \ (kl=13,24,15,25,35,45).
\end{gathered}
\end{equation*}
{Multiply the $z$- and $w$-equations for bodies $1,\dots,4$, together with the gauge
equation, by $t$, and those for body $5$ by $t^{-2}$. The resulting $21$ equations are polynomial:
$F=F^0+t^2F^2+\dots+t^6F^6=0$.
Setting $t=0$ gives the \emph{boundary equations} $F^0=0$, whose left-hand sides are}
\begin{equation}\label{eq:d3-boundary}
\begin{gathered}
R^z_1=A-m_4D_{14}^3(A-B),\qquad R^z_4=B-m_1D_{14}^3(B-A),\\
R^z_2=A-m_3D_{23}^3(A-B),\qquad R^z_3=B-m_2D_{23}^3(B-A),\\
R^w_1=C-m_2D_{12}^3(C-D),\qquad R^w_2=D-m_1D_{12}^3(D-C),\\
R^w_3=D-m_4D_{34}^3(D-C),\qquad R^w_4=C-m_3D_{34}^3(C-D),\\
\hat z_5+(m_1D_{15}^3+m_2D_{25}^3)A+(m_3D_{35}^3+m_4D_{45}^3)B,\\
\hat w_5+(m_1D_{15}^3+m_4D_{45}^3)C+(m_2D_{25}^3+m_3D_{35}^3)D,\\
D_{12}^2x(D-C)-1,\quad D_{34}^2y(C-D)-1,\quad D_{14}^2u(B-A)-1,\\
D_{23}^2v(B-A)-1,\quad D_{13}^2(B-A)(D-C)-1,\quad D_{24}^2(B-A)(C-D)-1,\\
D_{15}^2AC-1,\quad D_{25}^2AD-1,\quad D_{35}^2BD-1,\quad D_{45}^2BC-1,\quad (B-A)-(D-C);
\end{gathered}
\end{equation}
{Here $R^z_k$, $R^w_k$ are the boundary equations for body $k$.
We restrict to the slice $B=A+1$. Proposition~\ref{prop:d3-cover} justifies this choice
by showing that every germ can be represented with $t=1/z_{13}$.}

\subsection{Representation of germs in the chart}
{We first verify that this chart contains every germ with the prescribed diagram.
We write $s$ for the parameter of a germ.}

\begin{proposition}[Representation of germs]\label{prop:d3-cover}
Let $\gamma$ be a Diagram-3 germ, renormalized as in Lemma~\ref{lem:germ}.
\begin{romanlist}
  \item Bodies $1,\dots,4$ are circled in both colours, body $5$ in neither.
  \item The sides are maximal edges: $r_{12},r_{14},r_{23},r_{34}\approx1$,
  $z_{14},z_{23},w_{12},w_{34}\approx\epsilon^{-2}$ and $w_{14},w_{23},z_{12},z_{34}\approx\epsilon^2$.
  \item $z_{13},w_{13},z_{24},w_{24},z_{k5},w_{k5}\approx\epsilon^{-2}$ for $k\le4$, and
  $z_5,w_5\preccurlyeq\epsilon^4$.
  \item Renormalize instead by $\lambda=(w_{13}/z_{13})^{1/2}$, so that $z_{13}=w_{13}$, and put
  $t=1/z_{13}\approx\epsilon^2$.  Then $B-A=1$ and every chart coordinate converges as $s\to0$; the
  limit $X_0$ satisfies \eqref{eq:d3-boundary}, and $A,B,C,D,x,y,u,v$ and all $D_{kl}$ are
  nonzero at $X_0$.
\end{romanlist}
\end{proposition}

\begin{proof}
{Lemma~\ref{lem:germ} allows us to apply the order relations of \cite[\S3]{AK}
along $\gamma$, and hence the rules in Fact~\ref{fact:rules}.

For (i), each side is the only stroke of its colour incident to either endpoint.
Rule~1a \cite[pp.~545--546]{AK} therefore circles $1,\dots,4$ in both colours.
Since every circle requires a stroke of the same colour, $5$ has no circle.

For (ii), the isolated component $\{1,4\}$ of the $z$-diagram has $z$-circles.
Rule~1e \cite[p.~546]{AK} makes its stroke $14$ maximal. The same argument applies to $23$,
and to $12$, $34$ in the $w$-diagram. The right equality case of Estimate~2
\cite[p.~548]{AK} gives the asserted orders.

For (iii), (ii) gives $z_{13}=z_{14}-z_{34}$, $w_{13}=w_{12}+w_{23}$, $z_{24}=z_{23}+z_{34}$ and
$w_{24}=w_{14}-w_{12}$, all of order $\approx\epsilon^{-2}$.
Part (i) also gives $z_k\approx Z=\epsilon^{-2}\succ z_5$ for $k\le4$, and the corresponding
relations for $w$. In the equation
$z_5(1-\sum_{k\le4}m_k\delta_{k5}^3)=-\sum_{k\le4}m_k\delta_{k5}^3z_k$ for body $5$,
we have $\delta_{k5}\approx\epsilon^2$. It follows that $z_5\preccurlyeq\epsilon^4$.
The same argument applies to $w_5$.

For (iv), (iii) shows that $\lambda$ differs from the normalization factor in
Lemma~\ref{lem:germ} by a Puiseux series of order $0$. Thus the orders in (i)--(iii) remain
unchanged. We now have $B-A=tz_{13}=1$, and the coordinates $A=z_1/z_{13}$, $B=z_3/z_{13}$,
$C=w_1/w_{13}$, $D=w_2/w_{13}$, $x=z_{12}z_{13}$, $y=z_{34}z_{13}$, $u=w_{14}w_{13}$,
$v=w_{23}w_{13}$ and $D_{kl}$ ($=\delta_{kl}$ for the sides, $\delta_{kl}z_{13}$ otherwise) are
$\approx1$. The remaining coordinates satisfy $\hat z_5=z_5z_{13}^2$,
$\hat w_5=w_5w_{13}^2\preccurlyeq1$. Every chart coordinate therefore has a finite limit,
and the limiting point satisfies $F=0$.}
\end{proof}

By Proposition~\ref{prop:d3-cover}(ii) and~(iii), every Diagram-3 germ has an unbounded
four-product: $p_{12}=r_{12}^2r_{34}^2r_{35}^2r_{45}^2\approx\epsilon^{-8}$.

\subsection{The first mass relation}

A weighted sum of the boundary equations gives the mass relation
$m_1m_3=m_2m_4$. This relation will also make the two omitted boundary equations
vanish in the analytic reduction.

\begin{proposition}[First condition]\label{prop:d3-first}
At every point of $\{t=0\}$ satisfying \eqref{eq:d3-boundary}, $m_1m_3=m_2m_4$.  Hence every
Diagram-3 germ satisfies (29) of \cite[\S8.3]{AK}.
\end{proposition}

\begin{proof}
Since $m_1R^z_1+m_4R^z_4=m_1A+m_4B$ and $m_2R^z_2+m_3R^z_3=m_2A+m_3B$ identically,
\begin{equation}\label{eq:d3-syzygy}
(B-A)(m_1m_3-m_2m_4)=(m_1+m_4)(m_2R^z_2+m_3R^z_3)-(m_2+m_3)(m_1R^z_1+m_4R^z_4),
\end{equation}
and likewise
$(D-C)(m_1m_3-m_2m_4)=(m_1+m_2)(m_3R^w_3+m_4R^w_4)-(m_3+m_4)(m_1R^w_1+m_2R^w_2)$.  With
$T_{13}=D_{13}^2(B-A)(D-C)-1$ (the distance equation of $13$) and $g=(B-A)-(D-C)$ (the gauge),
$D_{13}^2(B-A)^2-1=T_{13}+D_{13}^2(B-A)g$.  At a boundary point $T_{13}=g=0$, so
$D_{13}^2(B-A)^2=1$, $B-A\ne0$, and \eqref{eq:d3-syzygy} gives $m_1m_3=m_2m_4$.  The second
statement follows from Proposition~\ref{prop:d3-cover}(iv).
\end{proof}

{These identities express in chart coordinates the argument of \cite[\S5.2]{AK}:
Rule~1c makes $m_1z_1+m_4z_4$ and $m_2z_2+m_3z_3$ non-maximal, whereas
$z_1\sim z_2$ and $z_3\sim z_4$.}

\subsection{Boundary solutions and analytic continuation}
{Assume henceforth that $m\in\Rp^5$ satisfies (29), and set
$M=m_1+m_2+m_3+m_4$. After eliminating $B=A+1$, denote the remaining $19$ chart coordinates
by $X$. We will solve nineteen independent equations by the implicit function theorem,
leaving two equations that must vanish along any germ.

Let $F_{z_k},F_{w_k}$ denote the rescaled equations for body $k$.
Define $\Lambda$ to be the $19\times21$ matrix selecting the combinations
$m_1F_{z_1}+m_4F_{z_4}$, $F_{z_4}-F_{z_1}$, $F_{z_3}-F_{z_2}$,
$m_1F_{w_1}+m_2F_{w_2}$, $F_{w_2}-F_{w_1}$, $F_{w_3}-F_{w_4}$,
together with the other $13$ equations. The two remaining linear combinations are
$\ell_zF=F_{z_1}-F_{z_2}-\frac{m_3}{m_2}F_{z_3}+\frac{m_4}{m_1}F_{z_4}$ and
$\ell_wF=F_{w_1}+\frac{m_2}{m_1}F_{w_2}-\frac{m_2}{m_1}F_{w_3}-F_{w_4}$.}

{\begin{lemma}\label{lem:d3-fibre}
\begin{alphalist}
  \item On the slice, \eqref{eq:d3-boundary} has exactly $3^4\cdot2^6=5184$ boundary points.
  They are obtained as follows: $D_{12},D_{34},D_{14},D_{23}$ are arbitrary cube roots of $\frac1{m_1+m_2}$,
  $\frac1{m_3+m_4}$, $\frac1{m_1+m_4}$, $\frac1{m_2+m_3}$; $D_{13},D_{24},D_{15},D_{25},D_{35},
  D_{45}$ are arbitrary square roots of $1$, $-1$, $\frac M{m_3}$, $-\frac M{m_4}$, $\frac M{m_1}$,
  $-\frac M{m_2}$; $A=-\frac{m_4}{m_1+m_4}$, $C=-\frac{m_2}{m_1+m_2}$, $D=\frac{m_1}{m_1+m_2}$,
  $x=(m_1+m_2)D_{12}$, $y=-(m_3+m_4)D_{34}$, $u=(m_1+m_4)D_{14}$, $v=(m_2+m_3)D_{23}$, and
  $\hat z_5,\hat w_5$ are determined by the equations for body $5$. These boundary solutions
  are independent of the mass $m_5$.
\item At every boundary point the matrix $\partial_X(\Lambda F^0)$ is invertible.

\end{alphalist}
\end{lemma}}

{\begin{proof}
On the slice $B=A+1$, the three selected combinations of the $z$-equations are
\[
 m_1R^z_1+m_4R^z_4=(m_1+m_4)A+m_4,\qquad
 R^z_4-R^z_1=1-(m_1+m_4)D_{14}^3,
\]
\[
 R^z_3-R^z_2=1-(m_2+m_3)D_{23}^3.
\]
The gauge gives $D-C=1$, while the weighted combination of the $w$-equations gives
$m_1C+m_2D=0$. The other two selected combinations are
\[
 R^w_2-R^w_1=(D-C)\bigl(1-(m_1+m_2)D_{12}^3\bigr),\qquad
 R^w_3-R^w_4=(D-C)\bigl(1-(m_3+m_4)D_{34}^3\bigr).
\]
These equations determine $A,C,D$ and the four cubes listed in (a).
The four distance equations for the sides then determine $x,y,u,v$.
The other six distance equations determine the six squares, and the equations for body $5$
determine $\hat z_5,\hat w_5$.
The relation $m_1m_3=m_2m_4$ ensures that the omitted boundary coordinate rows
also vanish. Since all roots are nonzero and simple, each choice
of the four cube roots and six square roots gives exactly one boundary point.
This accounts for all $3^4 2^6$ boundary points.

The same elimination proves (b). Let a vector belong to the kernel of the differential.
The linear equations force its variations of $A,C,D$ to vanish.
The four cube equations then force the variations of the side coordinates $D_{kl}$ to vanish,
because their coefficients are nonzero multiples of $D_{kl}^2$.
The four side distance equations next force the variations of $x,y,u,v$ to vanish.
Each remaining distance equation does the same for its own $D_{kl}$, since its
derivative in that coordinate is $2/D_{kl}\ne0$.
Finally, the two equations for body $5$ force the variations of $\hat z_5,\hat w_5$ to vanish.
The square matrix $\partial_X(\Lambda F^0)$ therefore has trivial kernel.
\end{proof}

}
{For all $X$, we have $\ell_zF^0=-B(m_1m_3-m_2m_4)/(m_1m_2)$ and
$\ell_wF^0=-(D-C)D_{34}^3(m_1m_3-m_2m_4)/m_1$. Thus $\ell_z,\ell_w$ annihilate
$J_0=\partial_XF^0(X_0)$. Since $\Lambda J_0$ is invertible, $J_0$ has rank $19$,
$\ell_z,\ell_w$ span its left kernel, and the $21$ rows of $\Lambda,\ell_z,\ell_w$ are independent.}

\begin{lemma}\label{lem:d3-graph}
Let $X_0$ be a boundary point.  There is a unique analytic $t\mapsto X(t)$ near $0$ with $X(0)=X_0$ and
$\Lambda F(t,X(t))=0$; put $G(t)=(\ell_zF(t,X(t)),\ell_wF(t,X(t)))=\sum_{k\ge1}G_kt^k$.  If the
chart coordinates of a Diagram-3 germ converge to $X_0$, then along the germ they are $X(t)$, and
$G\equiv0$.  Every analytic $Y(t)$ with $Y(0)=X_0$ and $F(t,Y(t))\equiv0$ equals $X(t)$.
\end{lemma}

\begin{proof}
{The analytic implicit function theorem applies by Lemma~\ref{lem:d3-fibre}(b)
and gives the unique analytic graph. After a reparametrization $s\mapsto s^N$,
the chart coordinates $Y(s)$ and $t(s)=1/z_{13}(s)$ of the germ are convergent power series
(Lemma~\ref{lem:germ} and Proposition~\ref{prop:d3-cover}), with $Y(0)=X_0$, $t(0)=0$ and
$t\not\equiv0$. The identity $\Lambda F(t(s),Y(s))=0$ and uniqueness imply
$Y(s)=X(t(s))$. Consequently, $F(t,X(t))$ vanishes on the image of $t(s)$,
which contains a neighbourhood of $0$. Thus $F(t,X(t))\equiv0$ and $G\equiv0$.
The last assertion follows from the same uniqueness.}
\end{proof}

{Lemma~\ref{lem:d3-graph} reduces the existence of a germ to $G_k=0$ for all $k$.
Since $F^1=0$ and $\Lambda J_0$ is invertible, $X_1=0$ and
\[
 X(t)=X_0+X_2t^2+X_3t^3+\cdots.
\]
Put $\ell=(\ell_z,\ell_w)$. The identity $\ell J_0=0$ removes $X_k$ from the
coefficient $G_k$, giving, with all derivatives evaluated at $X_0$,
\[
\begin{aligned}
 G_2&=\ell F^2,\qquad G_3=\ell F^3,\\
 G_4&=\ell\bigl(F^4+\partial F^2\cdot X_2
                  +\tfrac12\partial^2F^0[X_2,X_2]\bigr),\\
 G_5&=\ell\bigl(F^5+\partial F^3\cdot X_2+\partial F^2\cdot X_3
                  +\partial^2F^0[X_2,X_3]\bigr).
\end{aligned}
\]
Thus each successive obstruction is determined by the coefficients already obtained.}

\subsection{Orders three and four}

The third-order coefficient gives two polynomial mass equations. The fourth-order
coefficient then imposes equality of two masses, reducing the candidates to four
explicit one-parameter families up to scaling.

\begin{lemma}[computer-assisted]\label{lem:d3-order3}
At every boundary point $G_2=0$ and
$G_3=\bigl(\frac{m_1+m_2}{m_1^2m_2}\Omega,\frac{m_1+m_4}{m_1^2m_4}\Omega\bigr)$, where
$\Omega=\Omega_1+\Omega_2$,
\[
\begin{aligned}
\Omega_1&=m_1m_2m_4D_{13}+m_5m_1^2D_{15}+m_5m_2m_4D_{35},\\
\Omega_2&=m_1m_2m_4D_{24}+m_5m_1m_2D_{25}+m_5m_1m_4D_{45}.
\end{aligned}
\]
\end{lemma}

{This is an exact symbolic computation over $\mathbb Q(m_1,m_2,m_4,m_5)$,
with $m_3=m_2m_4/m_1$. It evaluates $\ell F^2$ and $\ell F^3$ at the boundary points and reduces the
results modulo the cubic and quadratic equations for $D_{kl}$.
The verification is described in Appendix~\ref{app:verification}.}

\begin{corollary}\label{cor:d3-order3}
If $G_3=0$ at some boundary point, then $N_1(m)=N_2(m)=0$, where $Q=m_5^2M$ and
\begin{align*}
N_1&=\bigl((m_1m_2m_4)^3-Q(m_1^6+m_2^3m_4^3)\bigr)^2-4Q^2m_1^6m_2^3m_4^3,\\
N_2&=\bigl((m_1m_2m_4)^3-Qm_1^3(m_2^3+m_4^3)\bigr)^2-4Q^2m_1^6m_2^3m_4^3 .
\end{align*}
As $N_1-N_2+QS=0$ with
$S=(m_1^3-m_2^3)(m_1^3-m_4^3)\bigl(2(m_1m_2m_4)^3-Q(m_1^3+m_2^3)(m_1^3+m_4^3)\bigr)$, also
$S(m)=0$.
\end{corollary}

\begin{proof}
At positive masses $D_{13},D_{15},D_{35}$ are real and $D_{24},D_{25},D_{45}$ purely imaginary
(Lemma~\ref{lem:d3-fibre}), so $\Omega_1=\Omega_2=0$.  If
$\alpha+\beta+\gamma=0$ then $(\alpha^2-\beta^2-\gamma^2)^2-4\beta^2\gamma^2=\prod_\pm(\alpha+\beta
\pm\gamma)\prod_\pm(\alpha-\beta\pm\gamma)=0$.  Inserting the squares $(m_1m_2m_4)^2$,
$m_5^2m_1^4M/m_3$, $m_5^2m_2^2m_4^2M/m_1$ of the terms of $\Omega_1$, using
$m_3=m_2m_4/m_1$ and multiplying by $(m_1m_2m_4)^2$ gives $N_1=0$; the terms of
$\Omega_2$ are $i$ times reals with squares $(m_1m_2m_4)^2$, $m_5^2m_1^2m_2^2M/m_4$,
$m_5^2m_1^2m_4^2M/m_2$, which gives $N_2=0$.  Finally $Q>0$.
\end{proof}

\begin{lemma}[computer-assisted]\label{lem:d3-order4}
At every boundary point $G_4=-\frac{\Pi}{m_1^2}\bigl(m_4(m_1+m_2),m_2(m_1+m_4)\bigr)$ with
$\Pi=(m_1D_{12}-m_4D_{34})(m_1D_{14}-m_2D_{23})$.
\end{lemma}

{As in Lemma~\ref{lem:d3-order3}, the computation uses exact arithmetic and reduction
modulo the equations for $D_{kl}$. After solving for $X_2$, it evaluates the preceding formula
for $G_4$ (Appendix~\ref{app:verification}). At every boundary point,}
$(m_1D_{12})^3-(m_4D_{34})^3=m_1(m_1+m_4)(m_1-m_4)/(m_1+m_2)$ and
$(m_1D_{14})^3-(m_2D_{23})^3=m_1(m_1+m_2)(m_1-m_2)/(m_1+m_4)$, so $\Pi=0$ forces $m_1=m_4$ or
$m_1=m_2$; with (29), $m=(p,q,q,p,m_5)$ or $(p,p,q,q,m_5)$.

{\begin{lemma}\label{lem:d3-curves}
Let $m\in\Rp^5$ satisfy (29), $N_1(m)=N_2(m)=0$ and $m_1=m_2$ (resp.\
$m_1=m_4$).  Write $m_1=a^2$, $m_3=b^2$ (resp.\ $m_2=b^2$) with $a,b>0$, and
$c=(2(a^2+b^2))^{1/2}$.  Then
\[
m_5=\frac{a^3b^3}{c\,|a^3+b^3|}\qquad\text{or}\qquad m_5=\frac{a^3b^3}{c\,|a^3-b^3|}\quad(a\ne b).
\]
Conversely $N_1$, $N_2$ and $S$ vanish on these four families, called
$E^\pm_{12}$ (if $m_1=m_2$) and $E^\pm_{14}$ (if $m_1=m_4$).  They meet only at $a=b$, i.e.\ at
$m$ proportional to $(1,1,1,1,\frac14)$, which lies on $E^+_{12}\cap E^+_{14}$.
\end{lemma}}

{\begin{proof}
For either mass pattern, the equations $N_1=0$ and
$N_2=0$ coincide. Their common polynomial factors as
\[
 N_1=N_2
 =a^{12}\bigl(a^6b^6-m_5^2c^2(a^3+b^3)^2\bigr)
          \bigl(a^6b^6-m_5^2c^2(a^3-b^3)^2\bigr).
\]
Positivity of $a,b,c,m_5$ shows that the first factor gives exactly the plus
formula, and the second gives the minus formula, with $a\ne b$.
One of the first two factors of $S$ vanishes, so $S=0$ as well.
The two mass patterns coincide only when all four outer masses are equal.
In that case the minus formula is impossible, while the plus formula gives $m_5=a^2/4$.
For each mass pattern, the plus and minus formulas are distinct when $a,b>0$.
\end{proof}

\begin{lemma}\label{lem:d3-census}
Let $m$ lie on one of the four families.
\begin{alphalist}
  \item If $a\ne b$, exactly $108$ boundary points have $G_3=G_4=0$: two of the eight sign choices for
  $(D_{13},D_{15},D_{35})$ give $\Omega_1=0$, two of those for $(D_{24},D_{25},D_{45})$
  give $\Omega_2=0$, and $\Pi=0$ exactly when $D_{14}=D_{23}$ (on $E^\pm_{12}$) or
  $D_{12}=D_{34}$ (on $E^\pm_{14}$), for $27$ of the $81$ choices of cube roots.
  \item If $a=b$, exactly $180$ boundary points have $G_3=G_4=0$: those with $D_{15}=D_{35}=-2D_{13}$,
  $D_{25}=D_{45}=-2D_{24}$, and $D_{12}=D_{34}$ or $D_{14}=D_{23}$ ($45$ choices of cube roots).
\end{alphalist}
\end{lemma}}

{\begin{proof}
Exchanging $2$ and $4$ together with $z$ and $w$ reduces the proof to
$E_{12}^{\epsilon}$, $\epsilon\in\{1,-1\}$.
Write $A_0=a^3$, $B_0=b^3$, $\nu=\operatorname{sgn}(A_0+\epsilon B_0)$,
and set $D_{13}=\sigma$, $D_{15}=s_1c/b$, $D_{35}=s_2c/a$.
After division by a nonzero real factor, $\Omega_1=0$ becomes
\[
 (\sigma\nu+s_1)A_0+(\epsilon\sigma\nu+s_2)B_0=0.
\]
For $\epsilon=1$, any nonzero coefficients have the same sign, so the equation
forces both coefficients to vanish. For $\epsilon=-1$, cancellation between two nonzero
coefficients would require $A_0=B_0$, which is excluded on the minus family.
There are therefore exactly two real sign triples. Applying the same argument
to $D_{24}/i,D_{25}/i,D_{45}/i$ gives exactly two imaginary sign triples.
If $a\ne b$, the first factor of $\Pi$ cannot vanish because the cubes of
$a^2D_{12}$ and $b^2D_{34}$ are $a^4/2$ and $b^4/2$.
Consequently $\Pi=0$ exactly when $D_{14}=D_{23}$, leaving $3^3=27$ cube-root
choices and hence $4\cdot27=108$ boundary points.

At $a=b$, the sign conditions reduce to
$D_{15}=D_{35}=-2D_{13}$ and $D_{25}=D_{45}=-2D_{24}$, again giving four choices.
Each equality $D_{12}=D_{34}$ or $D_{14}=D_{23}$ admits $27$ cube-root choices,
and their intersection admits $9$. Thus their union gives
$4(27+27-9)=180$ boundary points; the intersection gives $4\cdot9=36$.
\end{proof}

}\subsection{Orders five and six}

The fifth-order coefficient excludes the unequal-mass points of the four families.
At their intersection, the sixth-order coefficient leaves $36$ boundary points.
We construct the corresponding analytic germs explicitly and record the identity
that will exclude positive points on their algebraic curves.

\begin{lemma}[computer-assisted]\label{lem:d3-order5}
Let $m$ lie on one of the four families with $a\ne b$.  Then $G_5\neq0$ at each of the $108$
boundary points of Lemma~\ref{lem:d3-census}(a).  At $a=b$, $G_5=0$ at all $180$ boundary points of
Lemma~\ref{lem:d3-census}(b).
\end{lemma}

{The fifth-order calculation gives the following identities. On $E^+_{12}$,
where $D_{14}=D_{23}$, the $z$-component of $G_5$ is a nonzero factor times
$(a-b)L_+(\xi_1,\xi_2)$, where $\xi_1=D_{12}/D_{14}$,
$\xi_2=D_{34}/D_{14}$ and
\[
 L_\epsilon(\xi_1,\xi_2)
 =a^2F_{4,\epsilon}\xi_1-a^{-2}b^4F_{4,\epsilon}\xi_2
       -b^2(a^2-b^2)F_{2,\epsilon},
\]
\[
 F_{4,\epsilon}=a^4+\epsilon a^3b+a^2b^2+\epsilon ab^3+b^4,
 \qquad F_{2,\epsilon}=a^2+\epsilon ab+b^2,
 \qquad \epsilon\in\{1,-1\},
\]
with $\xi_1^3=(a^2+b^2)/(2a^2)$ and
$\xi_2^3=(a^2+b^2)/(2b^2)$.
The product of the nine values of $L_+$ is, up to sign and a nonzero monomial
factor, $(a-b)^3(a+b)^3P_{66}(a,b)$.
On $E^-_{12}$, the calculation gives $G_5=0\Rightarrow P_{66}(a,-b)=0$
at each of the $108$ boundary points. All factors removed in obtaining this implication
are regular and nonzero for $a,b>0$, $a\ne b$.
These identities are the computer-assisted part of the argument.
Their required nonvanishing follows from the elementary estimates below.

We treat both signs together. Interchanging $a,b$ and $\xi_1,\xi_2$ multiplies $L_\epsilon$
by the nonzero factor $-a^2/b^2$. We may therefore assume $a>b>0$ and write
\[
 r=b/a\in(0,1),\quad c_0=\left(\frac{1+r^2}{2}\right)^{1/3},\quad
 f_{4,\epsilon}=1+\epsilon r+r^2+\epsilon r^3+r^4,\quad
 f_{2,\epsilon}=1+\epsilon r+r^2.
\]
Writing $\xi_1=c_0\omega^j$ and $\xi_2=c_0r^{-2/3}\omega^k$,
$\omega=e^{2\pi i/3}$, gives
\[
 a^{-6}L_\epsilon=A_\epsilon\omega^j-B_\epsilon\omega^k-C_\epsilon,
 \quad A_\epsilon=c_0f_{4,\epsilon},\quad
 B_\epsilon=A_\epsilon r^{10/3},\quad
 C_\epsilon=r^2(1-r^2)f_{2,\epsilon}.
\]
All three coefficients are positive, with $A_\epsilon>B_\epsilon$. Since
\[
 \operatorname{Im}(A_\epsilon\omega^j-B_\epsilon\omega^k-C_\epsilon)
   =A_\epsilon\operatorname{Im}\omega^j-B_\epsilon\operatorname{Im}\omega^k,
 \qquad \operatorname{Im}\omega^j\in\{0,\sqrt3/2,-\sqrt3/2\},
\]
vanishing of the imaginary part forces $j=k=0$. We therefore compare the real
coefficients. For the plus sign, $c_0>r^2$ and
$f_{4,+}>f_{2,+}$. For the minus sign, we have
\[
 \frac{f_{4,-}}{f_{2,-}}=\frac{1+r^5}{1+r^3}>r^{4/3},\qquad c_0>r^{2/3},
\]
since
\[
 1+r^5-r^{4/3}-r^{13/3}
 =(1-r^{4/3})(1-r^{13/3})+r^5-r^{17/3}>0.
\]
In both cases, therefore, $c_0f_{4,\epsilon}>r^2f_{2,\epsilon}$, and
\[
 A_\epsilon-B_\epsilon
 =c_0f_{4,\epsilon}(1-r^{10/3})
 >r^2f_{2,\epsilon}(1-r^2)=C_\epsilon.
\]
This excludes the remaining real-root case, so all nine values of $L_\epsilon$ are nonzero.
The product identity for $L_+$ gives $P_{66}(a,b)\ne0$.
Replacing $b$ by $-b$ in the same identity gives $P_{66}(a,-b)\ne0$ for $a\ne b$.
The fifth-order identities thus exclude both families when $a\ne b$.

}
{Exchanging bodies $2$ and $4$ together with $z$ and
$w$ maps the chart at $m$ onto the chart at $(m_1,m_4,m_3,m_2,m_5)$, $E^\pm_{12}$ onto
$E^\pm_{14}$ and the $108$ boundary points onto the $108$ boundary points; the relabelling $(13)(24)$ exchanges $a$
and $b$, accounting for the minus families with $a<b$.
At $a=b$, the calculation gives zero for both components of $G_5$ at the boundary points of $E^+_{12}$.
The exchange then gives vanishing at all $180$ boundary points.}

\begin{lemma}[computer-assisted]\label{lem:d3-order6}
Let $m=(\frac12,\frac12,\frac12,\frac12,\frac18)$.  Of the $180$ boundary points of
Lemma~\ref{lem:d3-census}(b), the $144$ boundary points with $D_{12}\ne D_{34}$ or $D_{14}\ne D_{23}$ have
$G_6\neq0$.
\end{lemma}

{The lemma is proved by computing the jets modulo two primes
(Appendix~\ref{app:verification}). All data lie in $R=\mathbb Z[\frac12,i,\omega]$,
where $\omega=e^{2\pi i/3}$. Choose a prime $\mathfrak p$ of $R$ above $p\in\{13,37\}$.
Since $p\equiv1\pmod{12}$, we have $R/\mathfrak p=\mathbb F_p$.

If a germ passed through a boundary point, Lemma~\ref{lem:d3-graph} would give $F(t,X(t))\equiv0$.
At each order, $X_k$ would then be uniquely determined by the order-$k$ equations,
using any $19$ rows with an invertible minor, and all $21$ residuals would vanish.
We use $19$ rows whose minor is a unit modulo $\mathfrak p$.
The calculation produces a nonzero order-$6$ residual at each of the $144$ boundary points.
Changing $\mathfrak p$ acts on the boundary points through $\mathrm{Gal}(\mathbb Q(i,\omega)/\mathbb Q)$,
with $\omega\mapsto\omega^2$ and $i\mapsto-i$. This action preserves the set of $144$ boundary points,
and the two primes give the same excluded set.}

{\begin{lemma}\label{lem:d3-twisted}
Let $m=(\frac12,\frac12,\frac12,\frac12,\frac18)$, $\alpha,\beta\in\{1,\omega,\omega^2\}$ and
$\varepsilon,\eta\in\{\pm1\}$.  Let $(P,Q,Y)$ be the analytic solution near $\tau=0$ of
\[
(1-P)(1+Q)=\alpha\tau^2,\qquad (1+P)(1-Q)=\beta\tau^2,\qquad 4PQY^2+1=0
\]
with $(P,Q,Y)(0)=(1,1,\eta i/2)$; the Jacobian there is $\pm16i$.  Then
\[
z=\tfrac1\tau(1,P,-1,-P,0),\qquad w=\tfrac1\tau(1,-Q,-1,Q,0),
\]
\[
\delta_{12}=\delta_{34}=\alpha,\ \ \delta_{14}=\delta_{23}=\beta,\ \
\delta_{13}=\tfrac{\varepsilon\tau}2,\ \ \delta_{15}=\delta_{35}=-\varepsilon\tau,\ \
\delta_{24}=\tau Y,\ \ \delta_{25}=\delta_{45}=-2\tau Y
\]
solves \eqref{eq:system} with the gauge $z_{13}=w_{13}$.  In the chart, with $t=-\tau/2$, it is an
analytic solution of $F=0$ through the boundary point with $D_{12}=D_{34}=\alpha$, $D_{14}=D_{23}=\beta$,
$D_{13}=-\varepsilon$ and $D_{24}=-\eta i$.  These $36$ boundary points are the boundary points of
Lemma~\ref{lem:d3-census}(b) with $D_{12}=D_{34}$ and $D_{14}=D_{23}$.  Along these germs
\begin{equation}\label{eq:d3-identities}
\delta_{15}+2\delta_{13}=\delta_{35}+2\delta_{13}=\delta_{25}+2\delta_{24}=\delta_{45}+2\delta_{24}=0 .
\end{equation}
\end{lemma}}

{\begin{proof}
The derivative of the three defining equations with respect to $(P,Q,Y)$ at
$(\tau,P,Q,Y)=(0,1,1,\eta i/2)$ is
\[
 \begin{pmatrix}-2&0&0\\0&-2&0\\-1&-1&4\eta i\end{pmatrix},
\]
with determinant $16\eta i\ne0$. The implicit function theorem gives the
claimed analytic solution. Since the equations depend on $\tau$ only through
$\tau^2$, uniqueness makes $P,Q,Y$ even, and $P-1,Q-1=O(\tau^2)$.

We verify the body equations first. For either coordinate $q=z$ or $q=w$,
the two adjacent side forces at outer body $k$ sum to $q_k$:
the cubes of their inverse distances are $1$, the adjacent masses are $1/2$,
and the adjacent bodies have opposite positions.
At body $1$ or $3$, the diagonal force and the force from body $5$ cancel,
since
\[
 \tfrac12(\varepsilon\tau/2)^3(2q_k)
       +\tfrac18(-\varepsilon\tau)^3q_k=0.
\]
At body $2$ or $4$ the same cancellation is
\[
 \tfrac12(\tau Y)^3(2q_k)+\tfrac18(-2\tau Y)^3q_k=0.
\]
Pairing opposite bodies also proves the equation for body $5$.
For the distance equations, the pairs $12,34$ reduce to
$\alpha^2(1-P)(1+Q)/\tau^2=1$, the pairs $14,23$ to
$\beta^2(1+P)(1-Q)/\tau^2=1$, and $24,25,45$ to $-4PQY^2=1$;
the remaining distance equations hold identically. The gauge follows from
$z_{13}=w_{13}=-2/\tau$.
Finally $t=-\tau/2$ makes the chart coordinates analytic, since the only
apparent quotients include $(P-1)/\tau^2$ and $(Q-1)/\tau^2$.
The displayed inverse distances give \eqref{eq:d3-identities} identically
and yield precisely the $9\cdot4=36$ boundary points stated in the lemma.
\end{proof}

}\subsection{Exclusion of curves containing a positive point}

We now pass from the local analytic classification to the original algebraic curve.
The inverse-distance identities hold along its germ and hence, by irreducibility,
on the whole curve. Their signs at a positive point give the required contradiction.

\begin{theorem}\label{thm:d3}
Let $m\in\Rp^5$, let $\Gamma\subset A_m$ be an irreducible curve, and let $\gamma$ be a Diagram-3
germ on $\Gamma$, with the bodies numbered as in Diagram~3.  Then $m_1=m_2=m_3=m_4=4m_5$, and
\eqref{eq:d3-identities} holds identically on $\Gamma$.  In particular $\Gamma$ contains no
positive point, and Theorem~\ref{thm:exclusion}(c) holds.
\end{theorem}

\begin{proof}
{Exchanging $z$ and $w$ and renumbering if necessary puts $\gamma$ in the form
described at the start of this section. These operations preserve the conclusions.
Proposition~\ref{prop:d3-cover} gives a limiting chart point $X_0$ satisfying
\eqref{eq:d3-boundary} and $B-A=1$. Relation (29) holds by Proposition~\ref{prop:d3-first},
so Lemma~\ref{lem:d3-fibre} identifies $X_0$ as a boundary point. Lemma~\ref{lem:d3-graph} then gives
$G\equiv0$ at $X_0$.

The third-order condition, from Lemma~\ref{lem:d3-order3} and
Corollary~\ref{cor:d3-order3}, gives $N_1(m)=N_2(m)=0$.
The fourth-order condition in Lemma~\ref{lem:d3-order4} gives $m_1=m_4$ or $m_1=m_2$.
Thus $m$ belongs to one of the four families in Lemma~\ref{lem:d3-curves}.
Lemmas~\ref{lem:d3-census}(a) and~\ref{lem:d3-order5} exclude every point with $a\ne b$,
leaving $a=b$, or equivalently $m_1=m_2=m_3=m_4=4m_5$.

The covariance of \eqref{eq:system} under
$(m,z,w,\delta)\mapsto(\kappa^3m,\kappa z,\kappa w,\kappa^{-1}\delta)$, $\kappa>0$,
preserves both Diagram-3 germs and \eqref{eq:d3-identities}.
We may therefore normalize to $m=(\frac12,\frac12,\frac12,\frac12,\frac18)$.
Lemmas~\ref{lem:d3-census}(b) and~\ref{lem:d3-order6} now imply $D_{12}=D_{34}$ and
$D_{14}=D_{23}$ at $X_0$.
By Lemmas~\ref{lem:d3-twisted} and~\ref{lem:d3-graph}, the germ $\gamma$ coincides in the chart
with the analytic germ of Lemma~\ref{lem:d3-twisted} through $X_0$. Since renormalization leaves $\delta_{kl}$ unchanged,
\eqref{eq:d3-identities} holds along $\gamma$.

Each expression in \eqref{eq:d3-identities} is a regular function on the irreducible curve
$\Gamma$. It vanishes at infinitely many points of $\gamma$, and hence vanishes identically
on $\Gamma$. A positive point would have $\delta_{15}+2\delta_{13}>0$, a contradiction.}
\end{proof}

{Theorem~\ref{thm:d3} classifies all Diagram~3 germs: the mass vector is
proportional to $(4,4,4,4,1)$ and, up to the stated symmetries, the germ is one of the
$36$ analytic families in Lemma~\ref{lem:d3-twisted}. The modified Roberts continuum
provides an explicit member of this classification.}

\begin{remark}[The modified Roberts continuum]\label{rem:d3-roberts}
The family of Roberts~\cite{Roberts} has the following signed-distance variant in \cite{AK}.
Let $m=(\frac12,\frac12,\frac12,\frac12,\frac18)$, proportional to $(4,4,4,4,1)$
\cite[Rem.~8, p.~583]{AK}.  The curve
\[
z=(a,b,-a,-b,0),\quad w=(a,-b,-a,b,0),\quad b^2=a^2-1,
\]
\[
\delta_{12}=\delta_{23}=\delta_{34}=\delta_{14}=1,\ \ \delta_{13}=\tfrac1{2a},\ \
\delta_{24}=\tfrac{i}{2b},\ \ \delta_{15}=\delta_{35}=-\tfrac1a,\ \ \delta_{25}=\delta_{45}=-\tfrac ib
\]
is the modified Roberts continuum \cite[p.~543]{AK}.  It solves \eqref{eq:system} with the gauge
$z_{13}=w_{13}$, $U=I=1$ on it, and \eqref{eq:d3-identities} holds, so it has no positive point.
Since $w_{12}/z_{12}=(a+b)^2$ on it, the rotation by $\mu=a+b$ gives an irreducible curve
$\Gamma_R\subset A_m$.  For $0<a<1$ and $b\in i\R$, $|\mu|=1$ and the points are real, with
$\delta_{13}>0>\delta_{15}=\delta_{35}$.
{Before this rotation, the end $a\to\infty$ with $b\sim a$ is precisely the analytic
germ of Lemma~\ref{lem:d3-twisted} with
$(\alpha,\beta,\varepsilon,\eta)=(1,1,1,1)$, $\tau=1/a$ and $P=Q=b/a$.
In this balanced $z_{13}=w_{13}$ representative its strokes are
$z_{14}$, $z_{23}$, $w_{12}$, $w_{34}$, so the corresponding end of $\Gamma_R$ is a Diagram-3 germ}
(Roberts' continuum at infinity \cite[\S7.2]{AK}); the sides keep length $1$,
$r_{13},r_{24},r_{k5}\to\infty$, $p_{12}=a^2(1-a^2)$ is unbounded, and (29) holds.
{Thus $\Gamma_R$ gives a real arc with signed inverse distances,
$U=I=1$, and an unbounded four-product at its Diagram~3 end. The identity
$\delta_{15}+2\delta_{13}=0$ describes its sign constraint directly.}
\end{remark}

\section{The triangle in Diagrams 4 and 5}\label{sec:d45}

{Diagrams~4 and~5 contain the same isolated $zw$-triangle $345$ with uncircled vertices.
We first identify its limiting shape and mass relation. We then use the forces between
this triangle and the two remaining bodies to exclude Diagram~4 and to prove
$\Delta\to0$ in Diagram~5. These are the two local conclusions required by
Theorem~\ref{thm:exclusion}.}

\begin{proposition}[Relation (23)]\label{prop:d45-first}
Let $m\in\Rp^5$ and let $\gamma$ be a Diagram 4 or Diagram 5 germ, renormalized as in
Lemma~\ref{lem:germ} so that $\mathcal Z$ and $\mathcal W$ have the same least order $-2\nu<0$.
Let $z_c$ and $w_c$ be the centres of mass of $T=\{3,4,5\}$.  Then the limits
\[
  \zeta_k=\lim_{t\to0}t^{-\nu}(z_k-z_c),\qquad \eta_k=\lim_{t\to0}t^{-\nu}(w_k-w_c),\qquad
  D_{kl}=\lim_{t\to0}t^{\nu}\delta_{kl}\qquad(k,l\in T)
\]
exist and solve the zero-multiplier triangle system of Section~\ref{sec:d1}, and
$\sum_km_k\zeta_k=\sum_km_k\eta_k=0$.  Hence $(\zeta,\eta,D)$ is as in
Lemma~\ref{lem:d1-shape}, and $Q(m_3,m_4,m_5)=0$, that is, the masses satisfy (23).
\end{proposition}

\begin{proof}
{The sides of $T$ are $zw$-edges. Estimate~1 and the left equality case of Estimate~2,
as stated in Fact~\ref{fact:rules} and justified for germs by Lemma~\ref{lem:germ}, give
$\ord z_{kl}=\ord w_{kl}=\nu$ and $\ord\delta_{kl}=-\nu$ for $k,l\in T$.
Since $z_k-z_c$ and $w_k-w_c$ are linear combinations of these differences, the limits exist.
The choice of the centres of mass makes both sums $\sum_km_k\zeta_k$, $\sum_km_k\eta_k$ vanish.

For $k\in T$, consider the equation $z_k=\sum_{l\ne k}m_l\delta_{kl}^3(z_k-z_l)$.
The terms with $l\in T$ have order $-2\nu$. Every other term has larger order:
$z_k$ because $k$ is uncircled, and the terms with $l\notin T$ because $T$ is isolated
and no $z$-stroke leaves it. Multiplication by $t^{2\nu}$ followed by the limit $t\to0$
therefore gives $\phi_k=0$. The same argument gives $\psi_k=0$.
The distance equations within $T$ give $\chi_{kl}=0$.
Lemma~\ref{lem:d1-shape} applies, and for positive masses the relation $Q=0$ is equivalent to (23).}
\end{proof}

\subsection{Diagram 4}\label{sec:d4}

Diagram~4 consists of the isolated $zw$-triangle $345$, whose vertices are not circled, and the
isolated $z$-stroke $12$ with $z$-circles at $1$ and $2$; there are no other strokes or circles
\cite[\S7.4, p.~567, and Fig.~11]{AK}.
The common triangle gives relation~(23), as in \cite[\S8.4]{AK}.
We shall combine it with two force balances to exclude the diagram. We use the notation of
Section~\ref{sec:d1}: $P=\{1,2\}$, $T=\{3,4,5\}$, $i\in P$, $k,l\in T$, $M_P$, $M_T$, $Z_P$, $Z_T$,
$W_P$, $W_T$, the equations $E^z_j,E^w_j$, the triangle functions $\phi_k,\psi_k,\chi_{kl}$, the
torque $T_P$ of Lemma~\ref{lem:d1-syzygies} and the orbit $O$ of Lemma~\ref{lem:d1-mb}.
Renumbering inside $P$ or inside $T$ does not change the diagram.

\begin{theorem}\label{thm:d4}
Let $m\in\Rp^5$.  No curve in $A_m$ carries a Diagram~4 germ.
\end{theorem}

{To prove part~(a) of Theorem~\ref{thm:exclusion} for $r=4$, we proceed in two steps.

First, we show that the chart contains every germ. A torque identity rules out a triangle
body at the triangle's $w$-centre (Lemma~\ref{lem:d4-central}).
Every germ then has a limit in the boundary fibre $\mathcal B_4$,
which is smooth near that limit (Proposition~\ref{prop:d4-cover} and Lemma~\ref{lem:d4-fibre}).

Second, we derive two necessary balances of the leading forces.
The Lyapunov--Schmidt reduction in Lemma~\ref{lem:d4-ls} expresses them as successive
obstructions to the existence of a germ. The first obstruction requires equal leading
$w$-forces on the two outer bodies. The second vanishes identically, and the third requires
equal leading $z$-forces (Lemma~\ref{lem:d4-obstr}).
{Positivity makes these two balances incompatible}
(Lemma~\ref{lem:d4-census}).}

\subsubsection{The chart}\label{d4-chart}
{The invariance of \eqref{eq:system}, apart from the gauge, under
$(z,w,\delta)\mapsto(cz,c^{-1}w,\delta)$, $c\in\C^*$ allows us to use
$z_{34}=w_{34}$ in place of $z_{12}=w_{12}$. With parameter $t$, write}
\[
  (\zeta_3,\zeta_4,\zeta_5)=(m_5u,\ m_5v,\ -m_3u-m_4v),\qquad
  (\eta_3,\eta_4,\eta_5)=(m_5p,\ m_5q,\ -m_3p-m_4q),
\]
so that $\sum_km_k\zeta_k=\sum_km_k\eta_k=0$, and
\begin{equation}\label{eq:d4-chart}
\begin{gathered}
  z_i=t^{-4}A_i,\qquad w_i=t^{4}B_i\qquad(i\in P),\\
  z_k=t^{-1}\hat Z+t^{2}\zeta_k,\qquad w_k=t^{5}\hat W+t^{2}\eta_k\qquad(k\in T),\\
  \delta_{12}=D_{12},\qquad\delta_{ik}=t\,D_{ik},\qquad\delta_{kl}=t^{-2}D_{kl}.
\end{gathered}
\end{equation}
{The chart has twenty coordinates $x$: $A_1,A_2,B_1,B_2,\hat Z,\hat W,u,v,p,q$
and the ten $D_{jl}$. The substitution is invertible for $t\ne0$.
For convenience, set}
\[
\begin{gathered}
  \Sigma_i=\sum_km_kD_{ik}^3,\qquad f_i=\sum_km_kD_{ik}^3\eta_k,\\
  R^z_i=A_i-m_{i'}D_{12}^3(A_i-A_{i'}),\qquad R^w_i=B_i-m_{i'}D_{12}^3(B_i-B_{i'}),
\end{gathered}
\]
where $\{i,i'\}=P$.  Substituting \eqref{eq:d4-chart} gives
\begin{equation}\label{eq:d4-rows}
\begin{aligned}
  t^{4}E^z_i&=R^z_i-\textstyle\sum_km_kD_{ik}^3\,(t^3A_i-t^6\hat Z-t^9\zeta_k),\\
  t^{-4}E^w_i&=R^w_i+t\,f_i-\textstyle\sum_km_kD_{ik}^3\,(t^3B_i-t^4\hat W),\\
  t^{4}E^z_k&=-\phi_k+t^3\hat Z+t^6\zeta_k-\textstyle\sum_im_iD_{ik}^3\,(t^6\hat Z+t^9\zeta_k-t^3A_i),\\
  t^{4}E^w_k&=-\psi_k+t^6\eta_k+t^9\hat W-\textstyle\sum_im_iD_{ik}^3\,(t^9\eta_k+t^{12}\hat W
  -t^{11}B_i),
\end{aligned}
\end{equation}
{The distance equations have left-hand sides
$D_{12}^2(A_1-A_2)(B_1-B_2)-1$, the crossed expressions
$D_{ik}^2(t^3\hat Z+t^6\zeta_k-A_i)(\eta_k-t^2B_i+t^3\hat W)-1$, and $\chi_{kl}$.
The gauge becomes $t^{-2}(z_{34}-w_{34})=(\zeta_4-\zeta_3)-(\eta_4-\eta_3)$.
Together these give $21$ polynomials, denoted by $F(t,x)=\sum_{j\ge0}t^jF_j(x)$.
Their zero set at the boundary is the \emph{boundary fibre} $\mathcal B_4=\mathcal B_4(m)$,
the zero set of $F_0$. Equation~\eqref{eq:d4-rows} shows that $F_1$ vanishes except in the two
rows $t^{-4}E^w_i$, where it equals $f_i$. Likewise, $F_2$ vanishes except in the crossed
distance equations, where it equals $D_{ik}^2A_iB_i$.}

\subsubsection{Representation of germs in the chart}\label{d4-cover}
{Let $\gamma$ be a Diagram~4 germ at $m\in\Rp^5$.
Normalize it as in Lemma~\ref{lem:germ}, so that $Z\approx W\approx\epsilon^{-2}$.
After a base change of the branch parameter $\tau$, all coordinates are Laurent series
in $\tau$, with $\ord_\tau Z=\ord_\tau W=-4q$ for an integer $q\ge1$.
Set $t=\tau^q$ and measure orders in units of $\ord t$.
The maximal components of $\mathcal Z$ and $\mathcal W$ then have order $-4$,
and $\ord\epsilon=2$.}

\begin{lemma}\label{lem:d4-val}
Along $\gamma$, for $i\in P$ and $k\ne l$ in $T$:
\begin{romanlist}
  \item $\ord z_k>-4$ and $\ord z_{ik}=-4$;
  \item $\ord z_1=\ord z_2=\ord z_{12}=-4$, $\ord\delta_{12}=0$, $\ord w_{12}=4$,
  $\ord w_1,\ord w_2\ge4$, and hence $\ord W_T=\ord W_P\ge4$;
  \item $\ord z_{kl}=\ord w_{kl}=2$ and $\ord\delta_{kl}=-2$.
\end{romanlist}
\end{lemma}

\begin{proof}
{Part (i) follows because the bodies in $T$ have no $z$-circles, whereas $1$ and $2$ do.

For (ii), the $z$-circles at bodies $1,2$ give $\ord z_1=\ord z_2=-4$.
The isolated component $\{1,2\}$ of the $z$-diagram has circled vertices, so Rule~1e
\cite[p.~546]{AK} makes its stroke $12$ maximal. The right equality case of Estimate~2
\cite[p.~548]{AK} gives $r_{12}\approx1$, $z_{12}\approx\epsilon^{-2}$ and
$w_{12}\approx\epsilon^2$. The triangle edges are not maximal: part (iii), which is independent
of (ii), gives $z_{kl}\approx\epsilon$ on them. Thus $12$ is the only maximal $z$-edge.
Estimate~1 \cite[p.~548]{AK}, with $z$ and $w$ exchanged, now gives $w_{12}\prec w_{ik}$
for $i\in P$, $k\in T$.

Apply \cite[Prop.~1, p.~559]{AK}, with $z$ and $w$ exchanged, to bodies $1$ and $2$,
as in the proof of \cite[Prop.~4, p.~564]{AK}. It places the origin in the $w$-cluster
formed by $w_1$ and $w_2$, whose size is $w_{12}$. Hence $w_1,w_2\preccurlyeq w_{12}$.
The relation $M_PW_P+M_TW_T=0$ gives the assertion about the two centres.

For (iii), the triangle edges are $zw$-edges. Estimate~1 and the left equality case
of Estimate~2 give $r_{kl}\approx z_{kl}\approx w_{kl}\approx\epsilon$.}
\end{proof}

{\begin{lemma}\label{lem:d4-torque}
Let $0<e<2$.  In \eqref{eq:d4-chart} replace $\hat Z$ by $t^{-3e/2}\hat Z$, $\hat W$ by
$t^{-e/2}\hat W$, $(\eta_3,\eta_4,\eta_5)$ by $(t^e\omega,\ m_5q,\ -m_3m_5^{-1}t^e\omega-m_4q)$ and
$D_{i3}$ by $t^{-e/2}d_i$ $(i\in P)$.  Then
\[
  T_P=t^{1-e/2}\Big(m_3\sum_{i\in P}m_id_i\,\big(d_i^2A_i\omega\big)+\sum_{g\in G}c_g\,t^g\Big),
\]
where the $c_g$ are polynomials in the chart coordinates and
$G=\{\tfrac12e,\ \tfrac32e,\ 3,\ 3-\tfrac32e,\ 5-\tfrac52e,\ 5-e,\ 8-e,\ 8+\tfrac12e\}$; all
elements of $G$ are positive.
\end{lemma}
\begin{proof}
After the stated substitutions,
\[
 z_iw_k-w_iz_k
 =t^{-2}A_i\eta_k+t^{1-e/2}A_i\hat W
       -t^{3-3e/2}B_i\hat Z-t^6B_i\zeta_k.
\]
For $k=3$, we have $\eta_3=t^e\omega$ and
$\delta_{i3}^3=t^{3-3e/2}d_i^3$.  The four terms, divided by
$t^{1-e/2}$, have exponents
$0$, $3-3e/2$, $5-5e/2$ and $8-e$, and the first coefficient is
$d_i^3A_i\omega$.
For $k=4,5$, multiply by $t^3D_{ik}^3$ instead.  The constant and
$t^e$ parts of $\eta_k$ give the exponents $e/2$ and $3e/2$;
the last three terms give $3$, $5-e$ and $8+e/2$.
These exponents give the eight elements of $G$, all of which are positive
because $0<e<2$.
\end{proof}

}\begin{lemma}\label{lem:d4-central}
For $k\in T$ let $o_k$ be the coefficient of $t^2$ in $w_k-W_T$ along $\gamma$.  Then
$o_3,o_4,o_5\ne0$.
\end{lemma}

\begin{proof}
{Lemma~\ref{lem:d4-val}(iii) identifies $o_l-o_k$ as the leading coefficient
of $w_{kl}$. Thus at most one}
{$o_k$ vanishes.  Renumbering $T$, suppose that $o_3=0$.  Write
$\ord(w_3-W_T)=2+e$ with $e>0$ when $w_3-W_T\not\equiv0$, and put $e=+\infty$ otherwise.}
By Lemma~\ref{lem:d4-val}(ii), $\ord(W_T-w_i)\ge4$.

{There are two cases. If $e\ge2$, then $\ord w_{i3}\ge4$, so
$\ord\delta_{i3}=-\frac12(\ord z_{i3}+\ord w_{i3})\le0$ and
$\ord Z_{i3}=\ord z_{i3}+3\ord\delta_{i3}\le-4$. This would make $Z_{i3}$ a maximal component
of $\mathcal Z$, producing a $z$-stroke between $i$ and $3$, contrary to Diagram~4.

Suppose instead that $0<e<2$. We have $\ord w_{i3}=2+e$ and $\ord\delta_{i3}=1-\frac e2$,
whereas $\ord w_{ik}=2$ and $\ord\delta_{ik}=1$ for $k=4,5$.
The first two identities of Lemma~\ref{lem:d1-syzygies} express $M_PW_P$ and $M_TZ_T$
as sums of crossed terms $m_im_k\delta_{ik}^3(w_i-w_k)$ and $m_im_k\delta_{ik}^3(z_k-z_i)$.
These have orders at least $5-\frac e2$ and $-1-\frac32e$, respectively.
Together with $M_PW_P+M_TW_T=0$, this gives $\ord W_T\ge5-\frac e2$ and
$\ord Z_T\ge-1-\frac32e$.

The modified chart in Lemma~\ref{lem:d4-torque} therefore represents $\gamma$ with all coordinates
of nonnegative order, and with $A_i,\omega,d_i$ of order $0$.
At $t=0$, the distance equations $\delta_{i3}^2z_{i3}w_{i3}=1$ give $d_i^2A_i\omega=-1$,
while $m_1A_1+m_2A_2=t^4(m_1z_1+m_2z_2)=-t^4M_TZ_T$ vanishes.
Lemma~\ref{lem:d4-torque} now shows that $t^{e/2-1}T_P$ tends to $-m_3(m_1d_1+m_2d_2)$.
This limit is nonzero, since}
\[
  (m_1d_1)^2-(m_2d_2)^2=-\frac{m_1^2}{A_1\omega}+\frac{m_2^2}{A_2\omega}
  =-\frac{m_1^3+m_2^3}{m_1A_1\omega}\ne0 .
\]
{This contradicts $T_P=0$ on $A_m$, proved in Lemma~\ref{lem:d1-syzygies}.}
\end{proof}

\begin{proposition}[Coverage]\label{prop:d4-cover}
Let $m\in\Rp^5$ and let $\gamma$ be a Diagram~4 germ.  Then $m$ satisfies (23), and
$\ord Z_T\ge-1$, $\ord W_T\ge5$ along $\gamma$.  Renormalize by $(z,w)\mapsto(cz,c^{-1}w)$ with
$c=(w_{34}/z_{34})^{1/2}$, so that $z_{34}=w_{34}$.  After a base change of $\tau$, every chart
coordinate of \eqref{eq:d4-chart} is a power series in $\tau$; the value $y$ of $x$ at $\tau=0$
lies in $\mathcal B_4(m)$, and $\eta_k(y)\ne0$ for $k\in T$.
\end{proposition}

\begin{proof}
Relation (23) is Proposition~\ref{prop:d45-first}.  By Lemmas~\ref{lem:d4-val}
and~\ref{lem:d4-central}, $\ord w_{ik}=2$ for all $i\in P$, $k\in T$, hence $\ord\delta_{ik}=1$.
The first two identities of Lemma~\ref{lem:d1-syzygies} then give $\ord Z_T\ge3-4=-1$ and
$\ord W_P\ge3+2=5$, and $\ord W_T=\ord W_P$.  The factor $c$ has order $0$ by
Lemma~\ref{lem:d4-val}(iii), so all orders persist after the renormalization.  Inverting
\eqref{eq:d4-chart}, $A_i=t^4z_i$, $B_i=t^{-4}w_i$, $\hat Z=tZ_T$, $\hat W=t^{-5}W_T$,
$\zeta_k=t^{-2}(z_k-Z_T)$, $\eta_k=t^{-2}(w_k-W_T)$, $D_{12}=\delta_{12}$, $D_{ik}=t^{-1}\delta_{ik}$
and $D_{kl}=t^2\delta_{kl}$; by Lemma~\ref{lem:d4-val} all have nonnegative order, and
$\eta_k\to o_k\ne0$.  The rows of $F$ are monomial multiples of the equations, so they vanish
along $\gamma$, and $F_0(y)=0$.
\end{proof}

\subsubsection{The boundary fibre}\label{d4-fibre}
On $\C^{20}$ consider the scalings
\[
\begin{gathered}
  \alpha\cdot x:\ (A,B,D_{ik})\mapsto(\alpha^2A,\ \alpha^{-2}B,\ \alpha^{-1}D_{ik}),\\
  \beta\cdot x:\ (\zeta,\eta,D_{kl},D_{ik})\mapsto(\beta^2\zeta,\ \beta^2\eta,\ \beta^{-2}D_{kl},\
  \beta^{-1}D_{ik}),
\end{gathered}
\]
{Here $\alpha,\beta\in\C^*$, and coordinates not displayed remain fixed.
Together with translations of $(\hat Z,\hat W)$, these transformations preserve the zero set
of $F_0$: each row changes only by a nonzero factor.
They therefore act on $\mathcal B_4(m)$. We use these boundary symmetries to
identify its tangent spaces and the cokernels of $dF_0$.}

\begin{lemma}\label{lem:d4-fibre}
Let $m\in\Rp^5$ satisfy (23) and let $y\in\mathcal B_4(m)$.
\begin{alphalist}
  \item $D_{12}^3=M_P^{-1}$, $A_2=\rho A_1$ and $B_2=\rho B_1$ with $\rho=-m_1/m_2$,
  $D_{12}^2(A_1-A_2)(B_1-B_2)=1$ and $D_{ik}^2A_i\eta_k=-1$; in particular $A_i,B_i,\eta_k\ne0$.
  The point $(\zeta_3,\zeta_4,\eta_3,\eta_4,D_{34},D_{35},D_{45})$ solves the zero-multiplier
  triangle system, so, after renumbering $T$, it lies on the orbit $O$ of Lemma~\ref{lem:d1-mb}.
  \item $dF_0(y)$ has rank $16$.  Its kernel is spanned by $\partial_{\hat Z}$, $\partial_{\hat W}$
  {and the generators of the two scalings.  The annihilator of its image, equivalently the
dual of its cokernel, is spanned by the linear form}
  \[
    v_P(r)=(B_1-B_2)\,(r^z_1-r^z_2)-(A_1-A_2)\,(r^w_1-r^w_2),
  \]
  where $r^z_i,r^w_i$ are the components of $r\in\C^{21}$ in the rows $t^4E^z_i$, $t^{-4}E^w_i$, and
  by four linear forms that involve only the rows $t^4E^z_k$, $t^4E^w_k$, $\chi_{kl}$ and the gauge.
  \item Near $y$, $\mathcal B_4(m)$ coincides with the orbit of $y$ under the scalings and
  the translations; it is a complex submanifold of dimension $4$ with tangent space
  $\ker dF_0(y)$ at $y$.
\end{alphalist}
\end{lemma}

\begin{proof}
{For (a), the identities $R^z_1-R^z_2=(A_1-A_2)(1-M_PD_{12}^3)$ and
$m_1R^z_1+m_2R^z_2=m_1A_1+m_2A_2$ hold everywhere, as do the corresponding identities for
$R^w$. The distance equation for $12$ gives $A_1\ne A_2$ and $B_1\ne B_2$.
Thus $F_0(y)=0$ implies $D_{12}^3=M_P^{-1}$, $m_1A_1+m_2A_2=0$ and $m_1B_1+m_2B_2=0$.
The crossed distance equations are $D_{ik}^2A_i\eta_k=-1$.
The triangle equations in $F_0$ are $\phi_k,\psi_k,\chi_{kl}$ together with the gauge.
The proof of Proposition~\ref{prop:d1-first} gives the last assertion.

For (b), order the coordinates as $(D_{ik};\,A,B,D_{12};\,u,v,p,q,D_{kl};\,\hat Z,\hat W)$.
Order the rows as the crossed distance equations, the pair equations $R^z_i,R^w_i$
and the distance equation for $12$, followed by the triangle equations.
The pair and triangle equations involve only their respective coordinates.
Each $D_{ik}$ occurs only in its own crossed distance equation, where
$\partial/\partial D_{ik}=-2D_{ik}A_i\eta_k\ne0$.}
{So $dF_0(y)$ is block upper triangular, and its rank is $6+\operatorname{rank}J_P+
\operatorname{rank}J_T$}, {where $J_P$, $J_T$ are the pair and triangle blocks.
We compute their ranks separately.

At $y$, the}
differentials of $R^z_1-R^z_2$ and $R^w_1-R^w_2$ are $-3M_PD_{12}^2$ times
$(A_1-A_2)\,dD_{12}$ and $(B_1-B_2)\,dD_{12}$, those of $m_1R^z_1+m_2R^z_2$ and $m_1R^w_1+m_2R^w_2$
are $m_1dA_1+m_2dA_2$ and $m_1dB_1+m_2dB_2$, and the distance equation for $12$ adds $D_{12}^2$ times
$(B_1-B_2)(dA_1-dA_2)+(A_1-A_2)(dB_1-dB_2)$ modulo $dD_{12}$.
{These relations give $\operatorname{rank}J_P=4$, with kernel spanned by
$(A,-B,0)$. The linear form $v_P$, obtained by differentiating at $y$ the identity}
{$(B_1-B_2)(R^z_1-R^z_2)-(A_1-A_2)(R^w_1-R^w_2)=0$, spans the annihilator of its image.}

{By Lemma~\ref{lem:d1-mb}, the rows $\phi_3,\phi_4,\psi_3,\psi_4,\chi_{kl}$ have rank $5$
at $y$, and their kernel is spanned by the generators $(\zeta,0,-\frac12D)$ and
$(0,\eta,-\frac12D)$ of $\mathbb T$.
The rows $\phi_5$ and $\psi_5$ are redundant because $\sum_km_k\phi_k=\sum_km_k\psi_k=0$.
On the first generator, the gauge row takes the value $\zeta_4-\zeta_3$,
which is nonzero by $\chi_{34}=0$; it takes the value $0$ on $(\zeta,\eta,-D)$.
It follows that $\operatorname{rank}J_T=6$, with kernel spanned by $(\zeta,\eta,-D)$.

The total rank is therefore $16$. Completing the four stated vectors in their $D_{ik}$
components by the crossed distance equations gives a basis of the kernel.
Any linear form annihilating the image has zero coefficients on the crossed distance rows,
as is seen from the columns $D_{ik}$. It is consequently the}
{sum of a form annihilating the image of $J_P$ and one annihilating the image of $J_T$;
the latter annihilator has dimension $10-6=4$.}
{For (c), the differential of the orbit map is onto $\ker dF_0(y)$.
The orbit is thus a submanifold of dimension $4$ near $y$, contained in $\mathcal B_4(m)$.
Conversely, (b) and the rank theorem place $\mathcal B_4(m)$ near $y$ inside a submanifold
of dimension $20-16=4$. The two submanifolds therefore coincide locally.}
\end{proof}

\subsubsection{A Lyapunov--Schmidt lemma}\label{d4-ls}
{We now reduce continuation from the smooth boundary fibre to the successive
Taylor coefficients in the cokernel of $dF_0$. A nonzero coefficient forces
$t=0$ near the boundary point and therefore excludes a germ entering $t\ne0$.
For a polynomial curve $x(t)$, denote by $[F]_j$ the coefficient of $t^j$ in
$F(t,x(t))$.}

\begin{lemma}[Lyapunov--Schmidt]\label{lem:d4-ls}
Let $F(t,x)=\sum_{j\ge0}t^jF_j(x)$ be a polynomial map $\C\times\C^n\to\C^p$, let
$Z=F_0^{-1}(0)$, and let $y\in Z$ be a point near which $Z$ is a complex submanifold with
$T_{y'}Z=\ker dF_0(y')$ for all $y'\in Z$.  For $y'\in Z$ let $Q_{y'}$ be the cokernel of
$dF_0(y')$.  Let $N\ge1$ and suppose that every $y'\in Z$ near $y$ is the value at $t=0$ of a
polynomial curve $x(t)$ with $[F]_j=0$ for $j<N$.
\begin{alphalist}
  \item The class $\Omega_N(y')\in Q_{y'}$ of $[F]_N$ does not depend on the choice of such a curve.
  \item If $\Omega_N(y)\ne0$, no holomorphic germ $c\colon(\C,0)\to(F^{-1}(0),(0,y))$ has
  $t\circ c\not\equiv0$.
\end{alphalist}
For $N=1$ the hypothesis is empty and $\Omega_1(y')$ is the class of $F_1(y')$.
\end{lemma}

\begin{proof}
{Let $K=\ker dF_0(y)$ and $R=dF_0(y)(\C^n)$.
Choose complements $\C^n=K\oplus K'$ and $\C^p=R\oplus R'$, and let $\pi$ project onto $R$
along $R'$. The map $\pi\,dF_0(y)$ sends $K'$ isomorphically onto $R$.
The implicit function theorem therefore gives a holomorphic function $h(t,k)\in K'$,
defined near $(0,0)\in\C\times K$, with $h(0,0)=0$.
Near $(0,y)$, the equation $\pi F(t,y+k+h')=0$, for small $h'\in K'$, holds precisely when
$h'=h(t,k)$. This uniqueness also holds for formal power series in $t$.

The remaining equations are $G(t,k)=F(t,y+k+h(t,k))\in R'$.
Thus, near $(0,y)$, we have $F^{-1}(0)=\{(t,y+k+h(t,k)):G(t,k)=0\}$.
At the boundary, the map $k\mapsto y+k+h(0,k)$ parametrizes $\{\pi F_0=0\}$.
This submanifold contains $Z$ and has the same dimension $\dim K=\dim Z$;
it therefore coincides locally with $Z$. In particular, $G(0,k)=0$.
Write $G=\sum_{j\ge1}t^jG_j(k)$.
For every $y'\in Z$ near $y$, the space $R'$ remains a complement of $dF_0(y')(\C^n)$,
and hence $R'\to Q_{y'}$ is an isomorphism.

We now compare coefficients along a polynomial curve.
Fix $j\le N$ and $y'=y+k'+h(0,k')$, and suppose that $G_1,\dots,G_{j-1}$ vanish near $k'$.
Let $x(t)=y+\kappa(t)+\eta(t)$, with $\kappa(t)\in K$ and $\eta(t)\in K'$,
be a polynomial curve satisfying $x(0)=y'$ and $[F]_i=0$ for $i<j$.
Since $\pi F(t,x(t))=O(t^j)$, formal uniqueness gives
$\eta(t)-h(t,\kappa(t))=e\,t^j+O(t^{j+1})$ for some $e$.
Taylor's formula yields $F(t,x(t))=F(t,y+\kappa+h(t,\kappa))+dF_0(y')(e)\,t^j+O(t^{j+1})$,
while $F(t,y+\kappa+h(t,\kappa))=G(t,\kappa(t))=t^jG_j(k')+O(t^{j+1})$.
Consequently, $[F]_j\equiv G_j(k')$ modulo $dF_0(y')(\C^n)$.

For $j<N$, the curves in the hypothesis give $G_j(k')\in R'\cap dF_0(y')(\C^n)=0$
for all $k'$ near $0$. Induction shows that $G_1,\dots,G_{N-1}$ vanish near $0$.
The case $j=N$ proves (a): $\Omega_N(y')$ is the image of $G_N(k')$.

If $\Omega_N(y)\ne0$, then $G=t^NH$ for a holomorphic $H$ with $H(0,0)=G_N(0)\ne0$.
Thus, near $(0,0)$, the equation $G=0$ forces $t=0$.
Every germ $c$ in (b) has the form $(t,y+k+h(t,k))$ and satisfies $G(t,k)=0$,
so $t\circ c\equiv0$.}
\end{proof}

\subsubsection{The first three obstructions}\label{d4-obstructions}
{We apply the reduction to the chart equations. The first and third coefficients
will impose the two force balances used to exclude the germ.
For $y\in\mathcal B_4(m)$, define $S_i=\sum_km_kD_{ik}$ and}
\[
  \Theta:=\sum_km_k\,\frac{D_{1k}-D_{2k}}{\eta_k}=A_2\Sigma_2-A_1\Sigma_1 .
\]
{The second expression, as well as $f_i=-S_i/A_i$, follows from
$1/\eta_k=-A_iD_{ik}^2$. Under the boundary scalings, $S_i/A_i$ acquires the factors
$\alpha^{-3}$ and $\beta^{-1}$, while $\Theta$ acquires $\alpha^{-1}$ and $\beta^{-3}$.
Neither expression involves $\hat Z$, $\hat W$, $D_{12}$ or $m_1,m_2$.

These quantities describe the leading forces from the triangle.
On body $i$, the crossed forces are $\sum_km_k\delta_{ik}^3(z_i-z_k)=t^{-1}A_i\Sigma_i+O(t^2)$
and $\sum_km_k\delta_{ik}^3(w_i-w_k)=-t^5f_i+O(t^7)$.
Thus $f_1=f_2$ and $\Theta=0$ express equality of the leading forces on the two outer bodies
in the $w$- and $z$-coordinates, respectively.}

\begin{lemma}\label{lem:d4-obstr}
Let $m\in\Rp^5$ satisfy (23), let $y\in\mathcal B_4(m)$, and let $\Omega_N$ be the classes of
Lemma~\ref{lem:d4-ls} for $F$ and $Z=\mathcal B_4(m)$.
{We regard $v_P$ as a linear functional on the cokernel $Q_y$.}
\begin{alphalist}
  \item $\Omega_1(y)=0$ if and only if $f_1=f_2$, that is, $S_1/A_1=S_2/A_2$.
  \item If $f_1=f_2$ at $y$, then the hypothesis of Lemma~\ref{lem:d4-ls} holds at $y$ for
  $N=3$, and
  \[
    v_P\big(\Omega_3(y)\big)=\tfrac32\,(A_1-A_2)\,\frac{B_1}{A_1}\,\Theta .
  \]
\end{alphalist}
\end{lemma}

\begin{proof}
{For (a), Lemma~\ref{lem:d4-fibre}(b) shows that $\Omega_1(y)$, the class of $F_1(y)$,
vanishes precisely when $F_1(y)$ is annihilated by $v_P$ and the four triangle forms.
The triangle components of $F_1$ vanish, while
$v_P(F_1(y))=-(A_1-A_2)(f_1-f_2)$.

For (b), homogeneity implies $f_1=f_2$ throughout the orbit of $y$.
By Lemma~\ref{lem:d4-fibre}(c), this orbit is a neighbourhood of $y$ in $\mathcal B_4(m)$.
Choose a point $y'$ in it and $N\in\{2,3\}$.
Consider a polynomial curve $x(t)$ with $x(0)=y'$ and $F(t,x(t))=O(t^N)$;
part (a) supplies such a curve for $N=2$.
Write $A_i(t),\dots$ for its components and $A_i,\dots$ for their values at $y'$,
and set $\Phi(t)=1-M_PD_{12}(t)^3$.

We first examine the pair equations. Equation~\eqref{eq:d4-rows} gives
$t^4E^z_i=R^z_i+O(t^3)$. Since $R^z_1-R^z_2=(A_1(t)-A_2(t))\Phi(t)$ and
$m_1R^z_1+m_2R^z_2=m_1A_1(t)+m_2A_2(t)$, we obtain $\Phi=O(t^N)$ and
$A_i(t)=A_i(1+a(t))+O(t^N)$ for a common series $a(t)=O(t)$.

Next, the triangle rows of $F$ agree with those of $F_0$ up to $O(t^3)$.
The triangle components of $x(t)$ therefore solve the triangle block of $F_0$ modulo $t^N$.
The proof of Lemma~\ref{lem:d4-fibre}(b),(c) identifies the zero set of that block near $y'$
with the curve $\beta\mapsto\beta\cdot y'$, whose dimension is $\dim\ker J_T$.
Consequently, the triangle components of $x(t)$ agree modulo $t^N$ with a curve on this orbit.
In particular, $\eta_k(t)=\theta(t)\eta_k+O(t^N)$, where $\theta(0)=1$.
For $N=2$, the triangle components of $[F]_2$ therefore belong to the image of $J_T$.

The crossed distance equations give $D_{ik}(t)^2=-1/\big(A_i(t)(\eta_k(t)-t^2B_i)\big)+O(t^N)$,
and hence}
\[
  D_{ik}(t)^3\eta_k(t)=-\frac{D_{ik}(t)}{A_i(t)}\Big(1+t^2\frac{B_i}{\eta_k}\Big)+O(t^N),\qquad
  D_{ik}(t)=\frac{D_{ik}}{\sqrt{(1+a)\theta}}\Big(1+\frac{t^2B_i}{2\eta_k}\Big)+O(t^N),
\]
and, with $f_i(t)=\sum_km_kD_{ik}(t)^3\eta_k(t)$,
\[
  f_i(t)=-\frac{S_i/A_i}{(1+a)^{3/2}\theta^{1/2}}-\frac32\,t^2\,\frac{B_i}{A_i}\sum_km_k\frac{D_{ik}}{\eta_k}
  +O(t^N).
\]
Since $S_1/A_1=S_2/A_2$ and $B_1/A_1=B_2/A_2$, $f_1(t)-f_2(t)=-\frac32\,t^2\,\frac{B_1}{A_1}\,\Theta+O(t^N)$.

{We can now evaluate the obstruction.
Applying $v_P$ to $F(t,x(t))$, with its coefficients fixed at $y'$, gives, by \eqref{eq:d4-rows},}
\[
  \Phi(t)\big[(B_1-B_2)(A_1(t)-A_2(t))-(A_1-A_2)(B_1(t)-B_2(t))\big]
  -(A_1-A_2)\,t\,\big(f_1(t)-f_2(t)\big)-t^3\,{\Delta_4}+O(t^4),
\]
with ${\Delta_4}=(B_1-B_2)(A_1\Sigma_1-A_2\Sigma_2)-(A_1-A_2)(B_1\Sigma_1-B_2\Sigma_2)$, which vanishes
because $A_2=\rho A_1$ and $B_2=\rho B_1$.  The bracket vanishes at $t=0$ and $\Phi=O(t^N)$.

{For $N=2$, this gives $v_P([F]_2)=0$.
Every linear form annihilating the image of $dF_0(y')$ has zero coefficients on the
crossed distance rows. Together with the preceding result for the triangle part,
this gives $\Omega_2(y')=0$. We can therefore correct the curve at order $2$ to obtain}
{a curve with $F(t,x(t))=O(t^3)$.
This proves the hypothesis of the Lyapunov--Schmidt lemma for $N=3$.
For $N=3$, the displayed expression gives
$v_P([F]_3)=-(A_1-A_2)\cdot\big(-\frac32\frac{B_1}{A_1}\Theta\big)$, as claimed.}
\end{proof}

{\subsubsection{Incompatibility of the two force balances}\label{d4-census}
The triangle classification gives a real parametrization of its limiting offsets.
We use it to express the two force balances as scalar equations and then apply
positivity and strict convexity.
After renumbering $T$, write
$(\sqrt{m_3},\sqrt{m_4},\sqrt{m_5})=\lambda(bc,b(b+c),c(b+c))$
with $\lambda,b,c>0$. Exchanging $4$ and $5$ if necessary allows us to assume
$b>c$. The case $b=c$ would give a zero triangle offset, contrary to
Proposition~\ref{prop:d4-cover}. The triangle offsets are proportional to
\[
 (o_3,o_4,o_5)=\Bigl(-\frac{b^2-c^2}{bc},
 \frac{c(2b+c)}{b(b+c)},-\frac{b(b+2c)}{c(b+c)}\Bigr).
\]
Put $a=\sqrt{-o_3/o_4}>0$, $d=\sqrt{o_3/o_5}>0$, and
$(g_3,g_4,g_5)=(1,ia,d)$.  If $x=b/c>1$, then
\[
 a^2=\frac{(x-1)(x+1)^2}{2x+1},\qquad
 d^2=\frac{(x-1)(x+1)^2}{x^2(x+2)}.
\]
In particular $0<d<a$ and $d<1$: the differences of the relevant
denominators are
$x^2(x+2)-(2x+1)=(x-1)(x^2+3x+1)>0$ and
$x^2(x+2)-(x-1)(x+1)^2=x^2+x+1>0$.
The mass-centre relation $\sum_km_ko_k=0$ becomes
\begin{equation}\label{eq:d4-positive-centre}
 \frac{m_4}{a^2}=m_3+\frac{m_5}{d^2}.
\end{equation}
For $s=(s_{13},s_{14},s_{15},s_{23},s_{24},s_{25})\in\{1,-1\}^6$
and $j=1,3$, define
\[
 X_j=\sum_{k\in T}s_{1k}m_kg_k^j,\qquad
 Y_j=\sum_{k\in T}s_{2k}m_kg_k^j.
\]

\begin{lemma}\label{lem:d4-census}
For every sign vector $s$ and every $\nu>0$, the two equations
\[
 \nu^3X_1+iY_1=0,\qquad \nu X_3=iY_3
\]
cannot hold simultaneously.
\end{lemma}

\begin{proof}
Put $r=m_3$, $u=m_4a$, $v=m_5d$ and
$\varepsilon=s_{13}s_{15}$, $\delta=s_{23}s_{25}$.
By \eqref{eq:d4-positive-centre} and $a>d$,
\[
 u=a^3r+\frac{m_5a^3}{d^2}>m_5d=v.
\]
The real and imaginary parts of the first equation give
\[
 \nu^3s_{13}(r+\varepsilon v)=s_{24}u,
 \qquad \nu^3s_{14}u=-s_{23}(r+\delta v).
\]
Thus $|(r+\varepsilon v)(r+\delta v)|=u^2$.
Both factors $r+\varepsilon v$ and $r+\delta v$ must be positive.
Indeed, if a factor is nonpositive then $v\ge r$ and at least one sign
is negative; the absolute product is then either $(v-r)^2$ or
$v^2-r^2$, in both cases strictly less than $v^2<u^2$.
It follows that $s_{24}=s_{13}$ and $s_{14}=-s_{23}$, and hence
\begin{equation}\label{eq:d4-first-real}
 \nu^3(r+\varepsilon v)=u,\qquad
 \nu^3u=r+\delta v.
\end{equation}
Since $g_4^3=-ia^3$, the real and imaginary parts of the second equation
now give
\begin{equation}\label{eq:d4-third-real}
 \nu(r+\varepsilon vd^2)=ua^2,\qquad
 \nu ua^2=r+\delta vd^2.
\end{equation}

If $\varepsilon=\delta=1$, then \eqref{eq:d4-first-real} gives
$\nu=1$ and $u=r+v$.  Put
$w_0=r/(r+m_5/d^2)$ and $w_1=(m_5/d^2)/(r+m_5/d^2)$.
These weights are positive and sum to $1$.
Equations \eqref{eq:d4-positive-centre}, \eqref{eq:d4-first-real} and
\eqref{eq:d4-third-real} give, respectively,
\[
 a^3=w_0+w_1d^3,\qquad a^5=w_0+w_1d^5.
\]
But strict convexity of $t\mapsto t^{5/3}$ on $(0,\infty)$, applied
to $1$ and $d^3\ne1$, gives
$w_0+w_1d^5>(w_0+w_1d^3)^{5/3}=a^5$, a contradiction.

In the other cases \eqref{eq:d4-first-real} gives either
$u=r-v<r$ or $u^2=r^2-v^2<r^2$, so $u<r$.
This forces $a<1$: otherwise \eqref{eq:d4-positive-centre} would give
$r<m_4\le m_4a=u$.  Thus $d<a<1$.
If $\varepsilon=\delta=-1$, then $\nu=1$ and the two balance equations
give $r=u+v=ua^2+vd^2<u+v$, again a contradiction.
Finally, if $\varepsilon=-\delta$, adding the two relations in each
of \eqref{eq:d4-first-real} and \eqref{eq:d4-third-real} yields
\[
 2r=u(\nu^3+\nu^{-3})=ua^2(\nu+\nu^{-1}).
\]
Consequently
$a^2=(\nu^3+\nu^{-3})/(\nu+\nu^{-1})
=\nu^2+\nu^{-2}-1\ge1$, contradicting $a<1$.
\end{proof}

}\subsubsection{Proof of Theorem~\ref{thm:d4}}
{Suppose that $\gamma$ is a Diagram~4 germ at $m\in\Rp^5$.
Proposition~\ref{prop:d4-cover} gives relation (23) and a holomorphic chart representation
$\tau\mapsto(t(\tau),x(\tau))$, with $t=\tau^q$, in $F^{-1}(0)$.
It satisfies $t\not\equiv0$ and $x(0)=y\in\mathcal B_4(m)$.
Lemma~\ref{lem:d4-fibre}(c) permits the application of Lemma~\ref{lem:d4-ls} at $y$,
with $Z=\mathcal B_4(m)$. Taking $N=1$ gives $\Omega_1(y)=0$,
or $f_1=f_2$ by Lemma~\ref{lem:d4-obstr}(a).
Lemma~\ref{lem:d4-obstr}(b) then permits the same argument for $N=3$,
giving $\Omega_3(y)=0$. Since $(A_1-A_2)B_1/A_1\ne0$, it follows that $\Theta=0$ at $y$.}

{Both conditions are symmetric in the bodies of $T$ and homogeneous under
the boundary scalings. We may therefore renumber $T$ as in \S\ref{d4-census} with $b>c$.
Indeed, $b=c$ would contradict $\eta_k(y)\ne0$ in Proposition~\ref{prop:d4-cover},
because $o_3=0$ when $b=c$.
The boundary scalings allow us to normalize $A_1=1$ and $\eta_3=-1$.
Then $\eta_k=-o_k/o_3$ and $D_{1k}^2=-1/\eta_k=g_k^2$.
With $\nu=\sqrt{m_1/m_2}>0$ we have $A_2=-\nu^2$ and
$D_{2k}^2=-D_{1k}^2/\nu^2$, so
$D_{1k}=s_{1k}g_k$ and $D_{2k}=s_{2k}ig_k/\nu$ for a sign vector $s$.
With the original triangle masses in the definitions of $X_j,Y_j$, we obtain
\[
 S_1=X_1,\qquad S_2=iY_1/\nu,\qquad
 \Sigma_1=X_3,\qquad \Sigma_2=-iY_3/\nu^3.
\]
Thus $S_1/A_1=S_2/A_2$ and $A_2\Sigma_2-A_1\Sigma_1=0$ become
$\nu^3X_1+iY_1=0$ and $\nu X_3=iY_3$.
Lemma~\ref{lem:d4-census} gives the contradiction. \qed}

{
\subsection{Diagram 5}\label{sec:d5}

Diagram~5 consists of the isolated $zw$-triangle $345$, with no circled vertex and no
other stroke \cite[\S7.4 and Fig.~11]{AK}.  We write $T=\{3,4,5\}$ and call $1,2$
the outer bodies.  We compare the shrinking triangle with its distances to the outer bodies by
adding the orders of all ten squared distances. Our aim is
$\ord\Delta>0$, which gives the vanishing required in the global argument.

\begin{theorem}\label{thm:d5}
Let $m\in\Rp^5$ and let $\gamma$ be a Diagram~5 germ on an irreducible curve
$\Gamma\subset A_m$.  Then
\[
 \Delta:=\prod_{1\le i<j\le5}r_{ij}^{2}\longrightarrow0
 \quad\hbox{along }\gamma.
\]
\end{theorem}

\subsubsection{Coordinates relative to the triangle}\label{d5-frame}
We normalize as in Lemma~\ref{lem:germ} and measure orders in a Puiseux scale $t$ for
which $\epsilon\approx|t|$ and the least orders of both $\mathcal Z$ and
$\mathcal W$ are $-2$.  If initially $\epsilon\approx|s|^\nu$, we take
$t=cs^\nu$, $c\ne0$, after a finite ramification as necessary.  This choice only
multiplies all orders by a positive factor.  As usual, $\ord0=+\infty$.

Let $M_T=\sum_{k\in T}m_k$, $M=\sum_{k=1}^5m_k$, and let $z_c,w_c$ be the
mass centres of the triangle.  We distinguish the internal coordinates of
the triangle from the positions of the outer bodies relative to its centre:
\[
 \begin{aligned}
 \xi_k&=z_k-z_c,\quad \eta_k=w_k-w_c\quad(k\in T),\\
 X_j&=z_j-z_c,\quad Y_j=w_j-w_c\quad(j=1,2).
 \end{aligned}
\]
Proposition~\ref{prop:d45-first} and Lemma~\ref{lem:d1-shape} give relation~(23)
and show that the limiting triangle has zero multiplier.  After a constant
change of $t$ and a constant rotation $(z,w)\mapsto(az,a^{-1}w)$, its internal
coordinates have the form
\begin{equation}\label{d5-eq-offsets}
 \xi_k=t(s_k+o(1)),\qquad \eta_k=t(s_k+o(1)),\qquad
 s_k\in\R,\quad \sum_{k\in T}m_ks_k=0,
\end{equation}
where the three $s_k$ are pairwise distinct.  Every squared distance, as well
as $U$ and $I$, is preserved by these changes.

The rules and estimates in Fact~\ref{fact:rules} now give the orders that we
shall use.  Each coordinate difference within the triangle has order $1$ in
both colours, so each of its three squared distances has order $2$.
Since there is no circled vertex, all body coordinates, and therefore all
differences from the triangle centre, have order strictly greater than $-2$.
Every crossed
force term $\delta_{jk}^3(z_j-z_k)$ and $\delta_{jk}^3(w_j-w_k)$, $j\in\{1,2\}$,
$k\in T$, has order strictly greater than $-2$, since no stroke joins an outer
body to the triangle.  In particular,
\[
 \ord r_{kl}^2=2\quad(k<l,\ k,l\in T),\qquad
 \ord X_j,\ord Y_j>-2\quad(j=1,2),\qquad
 \ord r_{ij}\le1,
\]
with equality in the last bound exactly on $zw$-edges.

The scale of the triangle is thus fixed.  Relative to this scale, each outer
body belongs to exactly one of the following three types:
\begin{itemize}
 \item \emph{far}: $\ord X_j<1$ and $\ord Y_j<1$;
 \item \emph{near in $z$}: $\ord X_j\ge1$ and $\ord Y_j=\beta_j<1$;
 \item \emph{near in $w$}: the preceding type with the colours exchanged.
\end{itemize}
Indeed, $\ord X_j,\ord Y_j\ge1$ would give $\ord r_{jk}\ge1$
for every $k\in T$, hence additional $zw$-edges. In a near type, the coordinate
of order at least $1$ may vanish identically.

For a body near in $z$, the limit $\theta_j=\lim t^{-1}X_j$ exists.
Equation~\eqref{d5-eq-offsets} shows that at least two of the differences
$z_j-z_k$ have order $1$.  The third can have a higher order $h_j>1$ only
when its leading coefficient vanishes at one triangle vertex.  We record this distinction by the excess order
\[
 e_j:=\begin{cases}
 0,&\theta_j\ne s_k\ \text{for every }k\in T,\\
 h_j-1,&\theta_j=s_k\ \text{for some }k\in T.
 \end{cases}
\]
In the second case, we call the unique such $k$ the contact vertex.
The distance constraints give $z_j-z_k\ne0$, so $h_j<\infty$.  Summing
the three squared-distance orders gives
\begin{equation*}
 \sum_{k\in T}\ord r_{jk}^2=3(1+\beta_j)+e_j
 \qquad\hbox{for a body near in }z.
\end{equation*}

We now separate the forces exerted by the triangle on each outer body from
the mutual force of the two outer bodies.  These are, respectively,
\[
 \begin{gathered}
 F_j=\sum_{k\in T}m_k\delta_{jk}^3(z_j-z_k),\qquad
 G_j=\sum_{k\in T}m_k\delta_{jk}^3(w_j-w_k),\\
 P=\delta_{12}^3(X_1-X_2),\qquad H=\delta_{12}^3(Y_1-Y_2).
 \end{gathered}
\]
In these relative coordinates, the original equations and the total
mass-centre identity become
\begin{equation}\label{d5-eq-outer}
 \begin{aligned}
 X_1-X_2&=(m_1+m_2)P+F_1-F_2,\\
 X_1&=m_2P+(1+m_1/M_T)F_1+(m_2/M_T)F_2,\\
 X_2&=-m_1P+(m_1/M_T)F_1+(1+m_2/M_T)F_2.
 \end{aligned}
\end{equation}
The same identities hold with $(X,P,F)$ replaced by $(Y,H,G)$.  Their weighted
sums are
\begin{equation}\label{d5-eq-centre}
 m_1X_1+m_2X_2=\frac M{M_T}(m_1F_1+m_2F_2),\qquad
 m_1Y_1+m_2Y_2=\frac M{M_T}(m_1G_1+m_2G_2).
\end{equation}
For example, multiplying the two original outer $z$-equations by $m_1,m_2$
and adding cancels their mutual force.  The mass-centre identity gives
$m_1z_1+m_2z_2=(M_T/M)(m_1X_1+m_2X_2)$, which proves the first identity.
The coefficients of $(F_1,F_2)$ in the last two equations of
\eqref{d5-eq-outer} form the matrix
\begin{equation}\label{d5-eq-force-matrix}
 C=\begin{pmatrix}1+m_1/M_T&m_2/M_T\\m_1/M_T&1+m_2/M_T\end{pmatrix},
 \qquad
 C^{-1}=\frac1M\begin{pmatrix}M_T+m_2&-m_2\\-m_1&M_T+m_1\end{pmatrix}.
\end{equation}
Its determinant is $M/M_T>0$.  In particular, leading terms of the two
triangle forces cannot cancel in both outer-body equations simultaneously.

Throughout this subsection put
\[
 a=\ord(X_1-X_2),\qquad b=\ord(Y_1-Y_2),\qquad S=a+b,
 \qquad d=\ord\Delta.
\]
Because $12$ is an actual edge, $a,b$ are finite.  The distance constraints give
\begin{equation}\label{d5-eq-pair-orders}
 \ord P=a-\frac32S=-\frac{a+3b}{2},\qquad
 \ord H=b-\frac32S=-\frac{3a+b}{2}.
\end{equation}
If the coordinate differences on a crossed edge have orders $h,v$, its
contributions to $F_j,G_j$ have orders $-(h+3v)/2$ and $-(3h+v)/2$,
respectively.  Cancellation in the sum can only increase the order of a
force.  When $e_j>0$ and body $j$ is near in $z$, the edge to its contact vertex gives
the unique term of lowest order in each force.  Hence
\begin{equation}\label{d5-eq-special-force}
 \ord F_j=-\frac{h_j+3\beta_j}{2},\qquad
 \ord G_j=-\frac{3h_j+\beta_j}{2}.
\end{equation}

When both bodies are far, we shall also need the inertia.
Lemma~\ref{lem:U-constant} shows that $I$ is constant on every irreducible
complex solution curve.  Its decomposition relative to the triangle centre is
\begin{equation}\label{d5-eq-inertia}
 \begin{gathered}
 I=I_T+J,\qquad I_T=\sum_{k\in T}m_k\xi_k\eta_k=O(t^2),\\
 J=\frac{m_1m_2(X_1-X_2)(Y_1-Y_2)+m_1M_TX_1Y_1+m_2M_TX_2Y_2}{M}.
 \end{gathered}
\end{equation}
Since the internal inertia tends to zero, $J$ is bounded.

\subsubsection{The leading external force on the triangle}\label{d5-mb}
When both outer bodies are near in the same colour, the triangle equations
resolve the possible cancellation of their forces. We shall prove that the
leading external force is equal at the three vertices. Subtracting the
mass-weighted mean then removes this force from the internal equations.

To study the internal forces, put $\hat x=\xi/t$, $\hat y=\eta/t$, and let
\[
 \widehat U(\hat x,\hat y)=\sum_{k<l,\ k,l\in T}m_km_l
 \bigl((\hat x_k-\hat x_l)(\hat y_k-\hat y_l)\bigr)^{-1/2},
\]
where the inverse-distance branches are those selected by the germ.
Let $\phi,\psi$ denote the reduced triangle forces in the $z,w$ directions.
Differentiation and homogeneity give
\begin{equation}\label{d5-eq-potential-identities}
 d\widehat U=-\frac12\sum_km_k(\psi_k\,d\hat x_k+\phi_k\,d\hat y_k),
 \qquad
 \widehat U=\sum_km_k\hat x_k\psi_k=\sum_km_k\hat y_k\phi_k.
\end{equation}
In particular, both forces tend to zero at $(\hat x,\hat y)=(s,s)$.

The next estimate compares the orders of the two reduced forces. It will force
the leading external force to vanish after subtraction of its mean.

\begin{lemma}\label{d5-lem-mb}
For a germ of reduced triangles tending to $(s,s)$, put
$\kappa_\psi=\ord\psi$ and $\kappa_\phi=\ord\phi$, where the order of a vector is
the least order of its components.  If $\kappa_\psi<\kappa_\phi$, then
$\kappa_\phi\le2\kappa_\psi$.  The assertion also holds with the two forces
exchanged.
\end{lemma}
\begin{proof}
Translations and independent scalings in the two colours preserve the
zero-multiplier equations.  The resulting orbit is
\[
 O=\{(As+u\mathbf1,Bs+v\mathbf1):A,B\ne0\}.
\]
On this orbit $d\widehat U=0$ and $\widehat U=0$ by
\eqref{d5-eq-potential-identities}.  Let $d_{kl}$ be the inverse distances at
$(s,s)$ and define the symmetric matrix
\[
 L_{kl}=-m_km_ld_{kl}^3\quad(k\ne l),\qquad
 L_{kk}=\sum_{l\ne k}m_km_ld_{kl}^3.
\]
The relations $L\mathbf1=Ls=0$ express the translation and scaling
directions.  Since its off-diagonal entries are nonzero, the matrix has
rank one and kernel $\operatorname{span}\{\mathbf1,s\}$.  Explicitly,
$L=c\,nn^{\mathsf T}$, $c\ne0$, for
$n=(s_4-s_5,s_5-s_3,s_3-s_4)^{\mathsf T}$, so $\ell:=L_{33}\ne0$.
At $(As,Bs)$, choose $\varrho^2=AB$ consistently with the branches.  Differentiating
each potential term twice gives
\[
 \nabla^2\widehat U=\frac1{4\varrho}
 \begin{pmatrix}3A^{-2}L&(AB)^{-1}L\\(AB)^{-1}L&3B^{-2}L\end{pmatrix}.
\]
Its kernel is precisely the tangent space of $O$.  The remaining two
directions describe changes of shape.  We may choose the $\hat x_3,\hat y_3$
directions as a complement because $s_4\ne s_5$.  On this complement, the
Hessian $K$ satisfies
\[
 \det K=\frac{\ell^2}{2\varrho^6}\ne0,\qquad
 (K^{-1})_{11}=\frac{3A^2\varrho}{2\ell}\ne0,
 \qquad (K^{-1})_{22}=\frac{3B^2\varrho}{2\ell}\ne0.
\]
A nearby triangle can therefore be written as
$o(u)+v_1e_3^x+v_2e_3^y$, with $o(u)\in O$.
Expanding in the two transverse directions gives
$\widehat U=\frac12v^{\mathsf T}K(u)v+O(|v|^3)$ and
$\partial_v\widehat U=K(u)v+O(|v|^2)=-\frac{m_3}{2}(\psi_3,\phi_3)$.
Here $v=(v_1,v_2)$ measures the departure from the zero-multiplier orbit;
$u$ may vary along the germ.  Since $K(u(t))$ tends to an invertible matrix,
$v(t)$ and $\partial_v\widehat U$ have the same order.  The Taylor estimates
are uniform for nearby orbit coordinates.
Consequently
\[
 \ord\widehat U\ge2\ord v
 =2\ord\partial_v\widehat U\ge2\min(\kappa_\psi,\kappa_\phi).
\]
Suppose $\kappa_\psi<\kappa_\phi$.  Then $0<\kappa_\psi<\infty$.  The leading
coefficient of $\psi$ satisfies both $\sum_km_k\psi_{k,0}=0$ and
$\sum_km_ks_k\psi_{k,0}=0$, the latter by
\eqref{d5-eq-potential-identities} and $\ord\widehat U>\kappa_\psi$.
These two relations imply that its third component is nonzero: otherwise
$s_4\ne s_5$ would force the other two components to vanish as well.
Thus the transverse gradient has leading term
$t^{\kappa_\psi}(\chi,0)$, $\chi\ne0$.  Substituting it into
\[
 \widehat U=\frac12(\partial_v\widehat U)^{\mathsf T}K(u)^{-1}
 (\partial_v\widehat U)+O(|v|^3)
\]
and using $(K^{-1})_{11}\ne0$ gives $\ord\widehat U=2\kappa_\psi$.
The last identity in \eqref{d5-eq-potential-identities} also gives
$\ord\widehat U\ge\kappa_\phi$, which proves the assertion.
The same argument applies if one force vanishes identically.
\end{proof}

\begin{lemma}[Equality of the leading external forces]\label{d5-lem-uniform}
Suppose both outer bodies are near in $z$.  Let $\kappa_*$ be the least order of
the six individual terms $t^2m_j\delta_{jk}^3(w_k-w_j)$, $j\in\{1,2\}$, $k\in T$.
After summing over $j$, the coefficient of $t^{\kappa_*}$ is the same at all
three triangle vertices.  The same assertion holds with the colours exchanged.
\end{lemma}
\begin{proof}
Write $\ord(z_j-z_k)=1+e_{jk}$, $e_{jk}\ge0$, and
$\ord(w_j-w_k)=\beta_j\in(-2,1)$.  The orders of the reduced forcing terms in
the two colours are
\[
 \kappa_{jk}=\frac{1-\beta_j-3e_{jk}}2,
 \qquad
 \frac{3-3\beta_j-e_{jk}}2
 =2\kappa_{jk}+\frac{1-\beta_j+5e_{jk}}2>2\kappa_{jk}.
\]
The absence of crossed strokes gives $\kappa_*>0$.  At least two edges from each outer body satisfy $e_{jk}=0$, and $\beta_j>-2$, so also $\kappa_*<3/2$.
Let $g_k=t^2\sum_jm_j\delta_{jk}^3(w_k-w_j)$ and let $g^z$ be the analogous
$z$-forcing.  A hat on either forcing denotes subtraction of its mass-weighted
mean over $T$.  The exact centred triangle equations are
\[
 \psi=t^3\hat y-\widehat g,\qquad
 \phi=t^3\hat x-\widehat {g^z}.
\]
If the coefficient of $t^{\kappa_*}$ in $g$ were not constant across the
vertices, it would survive subtraction of the mean.  Then
$\kappa_\psi=\kappa_*$, while every term contributing to $\phi$ has order
strictly greater than $2\kappa_*$, including $t^3\hat x$.
This contradicts Lemma~\ref{d5-lem-mb}.
\end{proof}

\begin{lemma}[Support of the leading external force]\label{d5-lem-support}
For a single body near in $z$ with $e_j=0$, its three lowest $w$-forcing coefficients
are not all equal. If $e_j>0$, its lowest forcing is supported at its contact
vertex. If both bodies are near in $z$ and $e_1,e_2>0$, their contact vertices
coincide, their lowest forcings cancel there, and
\begin{equation}\label{d5-eq-special-balance}
 3h_1+\beta_1=3h_2+\beta_2.
\end{equation}
\end{lemma}
\begin{proof}
When $e_j=0$,
$z_j-z_k=t(\theta_j-s_k+o(1))$ with all $\theta_j-s_k\ne0$.
The squares of the three forcing coefficients equal a common nonzero constant times
$(\theta_j-s_k)^{-3}$.  If the coefficients were equal, the three distinct
numbers $\theta_j-s_k$ would have equal nonzero cubes and hence would be all
three cube roots of one number.  Their sum would be zero, making
$\theta_j=(s_3+s_4+s_5)/3$ real.  Injectivity of the real cube would then force
the three numbers to coincide, a contradiction.

When $e_j>0$, the edge to the contact vertex gives the unique lowest-order
force term. If $e_1,e_2>0$, unequal lowest orders would leave the lowest coefficient
supported at one vertex. Equal orders with different contact vertices would
leave it supported at two vertices.  Both are forbidden by
Lemma~\ref{d5-lem-uniform}: a coefficient supported at one or two vertices
cannot be constant at all three.  Therefore their vertices and orders agree,
and their leading coefficients cancel at the common vertex.  The unscaled order is
$-(3h_j+\beta_j)/2$, which proves \eqref{d5-eq-special-balance}.
\end{proof}

\subsubsection{Two far bodies}\label{d5-far-far}
We first consider the case in which both outer bodies are far.
Let $u_j=\ord X_j<1$, $v_j=\ord Y_j<1$, and
$A=u_1+v_1$, $B=u_2+v_2$.  All squared distances from body $1$ to the triangle have order
$A$, and those from body $2$ have order $B$, so
\begin{equation}\label{d5-eq-far-product}
 d=6+3A+3B+S.
\end{equation}
Individual force terms give
\[
 \ord F_1\ge u_1-\frac32A,\quad \ord G_1\ge v_1-\frac32A,
 \qquad
 \ord F_2\ge u_2-\frac32B,\quad \ord G_2\ge v_2-\frac32B.
\]
Thus $A<0$ makes the first two force orders strictly greater than $u_1,v_1$,
and $B<0$ gives the analogous conclusion for body $2$.

To show that the product has positive order, we exclude negative orders
among these three groups of distances.  The three terms in $J$ in
\eqref{d5-eq-inertia} have orders $S,A,B$ and nonzero mass coefficients.
Since $J$ is bounded, a negative minimum must occur at least twice.
Thus, if any order is negative, either
$A,B<0$, or, after exchanging the outer bodies, $A\ge0>B$ and $S=B<0$.
Indeed $S<0$ alone would be a unique negative minimum.

Suppose first that $A,B<0$.  In the first equation of \eqref{d5-eq-centre} the
right side has order greater than $\min(u_1,u_2)$, so
\[
 u_1=u_2=u,\qquad m_1\widehat X_1+m_2\widehat X_2=0.
\]
Here the hats denote nonzero leading coefficients.  Since $m_1+m_2>0$, they
are different and $a=u$.  The $w$ equation similarly gives $v_1=v_2=v$ and
$b=v$.  Hence $A=B=S=u+v<0$.  In the first equation of
\eqref{d5-eq-outer}, the left side has order $u$, whereas $P,F_1,F_2$ all have
orders greater than $u$, a contradiction.

Suppose instead that $A\ge0>B$ and $S<0$.  The original equation
$z_2=-m_1P+F_2$ and the total mass centre give
\[
 (m_1+M_T)X_2-m_1X_1=M(-m_1P+F_2).
\]
Its right side has order greater than
$\min(a,u_2)=\min(u_1,u_2)$.  The equality follows directly from
$X_1-X_2$: if $u_1\ne u_2$, then $a=\min(u_1,u_2)$; otherwise $a\ge u_2$.
Consequently $u_1=u_2=u$ and
$(m_1+M_T)\widehat X_2=m_1\widehat X_1$.  Since $M_T>0$, the leading
coefficients differ and $a=u$.  The same argument in $w$ gives
$v_1=v_2=v$ and $b=v$.  This forces $A=B=S$, contradicting $A\ge0>B$.
Exchanging $A,B$ treats the remaining possibility.
Thus $A,B,S\ge0$, and \eqref{d5-eq-far-product} gives $d\ge6>0$.

\subsubsection{One far body and one near body}\label{d5-far-near}
By exchanging the outer bodies and the colours if necessary, we may suppose
that body $1$ is far and body $2$ is near in
$z$.  Write
\[
 u=\ord X_1<1,\quad v=\ord Y_1<1,\quad
 \alpha=\ord X_2\ge1,\quad \beta=\ord Y_2<1.
\]
The orders $u,v,\beta$ are finite and greater than $-2$; $\alpha$ may be infinite.
Here $a=u$.  Put $c=\ord(Y_1-Y_2)\ge\min(v,\beta)>-2$,
$x=u+v$ and $y=1+\beta$, and let $e\ge0$ be the excess order of body $2$.
Then $S=u+c$ and
\begin{equation}\label{d5-eq-mixed-product}
 d=6+S+3x+3y+e.
\end{equation}
If $x\ge0$, then $u>-1$, $S>-3$ and $y>-1$, so $d>0$.

It remains to consider $x<0$.  In this case the triangle forces on body $1$
have higher order than its coordinates:
\[
 \ord F_1\ge u-\frac32x>u,\qquad
 \ord G_1\ge v-\frac32x>v.
\]
The first equation of \eqref{d5-eq-centre} has left side of order $u$, hence
$\ord F_2=u$.  The equation for $X_2$ in \eqref{d5-eq-outer} then forces
$\ord P=u$, since its left side and its $F_1$ term have higher orders.
But $\ord P=-(u+3c)/2$, so $c=-u$ and $S=0$.
As $v<c$, cancellation in $Y_1-Y_2$ requires $\beta=v$ and equal nonzero
leading coefficients in $Y_1,Y_2$.  Positive masses therefore make
$m_1Y_1+m_2Y_2$ have order $v$.  The second equation of
\eqref{d5-eq-centre} now implies $\ord G_2=v$.

If $e>0$, \eqref{d5-eq-special-force} gives
\[
 u=-\frac{h+3v}{2},\qquad v=-\frac{3h+v}{2}.
\]
These identities give $v=-h$ and $u=h>1$, a contradiction.
Thus $e=0$ and
$u=\ord F_2\ge-(1+3v)/2$.  Since $u<1$, this also gives $v>-1$.
Using $S=0$, $\beta=v$ in \eqref{d5-eq-mixed-product}, we obtain
\[
 d=9+3u+6v\ge\frac{15+3v}{2}>6>0.
\]

\subsubsection{Two near bodies of opposite colours}\label{d5-opposite-near}
Let body $1$ be near in $z$ and body $2$ near in $w$, and write
\[
 \ord X_1\ge1,\quad v=\ord Y_1<1,\qquad
 u=\ord X_2<1,\quad \ord Y_2\ge1.
\]
Here $u,v>-2$, $a=u$, $b=v$ and $S=u+v$.  If $e_1,e_2\ge0$ are the two
excess orders, summing the ten squared-distance orders gives
\begin{equation*}
 d=12+4S+e_1+e_2.
\end{equation*}
Suppose $d\le0$.  Then $S\le-3$, and $u,v>-2$ imply $u,v<-1$.
Equation~\eqref{d5-eq-pair-orders} gives $\ord P>u$ and $\ord H>v$.
The mutual forces therefore do not contribute at the leading coordinate
orders.  In matrix form, the outer-body equations read
\[
 C\binom{F_1}{F_2}=\binom{X_1-m_2P}{X_2+m_1P},\qquad
 C\binom{G_1}{G_2}=\binom{Y_1-m_2H}{Y_2+m_1H}.
\]
The right sides have leading terms $(0,X_2)^{\mathsf T}$ at order $u$ and
$(Y_1,0)^{\mathsf T}$ at order $v$, respectively.  Using
\eqref{d5-eq-force-matrix} yields
\[
 F_1\sim-\frac{m_2}{M}X_2,\quad F_2\sim\frac{M_T+m_1}{M}X_2,
 \qquad
 G_1\sim\frac{M_T+m_2}{M}Y_1,\quad G_2\sim-\frac{m_1}{M}Y_1.
\]
All four coefficients are nonzero, so $\ord F_1=\ord F_2=u$ and
$\ord G_1=\ord G_2=v$.
If $e_1=0$, then $\ord G_1\ge-(3+v)/2>v$, a contradiction.
If $e_1>0$, then
$v=-(3h_1+v)/2$ gives $v=-h_1$, whence
$\ord F_1=-(h_1+3v)/2=h_1>1$, also a contradiction.
Thus $d>0$.

\subsubsection{Two near bodies of the same colour}\label{d5-same-near}
It remains to consider two bodies near in $z$; the other colour is symmetric.
Put $\beta_j=\ord Y_j\in(-2,1)$ and retain $e_j\ge0$ for their excess
orders.  Now $a\ge1$, $b\ge\min(\beta_1,\beta_2)$, so $S>-1$, and
\begin{equation*}
 d=12+3(\beta_1+\beta_2)+S+e_1+e_2.
\end{equation*}
Suppose $d\le0$.  If $S\ge0$, the right side is positive since $\beta_j>-2$.
If, for example, $\beta_1\ge-1$, then $\beta_2>-2$ and $S>-1$ make it greater
than $2$.  Therefore
\begin{equation}\label{d5-eq-negative-orders}
 S<0,\qquad \beta_1<-1,\qquad \beta_2<-1.
\end{equation}

If $e_1=e_2=0$, then $\ord G_j\ge-(3+\beta_j)/2>\beta_j$.
The $w$ equation in \eqref{d5-eq-centre} rules out unequal $\beta_1,\beta_2$,
since its left side would have a unique lowest term.  Thus
$\beta_1=\beta_2=\beta$; writing $Y_j=v_jt^\beta(1+o(1))$, it gives
$m_1v_1+m_2v_2=0$.  Positivity makes $v_1\ne v_2$, hence $b=\beta$.
By $S<0$, $\ord H=b-3S/2>\beta$.  Both outer $w$-equations in
\eqref{d5-eq-outer} now have right sides of order greater than $\beta$, whereas
their left sides have order $\beta$, a contradiction.

Suppose next that $e_1>0$ and $e_2=0$, interchanging the bodies if needed,
and put $h=1+e_1>1$.  If the lowest individual $w$-forcing orders differed, only
one body would contribute at the lowest order.  By
Lemma~\ref{d5-lem-support}, its coefficient vector could not be constant
across the triangle vertices in either case $e_j=0$ or $e_j>0$.
Lemma~\ref{d5-lem-uniform} therefore requires the two lowest orders to agree:
\[
 3h+\beta_1=3+\beta_2,\qquad \beta_1<\beta_2.
\]
Their common unscaled order is $q=-(3h+\beta_1)/2=-(3+\beta_2)/2$.
Equation~\eqref{d5-eq-negative-orders} implies
\[
 h+\beta_1=\frac{3+\beta_2+2\beta_1}{3}<0,\qquad
 q-\beta_1=-\frac32(h+\beta_1)>0.
\]
Both $G_j$ have order at least $q$.  But the $w$ equation of
\eqref{d5-eq-centre} has a unique lowest term of order $\beta_1$ on the left,
and right side of order at least $q>\beta_1$, a contradiction.

Finally suppose $e_1,e_2>0$. Lemma~\ref{d5-lem-support} gives a
common contact vertex and $3h_1+\beta_1=3h_2+\beta_2$.
If $h_1=h_2=h$, then $\beta_1=\beta_2=\beta$, $a\ge h$ and $b\ge\beta$.
Thus $h+\beta\le S<0$ and
\[
 \ord G_j=-\frac{3h+\beta}{2}
 =\beta-\frac32(h+\beta)>\beta,\qquad
 \ord H=b-\frac32S>b\ge\beta.
\]
The outer $w$-equations again give a contradiction.
If $h_1\ne h_2$, interchange the bodies to arrange $h_1<h_2$, and put
$\eta=h_2-h_1>0$.  Then
\[
 \beta_1=\beta_2+3\eta,\qquad
 a=h_1=h_2-\eta,\qquad b=\beta_2.
\]
Here $a=h_1$ follows by subtracting the two actual differences to their common
contact vertex. Both $G_j$ have the same exact order
$q=-(3h_2+\beta_2)/2$, while $\ord H=\beta_2-3S/2>\beta_2$.
If $q>\beta_2$, the second outer $w$-equation has no term on the right of order
$\beta_2$.  If $q<\beta_2$, both left sides and $H$ have orders greater than
$q$, so the two nonzero leading coefficients of $G_1,G_2$ would be annihilated
by the invertible matrix $C$, also impossible.  Thus $q=\beta_2$, giving
$\beta_2=-h_2$ and $S=-\eta$.
The exact force orders in \eqref{d5-eq-special-force} and \eqref{d5-eq-pair-orders} now give
\[
 \ord F_1=h_2-4\eta,\qquad
 \ord F_2=h_2,\qquad \ord P=h_2+\frac\eta2.
\]
In the first equation of \eqref{d5-eq-outer}, $F_1$ is the unique lowest term, forcing
$a=h_2-4\eta$, contrary to $a=h_2-\eta$.
Hence $d>0$ also when both bodies are near in the same colour.

\subsubsection{Completion of the proof}
Up to exchanging the outer bodies and the colours, the four cases cover
all possible positions of the outer bodies.  They show that every Diagram~5
germ on a complex solution curve has $\ord\Delta>0$, proving
Theorem~\ref{thm:d5}.

\qed
}

\section{Diagram 6}\label{sec:d6}

We first represent every Diagram~6 germ in a chart with finite boundary limits.
A torque identity restricts the masses and the signs of the crossed inverse
distances. We then classify the corresponding boundary points and use exact
computer-assisted Taylor calculations to prove that they are isolated. This
excludes Diagram~6 and establishes the case $r=6$ of
Theorem~\ref{thm:exclusion}(a).

\subsection{The diagram and the butterfly chart}

{
Diagram~6 is the \emph{butterfly} of \cite[\S7.4]{AK}: a $z$-triangle and a
$w$-triangle, called the \emph{wings}, meet at one vertex.  We use the numbering
of \cite[\S8.6]{AK}.  The $z$-wing $\{1,2,5\}$ has strokes $12,15,25$ and
$z$-circles at $1,2$; the $w$-wing $\{3,4,5\}$ has strokes $34,35,45$ and
$w$-circles at $3,4$.  Vertex $5$ is uncircled, and there is no stroke between
$\{1,2\}$ and $\{3,4\}$.

The three possible circlings of the butterfly give Diagrams~6--8;
$5$ is uncircled in each of them \cite[\S7.4]{AK} (Fact~\ref{fact:fig11}).
Diagram~6 is the case in which both wings are circled.  Rules~1a and~1d place
the circles of the $z$-wing exactly at $1,2$, and give the corresponding
statement for $w$.
}

{
Permuting the bodies together with their masses, or exchanging
$z\leftrightarrow w$, preserves \eqref{eq:system}.  Thus the thirty coloured
labels of Diagram~6, specified by a common vertex and an ordered pair of
wings, are obtained from the numbering above.  The exchange
$z\leftrightarrow w$ composed with $(13)(24)$ preserves this numbering.
It suffices to treat this case at arbitrary positive masses.

The two circled wings have opposite coordinate scales.  We use a chart
adapted to these scales, in which a Diagram~6 germ has a finite limit.
We express the chart coordinates through rotation-invariant combinations,
which allows us to fix the two triangle scales directly.
}

\begin{definition}\label{d6-def-chart}
{
The \emph{butterfly chart} has $21$ coordinates: $t$, $Z_1,\dots,Z_5$,
$W_1,\dots,W_5$, and $D_{kl}$.  They are related to \eqref{eq:system} by
$(z_k,w_k)=(t^{-1}Z_k,tW_k)$ for $k=1,2$,
$(z_k,w_k)=(tZ_k,t^{-1}W_k)$ for $k=3,4$, and
$(z_5,w_5)=(tZ_5,tW_5)$.  On the \emph{crossed pairs} $13,14,23,24$ we set
$\delta_{kl}=tD_{kl}$; on the \emph{wing pairs} $12,15,25,34,35,45$ we set
$\delta_{kl}=D_{kl}$.

The defining equations, also called the \emph{rows}, are obtained as follows.
For each body, we substitute these expressions into the $z$-equation
$z_k-\sum_{l\ne k}m_l\delta_{kl}^3(z_k-z_l)=0$ and the corresponding equation
for $w$.  We multiply the two equations by $(t,t^{-1})$ for $k=1,2$,
by $(t^{-1},t)$ for $k=3,4$, and by $(t,t)$ for $k=5$; the resulting rows are
$R^z_k,R^w_k$.  We also retain the ten distance equations
$\delta_{kl}^2z_{kl}w_{kl}-1=0$ and impose the gauges
$Z_2-Z_1-1=0$, $W_4-W_3-1=0$.

We denote these $22$ rows by $F=0$ and call
$\mathcal B=F^{-1}(0)\cap\{t=0\}$ the \emph{boundary fibre}.
The notation $Z_k,W_k$ for chart coordinates is distinct from the force
terms $Z_{kl}=z_{kl}\delta_{kl}^3$ in \cite[\S3]{AK}.
}
\end{definition}

{\begin{lemma}\label{d6-lem-chart}
Let $m\in\Rp^5$.
\begin{romanlist}
  \item The rows are polynomials in the $21$ coordinates, of degree at most $5$ in $t$.
  \item Let $(z,w,\delta)\in A_m$ (so $z_{12}w_{34}\ne0$), choose $t$ with
  $t^2=(z_{12}w_{34})^{-1}$ and put $\nu=(tz_{12})^{-1}$.  The chart point of the rotated
  solution $(\nu z,\nu^{-1}w,\delta)$ has $Z_k=z_k/z_{12}$, $W_k=w_kz_{12}$ $(k=1,2)$,
  $Z_k=z_kw_{34}$, $W_k=w_k/w_{34}$ $(k=3,4)$, $Z_5=z_5w_{34}$, $W_5=w_5z_{12}$,
  $D_{kl}=\delta_{kl}/t$ on crossed pairs and $D_{kl}=\delta_{kl}$ on wing pairs, and satisfies
  $F=0$: each coordinate row is $t^{\pm1}\nu^{\pm1}$ times an equation of \eqref{eq:system}, the
  distance rows are those of \eqref{eq:system}, and both gauges hold.
  \item At $t=0$ the crossed distance rows read $D_{kl}^2Z_kW_l=-1$, and the wing distance rows
  read $D_{12}^2(Z_2-Z_1)(W_2-W_1)=1$, $-D_{15}^2Z_1(W_5-W_1)=1$, $-D_{25}^2Z_2(W_5-W_2)=1$,
  $D_{34}^2(Z_4-Z_3)(W_4-W_3)=1$, $-D_{35}^2(Z_5-Z_3)W_3=1$, $-D_{45}^2(Z_5-Z_4)W_4=1$.  So
  every $D_{kl}$ is nonzero on $\mathcal B$.
  \item Identically,
  $m_1R^z_1+m_2R^z_2+m_5R^z_5+t^2(m_3R^z_3+m_4R^z_4)=m_1Z_1+m_2Z_2+t^2(m_3Z_3+m_4Z_4+m_5Z_5)$,
  and likewise after $z\leftrightarrow w$, $1\leftrightarrow3$, $2\leftrightarrow4$.  Hence on
  $\mathcal B$: $Z_1=-m_2/(m_1+m_2)$, $Z_2=m_1/(m_1+m_2)$, $W_3=-m_4/(m_3+m_4)$,
  $W_4=m_3/(m_3+m_4)$.
  \item At $t=0$ the rows $R^z_1,R^z_2,R^z_5,R^w_1,R^w_2$, the distance rows of $12,15,25$ and
  the first gauge involve only the \emph{$z$-wing coordinates}
  $Z_1,Z_2,W_1,W_2,W_5,D_{12},D_{15},D_{25}$; the remaining rows except the crossed distance rows
  involve only the \emph{$w$-wing coordinates} $Z_3,Z_4,Z_5,W_3,W_4,D_{34},D_{35},D_{45}$; and
  $D_{13},D_{14},D_{23},D_{24}$ occur only in the crossed distance rows.
\end{romanlist}
\end{lemma}}
{\begin{proof}
Put $I=\{1,2\}$ and $J=\{3,4\}$, and let $i'$ be the other element of $I$.
Substitution in the defining equations gives
\begin{align*}
R_i^z={}&Z_i-m_{i'}D_{ii'}^3(Z_i-Z_{i'})
 -m_5D_{i5}^3(Z_i-t^2Z_5)
 -t^3\sum_{j\in J}m_jD_{ij}^3(Z_i-t^2Z_j),\\
R_i^w={}&W_i-m_{i'}D_{ii'}^3(W_i-W_{i'})
 -m_5D_{i5}^3(W_i-W_5)
 -t\sum_{j\in J}m_jD_{ij}^3(t^2W_i-W_j),\\
R_5^z={}&t^2Z_5-\sum_{i\in I}m_iD_{i5}^3(t^2Z_5-Z_i)
 -t^2\sum_{j\in J}m_jD_{j5}^3(Z_5-Z_j).
\end{align*}
The other coordinate rows follow by exchanging the wings and $z,w$.
The distance rows for $ii'$, $i5$, $j5$, and $ij$ are, respectively,
\begin{gather*}
D_{ii'}^2(Z_i-Z_{i'})(W_i-W_{i'})-1,\qquad
D_{i5}^2(Z_i-t^2Z_5)(W_i-W_5)-1,\\
D_{j5}^2(Z_j-Z_5)(W_j-t^2W_5)-1,\qquad
D_{ij}^2(Z_i-t^2Z_j)(t^2W_i-W_j)-1,
\end{gather*}
and the $34$ row is the wing exchange of the $12$ row.
These formulas prove (i), (iii), and (v). For (ii), the displayed change of
variables is precisely the substitution defining these rows, with
$t^2z_{12}w_{34}=1$ and $\nu tz_{12}=1$; the two gauges therefore hold.
Finally, in the weighted sum of the original $z$-equations every pair force
cancels with the force of the reversed pair. Multiplying this identity by
$t$ gives (iv). At $t=0$, its first identity and the first gauge give
$Z_1=-m_2/(m_1+m_2)$ and $Z_2=m_1/(m_1+m_2)$; the other wing is identical.
\end{proof}

}\subsection{Limits in the butterfly chart}

The diagram rules determine the orders of the coordinate differences.
We use these orders to prove that every germ reaches the boundary fibre
with nonzero inverse-distance coordinates.

\begin{proposition}[Limits in the chart]\label{prop:d6-cover}
Let $m\in\Rp^5$ and let $\gamma(s)$, $s\to0$, be a Diagram-6 germ, in the numbering above, on a
curve in $A_m$.  Along $\gamma$ define $t$ and $\nu$ as in Lemma~\ref{d6-lem-chart}(ii), and
let $\tilde\gamma(s)$ be the chart point of the rotated solution.  Then $t\to0$, every chart
coordinate is bounded along $\tilde\gamma$, and $\tilde\gamma(s)$ converges to a point
$p\in\mathcal B$ at which $Z_1,Z_2,W_3,W_4$ and all ten $D_{kl}$ are nonzero.
\end{proposition}

\begin{proof}
{
We normalize as in \cite[p.~547]{AK}, with the largest components of
$\mathcal Z$ and $\mathcal W$ both having modulus $\epsilon^{-2}$.
Lemma~\ref{lem:germ} realizes this normalization along $\gamma$, with
$\epsilon$ comparable to a positive power of $s$.  We write $\ord$ for order
in $\epsilon$; thus $\ord f>\alpha$ means $f\epsilon^{-\alpha}\to0$.
Since the chart coordinates are invariant under $z\mapsto az$, $w\mapsto a^{-1}w$,
we may compute their orders in this frame.
}

{
The circles give $\ord z_1=\ord z_2=\ord w_3=\ord w_4=-2$.
The uncircled coordinates $z_3,z_4,z_5,w_1,w_2,w_5$ have order $>-2$,
so $\ord z_{15}=\ord z_{25}=-2$.  A $z$-stroke $kl$ means
$\ord(z_{kl}\delta_{kl}^3)=\ord(z_{kl}^{-1/2}w_{kl}^{-3/2})=-2$, or equivalently
$\ord z_{kl}+3\ord w_{kl}=4$.  Thus $\ord w_{15}=\ord w_{25}=2$, and hence
$\ord w_{12}\ge2$.  The stroke $12$ then gives $\ord z_{12}\le-2$.
Combining this with $\ord z_{12}\ge-2$ (Estimate~1, Fact~\ref{fact:rules}),
we obtain $\ord z_{12}=-2$, $\ord w_{12}=2$.
Consequently $r_{kl}^2=z_{kl}w_{kl}$ and $\delta_{kl}$ have order $0$ on
all three $z$-wing pairs; these are the maximal edges with $r\approx1$ in
Estimate~2.  The same argument applies to the $w$-wing.

On a crossed pair, both $z_{kl}$ and $w_{kl}$ are differences of a circled
and an uncircled coordinate.  Both therefore have order $-2$, and
$\ord\delta_{kl}=2$.

We also need bounds on the small coordinate of each outer vertex.
In the $w$-equation $w_1=\sum_{l\ne1}m_l\delta_{1l}^3(w_1-w_l)$,
the terms have orders $2$ for $l=2,5$ and $4$ for $l=3,4$.
It follows that $\ord w_1\ge2$.  Similarly,
$\ord w_2,\ord z_3,\ord z_4\ge2$, while
$w_5=w_1+w_{15}$ and $z_5=z_3+z_{35}$ have order $\ge2$.
These are the estimates $A_w,B_z,m_1w_1+m_2w_2,w_5,B_w,A_z\preccurlyeq\epsilon^2$
of \cite[\S8.6]{AK}.

Finally, $\ord t=-\frac12(\ord z_{12}+\ord w_{34})=2$ and
$\ord\nu=-\ord t-\ord z_{12}=0$.
Lemma~\ref{d6-lem-chart}(ii) now shows that $Z_1,Z_2,W_3,W_4$ and all ten
$D_{kl}$ have order exactly $0$; for the crossed pairs the order is
$2+\frac12(-2-2)=0$.  The remaining coordinates
$W_1,W_2,W_5,Z_3,Z_4,Z_5$ have order $\ge0$.
Thus all chart coordinates are bounded.  Their Puiseux series in $s$
converge, giving $\tilde\gamma(s)\to p\in\mathcal B$.
Coordinates of order exactly $0$ have nonzero limits.
}
\end{proof}
Consequently $D_{12}=\delta_{12}$ and $D_{34}=\delta_{34}$ have finite
nonzero limits. The torque identity can therefore be evaluated at the
boundary point supplied by Proposition~\ref{prop:d6-cover}.

\subsection{A necessary mass condition}

{
The torque identity imposes a condition on the masses before any local
calculation at the boundary points.  Its leading crossed contribution is
$\Sigma=\sum_{k\in\{1,2\},\,l\in\{3,4\}}m_km_lD_{kl}^3Z_kW_l$.
}

{\begin{lemma}\label{d6-lem-torque}
There is a polynomial $Q$ in the chart coordinates, with coefficients polynomial in $m$, such
that identically
\begin{equation}\label{d6-eq-torque}
  \sum_{k=1,2}m_k\bigl(W_kR^z_k-Z_kR^w_k\bigr)+m_5W_5R^z_5=-t\,\Sigma+t^2Q.
\end{equation}
\end{lemma}}
{\begin{proof}
Let $I=\{1,2\}$ and $J=\{3,4\}$. The polynomial in the statement can be chosen as
\begin{align*}
Q={}&t^3\sum_{i\in I,\,j\in J}m_im_jD_{ij}^3W_iZ_j
 +m_5W_5Z_5\\
&+m_5Z_5\sum_{i\in I}m_iD_{i5}^3(W_i-W_5)
 -m_5W_5\sum_{j\in J}m_jD_{j5}^3(Z_5-Z_j).
\end{align*}
Indeed, the terms for the pair $12$ cancel in
$\sum_{i\in I}m_i(W_iR_i^z-Z_iR_i^w)$.
Each crossed pair contributes
$m_im_jD_{ij}^3(-tZ_iW_j+t^5W_iZ_j)$,
and each pair $i5$ contributes
$m_im_5D_{i5}^3(t^2W_iZ_5-Z_iW_5)$.
The terms $-m_im_5D_{i5}^3Z_iW_5$ cancel with the constant part of
$m_5W_5R_5^z$. The remaining terms are exactly $-t\Sigma+t^2Q$.
\end{proof}

}{
To relate this identity to \cite[\S8.6]{AK}, let $\rho^z_k,\rho^w_k$ be
the left-hand sides of \eqref{eq:system}.  The pair torque
$\sum_{k=1,2}m_k(w_k\rho^z_k-z_k\rho^w_k)$ is
$\sum_{k=1,2}\sum_{l=3,4,5}m_km_l\delta_{kl}^3(w_kz_l-z_kw_l)$,
since the terms with $l\in\{1,2\}$ cancel.  It is $2i$ times (31) of that paper.
In the chart, this torque becomes the first sum in \eqref{d6-eq-torque}.
The terms with $l=3,4$ contribute $-t\Sigma+t^5\sum m_km_lD_{kl}^3W_kZ_l$,
and those with $l=5$ contribute
$\sum_{k=1,2}m_km_5D_{k5}^3(t^2W_kZ_5-Z_kW_5)$.
The additional row in the lemma removes the latter contribution at order
zero, because $R^z_5-(m_1D_{15}^3Z_1+m_2D_{25}^3Z_2)$ is divisible by $t^2$.
}

\begin{proposition}[Necessary mass condition]\label{prop:d6-first}
Let $m\in\Rp^5$, let $\gamma$ be a Diagram-6 germ in the numbering above, and let $p$ be its
limit (Proposition~\ref{prop:d6-cover}).  Then:
\begin{romanlist}
  \item $\Sigma(p)=0$, and $\Sigma(p)=-(m_1m_3D_{13}+m_1m_4D_{14}+m_2m_3D_{23}+m_2m_4D_{24})(p)$;
  \item $m_1=m_2$ and $m_3=m_4$; this is relation (32) of \cite[\S8.6]{AK};
  \item $(D_{13},D_{14},D_{23},D_{24})(p)=(2i\eta_{13},\,2\eta_{14},\,2\eta_{23},\,2i\eta_{24})$
  with signs $\eta_{kl}\in\{\pm1\}$ satisfying $\eta_{14}=-\eta_{23}$ and $\eta_{13}=-\eta_{24}$.
\end{romanlist}
In another numbering, wings $\{a,b,c\}$ and $\{d,e,c\}$ force $m_a=m_b$ and $m_d=m_e$.
\end{proposition}

\begin{proof}
All rows vanish along $\tilde\gamma$, so \eqref{d6-eq-torque} gives $t\Sigma=t^2Q$, and $t\ne0$
gives $\Sigma=tQ$.  By Proposition~\ref{prop:d6-cover}, $Q$ is bounded and $t\to0$, so
$\Sigma(p)=0$.  By Lemma~\ref{d6-lem-chart}(iii), $D_{kl}^2Z_kW_l=-1$ at $p$ for the crossed
pairs, so $\Sigma(p)=\sum m_km_lD_{kl}(D_{kl}^2Z_kW_l)=-\sum m_km_lD_{kl}$; this is (i).

By Lemma~\ref{d6-lem-chart}(iv), $Z_1W_3=m_2m_4/s_0^2$, $Z_1W_4=-m_2m_3/s_0^2$,
$Z_2W_3=-m_1m_4/s_0^2$ and $Z_2W_4=m_1m_3/s_0^2$ at $p$, where $s_0>0$,
$s_0^2=(m_1+m_2)(m_3+m_4)$.  Write $m_k=\mu_k^2$ with $\mu_k>0$.  Then $D_{kl}^2=-1/(Z_kW_l)$ gives
\[
  D_{13}=\frac{i\eta_{13}s_0}{\mu_2\mu_4},\quad D_{14}=\frac{\eta_{14}s_0}{\mu_2\mu_3},\quad
  D_{23}=\frac{\eta_{23}s_0}{\mu_1\mu_4},\quad D_{24}=\frac{i\eta_{24}s_0}{\mu_1\mu_3}
\]
with signs $\eta_{kl}\in\{\pm1\}$, and
\[
  \Sigma(p)=-\frac{s_0}{\mu_1\mu_2\mu_3\mu_4}\bigl(P_1+iP_2\bigr),\qquad
  P_1=\eta_{14}\mu_1^3\mu_4^3+\eta_{23}\mu_2^3\mu_3^3,\quad
  P_2=\eta_{13}\mu_1^3\mu_3^3+\eta_{24}\mu_2^3\mu_4^3.
\]
Up to the factor $\pm i$, $P_1+iP_2$ is the left-hand side of (12) of
\cite[\S5.1]{AK} written in the $\mu_k$, and each of its eight sign choices arises from exactly two
of the sixteen sign classes $\eta$.
Since the $\mu_k$ and $s_0$ are real, $\Sigma(p)=0$ forces $P_1=P_2=0$.  If
$\eta_{14}=\eta_{23}$, then $P_1\ne0$, and if $\eta_{13}=\eta_{24}$, then
$P_2\ne0$; this excludes twelve classes.  In the other four,
$P_1=\eta_{14}(x-y)(x^2+xy+y^2)$ with $x=\mu_1\mu_4$, $y=\mu_2\mu_3$ and
$x^2+xy+y^2>0$, so $\mu_1\mu_4=\mu_2\mu_3$; likewise $P_2=0$ gives
$\mu_1\mu_3=\mu_2\mu_4$.  Then
$\mu_3(\mu_1\mu_4-\mu_2\mu_3)-\mu_4(\mu_1\mu_3-\mu_2\mu_4)=\mu_2(\mu_4-\mu_3)(\mu_4+\mu_3)$
vanishes, so $\mu_3=\mu_4$ and then $\mu_1=\mu_2$: this is (ii).  With $m_1=m_2$ and $m_3=m_4$
the displayed values of $D_{kl}$ become those of (iii).
\end{proof}
{
We have thus reduced the boundary limits to the masses
$m_1=m_2$, $m_3=m_4$ and the four sign classes in (iii). This is
relation~(32) of \cite[\S8.6]{AK}. We next solve the remaining boundary
equations on this mass locus.
}

\subsection{The boundary points}

{
We henceforth write $m=(a,a,b,b,c)$ with $a,b,c>0$.
At the boundary, the equations separate into the two wings and the crossed
distance equations.  We first solve the wing equations explicitly.  This
will give a finite list of boundary points at each positive mass.
}

{\begin{lemma}\label{d6-lem-atlas}
The points of $\mathcal B$ are exactly the points with $Z_1=W_3=-\frac12$, $Z_2=W_4=\frac12$,
\begin{gather*}
  (D_{12},D_{15},D_{25})=(d,\lambda d,\omega\lambda d),\\
  W_1=-d(a+2c\lambda),\quad
  W_2=d(a+2c\omega\lambda),\quad W_5=ad\bigl(4\lambda^{-2}-1\bigr),\\
  (D_{34},D_{35},D_{45})=(e,\kappa e,\omega'\kappa e),\\
  Z_3=-e(b+2c\kappa),\quad
  Z_4=e(b+2c\omega'\kappa),\quad Z_5=be\bigl(4\kappa^{-2}-1\bigr),\\
  (D_{13},D_{14},D_{23},D_{24})=(2i\eta_{13},\,2\eta_{14},\,2\eta_{23},\,2i\eta_{24}),\quad
  \eta_{kl}\in\{\pm1\},
\end{gather*}
where
\begin{equation*}
  \omega^3=1,\quad\lambda^2=2(1+\omega),\quad(2a+c\lambda^3)d^3=1,\qquad
  \omega'^3=1,\quad\kappa^2=2(1+\omega'),\quad(2b+c\kappa^3)e^3=1.
\end{equation*}
For each $\omega$ both roots $\lambda$ are nonzero.  The scale denominator $2a+c\lambda^3$
vanishes only for $(\omega,\lambda)=(1,-2)$ and $a=4c$: for $\omega=1$, $\lambda=\pm2$; for
$\omega\ne1$, $\lambda^6=-8$ and the two denominators have product $4a^2+8c^2$.  Hence a wing
has $18$ classes $(\omega,\lambda,d)$ if $a\ne4c$ and $15$ if $a=4c$, and $\mathcal B$ has
$18\cdot18\cdot16=5184$ points if $a\ne4c\ne b$, $4320$ if exactly one of $a,b$ equals $4c$, and
$3600$ if $a=b=4c$.
\end{lemma}}
{\begin{proof}
By Lemma~\ref{d6-lem-chart}(iv),(v), the boundary equations separate
into the two wing systems and an independent square-root choice on each
crossed pair. On the $z$-wing,
$Z_1=-1/2$, $Z_2=1/2$. Write
$d=D_{12}$, $s=D_{15}$, $q=D_{25}$; all three are nonzero.
The equation $R_5^z=0$ gives $q^3=s^3$, whereas $R_1^z=0$ gives
$2ad^3+cs^3=1$. The distance equations give
\[
W_2-W_1=d^{-2},\qquad W_5-W_1=2s^{-2},\qquad
W_2-W_5=2q^{-2}.
\]
Thus, with $\lambda=s/d$ and $\omega=q/s$,
\[
\omega^3=1,\qquad \lambda^2=2(1+\omega),\qquad
(2a+c\lambda^3)d^3=1.
\]
The equation $R_1^w=0$ now reads $W_1+ad+2cs=0$.
Consequently
\begin{align*}
W_1&=-d(a+2c\lambda),\\
W_2&=W_1+d^{-2}=d(a+2c\omega\lambda),\\
W_5&=W_1+2s^{-2}=ad(4\lambda^{-2}-1),
\end{align*}
where the second equality uses $\lambda^3-2\lambda=2\omega\lambda$.
Conversely, these formulas satisfy the three distance equations,
$R_1^w=0$, and
$R_2^w=W_2-ad-2cq=0$. They also satisfy $R_5^z=0$ and
$R_1^z=R_2^z=0$ by the two cubic equations above, as well as the gauge.
This accounts for all nine wing equations.  Exchanging the wings gives
the second system.  Each crossed equation has the two square roots
displayed in the statement, and the four choices are independent.

For $\omega=1$ the two ratios are $\lambda=\pm2$; only $\lambda=-2$
can have zero scale denominator, precisely at $a=4c$.
For $\omega\ne1$, $\lambda^6=-8$, so the product of the two scale
denominators is $4a^2+8c^2>0$. Every nonzero denominator gives three distinct
cube-root scales.  Counting these choices gives the stated number of
boundary points, including those on the exceptional mass faces.
\end{proof}

}{
The classification in Lemma~\ref{d6-lem-atlas} gives the points at which
we must test continuation into $t\ne0$.
}

\begin{definition}\label{d6-def-labels}
{
A \emph{label} specifies a point of $\mathcal B$ by the choices
$(\omega,\lambda,d)$, $(\omega',\kappa,e)$ and $\eta$ of
Lemma~\ref{d6-lem-atlas}.

We classify the ratios $(\omega,\lambda)$ and $(\omega',\kappa)$ on each wing by
\[
 \omega=1\ \Longleftrightarrow\ \lambda=\pm2,
 \qquad
 \omega\ne1\ \Longleftrightarrow\ \lambda^4-2\lambda^2+4=0.
\]
Thus each wing has two real ratios and four nonreal ratios. The same
classification applies to the primed parameters on the second wing.
A label is called \emph{corank-one} when $\eta_{14}=-\eta_{23}$ and
$\eta_{13}=-\eta_{24}$, and \emph{regular} otherwise.
When $a\ne4c\ne b$, there are $12\cdot324=3888$ regular labels and
$4\cdot324=1296$ corank-one labels.
}
\end{definition}

{
The symmetries reduce the number of symbolic calculations required.
At masses $(a,a,b,b,c)$, $F^{-1}(0)$ is preserved by exchanging $1$ and $2$
together with $t\mapsto it$,
$(Z_1,Z_2,W_1,W_2,W_5)\mapsto(-Z_2,-Z_1,-W_2,-W_1,-W_5)$ and
$D_{kl}\mapsto-iD_{kl}$ on crossed pairs.  It is also preserved by
$(t,D_{13},D_{14},D_{23},D_{24})\mapsto-(t,D_{13},D_{14},D_{23},D_{24})$
and by complex conjugation.  The \emph{wing exchange}
($z\leftrightarrow w$, $1\leftrightarrow3$, $2\leftrightarrow4$) maps the chart
at $(a,a,b,b,c)$ to the chart at $(b,b,a,a,c)$.

A \emph{state} records the ratios $(\omega,\lambda)$, $(\omega',\kappa)$ and $\eta$,
but leaves the two cube-root scales unspecified.  The $576$ states form
$29$ symmetry orbits: $21$ regular and $8$ corank-one
(Appendix~\ref{app:verification}).
In the computations, we express the masses through the scales as
$a=\frac12(d^{-3}-c\lambda^3)$ and $b=\frac12(e^{-3}-c\kappa^3)$.
One symbolic calculation for a state therefore covers its nine choices
of $(d,e)$ at every positive $(a,b,c)$.
}

{
It remains to exclude the corank-one labels allowed by the torque
condition.  The next lemma does this by solving twenty equations locally
and showing that a remaining equation has a nonzero leading Taylor
coefficient.  The determinant and coefficient calculations are
computer-assisted and use exact arithmetic.
}

\begin{lemma}[Computer-assisted isolation]\label{d6-lem-corank}
Every corank-one label is an isolated point of $F^{-1}(0)$.
\end{lemma}

\begin{table}[ht]
\centering
\small
\begin{tabular}{p{3.3cm}p{0.9cm}p{0.3cm}p{6.6cm}}
\hline
wing ratios ($z$, $w$) & labels & $k$ & $c_k$ is a nonzero multiple of \\
\hline
real, real & $144$ & $4$ &
$Q_{2,2}=-\frac{2a+2b+c}{60(a+4c)(b+4c)}$, $Q_{2,-2}=\frac{2a+2b+c}{4(a+4c)(b-4c)}$,
$Q_{-2,-2}=\frac{2a+2b+c}{68(a-4c)(b-4c)}$ for $(\lambda,\kappa)=(2,2),(2,-2),(-2,-2)$;
$(-2,2)$ by wing exchange \\
one real, one nonreal & $576$ & $3$ &
$P=A\kappa^3-4A\kappa+B\kappa^3-B\kappa^2-4B\kappa+B$ for $\lambda=\pm2$, $\kappa$ nonreal;
product over the nonreal $\kappa$: $9(8A^2+16AB+9B^2)^2$ \\
nonreal, nonreal: orbit of $(\lambda_0,\lambda_0)$ & $288$ & $2$ &
$1-2\kappa^3$; product over the nonreal $\kappa$: $1089$ \\
nonreal, nonreal: orbits of $(\lambda_0,-\lambda_0)$ & $288$ & $2$ &
$2B/(2B+\kappa^3)$; product over the nonreal $\kappa$: $B^4/(B^2+2)^2$ \\
\hline
\end{tabular}
\caption{Order $k$ and leading coefficient of $R^w_2$ along $C_p$ at the corank-one labels
(Lemma~\ref{d6-lem-corank}), at one representative state of each orbit.  Counts are for
$a\ne4c\ne b$; $A=a/c$, $B=b/c$; $\lambda_0=\frac12(\sqrt6+i\sqrt2)$ is a nonreal ratio for
$\omega=\frac12(-1+i\sqrt3)$.  In the first row $R^w_4$ has the same order.}\label{d6-tab-corank}
\end{table}

{
\begin{proof}
We first work at a representative state of each of the eight corank-one orbits.
At each such representative label $p$, exact symbolic computation gives:
\begin{romanlist}
  \item the $20$ rows other than $R^w_2$ and $R^w_4$, together with $t$, have
  invertible derivative, so their common zeros form an analytic graph $C_p$ over $t$;
  \item on this graph, $R^w_2=c_kt^k+O(t^{k+1})$, where $c_k\ne0$ and the
  order and coefficient are listed in Table~\ref{d6-tab-corank}.
\end{romanlist}
The coefficient field is $\mathbb Q(i)(c,d,e)$ when both wing ratios are real;
otherwise it is extended by a root $\kappa$ of $\kappa^4-2\kappa^2+4$.
The computations are given in Appendix~\ref{app:verification}.

For (i), eliminating the diagonal derivative block of the distance rows
reduces the determinant to a Schur complement.
When both wing ratios are real, its determinant, after multiplying the
rows by $d^3e^3$, is
$82944\,d^{61}e^{59}(cd^3-\frac{\lambda}{16})(ce^3-\frac{\kappa}{16})$.
The first factor is nonzero because
$cd^3-\frac{\lambda}{16}=-\frac{\lambda a}{8(2a+c\lambda^3)}\ne0$;
the second is treated by exchanging the wings.

In the other cases, after removing the unit row for $t$, the complement
is block diagonal with one $5\times5$ block for each wing
(Lemma~\ref{d6-lem-chart}(v)).  The computation checks the vanishing of an
off-diagonal block and evaluates the two diagonal blocks.
For one real and one nonreal wing ratio, the determinants are nonzero
multiples of $8cd^3-\frac{\lambda}{2}=-\lambda A/(2A+\lambda^3)$ and of
a polynomial in $ce^3$ and $\kappa$.  The product of the latter polynomial
over the nonreal $\kappa$ is
$(\frac9{16})^4\cdot64B^4(B^4-8B^2+64)/(B^2+2)^4$.
Its nonvanishing follows from $B^4-8B^2+64=(B^2-4)^2+48$.
This product, and those in Table~\ref{d6-tab-corank}, are evaluated using
resultants with $\kappa^4-2\kappa^2+4$.

For (ii), the Taylor coefficients of $C_p$ are found by solving the
$20$ rows in (i) successively in powers of $t$.  Substitution into all
$22$ rows checks the resulting expansions and gives the first nonzero
coefficient listed in Table~\ref{d6-tab-corank}.

Near a representative label $p$, (i) gives $F^{-1}(0)\subset C_p$.
By (ii), $R^w_2$ vanishes on $C_p$ only at $t=0$, so $p$ is isolated.
The chart symmetries described above are invertible and map the full zero set
to the corresponding full zero set, with the wing masses exchanged when required.
Complex conjugation has the same property. Each symmetry therefore carries an
isolated point to an isolated point. Since the representative states cover all
corank-one labels, this proves the assertion at every label.
\end{proof}

The lemma excludes all boundary limits allowed by
Proposition~\ref{prop:d6-first}(iii).
}

\subsection{Exclusion}

We now combine the coverage, mass restriction and isolation results.
Every Diagram~6 germ would converge to an isolated boundary point, which
forces its chart parameter to vanish identically.

\begin{theorem}\label{thm:d6}
Let $m\in\Rp^5$.  No curve in $A_m$ has a Diagram-6 germ, for any numbering of the bodies and
either colouring.  In particular the case $r=6$ of Theorem~\ref{thm:exclusion}(a) holds, at the
complex level.
\end{theorem}

\begin{proof}
{
The symmetries at the beginning of the section allow us to use the
numbering above.  Suppose that $\gamma(s)$ is a Diagram-6 germ on a curve
in $A_m$.  Proposition~\ref{prop:d6-cover} gives a chart germ
$\tilde\gamma(s)$ with $t(s)\to0$, $t(s)\ne0$ for $s\ne0$, bounded
coordinates, and a limit $p\in\mathcal B$.

Proposition~\ref{prop:d6-first} forces $m=(a,a,b,b,c)$ and
$\eta_{14}=-\eta_{23}$, $\eta_{13}=-\eta_{24}$ at $p$.
Lemma~\ref{d6-lem-atlas} therefore identifies $p$ as a label of corank one.
The computer-assisted Lemma~\ref{d6-lem-corank} shows that this point is
isolated in $F^{-1}(0)$.

On the other hand, the coordinates of $\tilde\gamma$ are bounded Puiseux
series in $s$.  After a substitution $s=\sigma^N$, they are convergent
power series in $\sigma$.  Thus $\sigma\mapsto\tilde\gamma(\sigma^N)$
is a holomorphic map from a disc into $F^{-1}(0)$, taking the value $p$
at $\sigma=0$.  Since $p$ is isolated, the map is constant near
$\sigma=0$.  This contradicts $t\not\equiv0$.
}
\end{proof}

\section{Diagram 7}\label{sec:d7}

We exclude Diagram~7 on a curve through a real configuration by showing that
an unbounded four-product forces the inertia to tend to zero.  The key step is
a decomposition of the inertia into terms whose orders are positive.
Diagram~7 is the butterfly without circles in \cite[\S7.4]{AK}; we use its
two triangular clusters to obtain this decomposition.

\begin{theorem}\label{thm:d7}
Let $m\in\Rp^5$.  No Diagram~7 germ with an unbounded four-product lies on a curve through a
real point: if an irreducible curve $\Gamma\subset A_m$ carries such a germ, then $\Gamma$
contains no real point.
\end{theorem}

This proves part~(b) of Theorem~\ref{thm:exclusion} for $r=7$.
For the orders $a,b$ introduced in Proposition~\ref{d7-cover}, the proof has
two cases:
\[
\begin{aligned}
a+b\ge0&\quad\Longrightarrow\quad p_{ij}\longrightarrow0
&&\text{for every }i<j,\\
a+b<0&\quad\Longrightarrow\quad I\longrightarrow0
&&\Longrightarrow\quad I|_\Gamma=0.
\end{aligned}
\]
The first implication follows from the distance orders
(Lemma~\ref{d7-bounded}); the second follows from the inertia identity
(Lemma~\ref{d7-vanish}) and its constancy on $\Gamma$.
At a real point, $I=\sum_km_k|z_k|^2>0$, which gives the required contradiction.

\subsection{The diagram and the nested chart}
In the numbering of \cite[\S8.7]{AK}, the \emph{upper wing} $\{1,2,5\}$ is an isolated triangle
of $z$-edges without $z$-circles, the \emph{lower wing} $\{3,4,5\}$ is an isolated triangle of
$w$-edges without $w$-circles, and the four \emph{crossed} pairs $kj$ ($k\in\{1,2\}$,
$j\in\{3,4\}$) are not edges.
Since Diagram~7 has no circles, Rule~1e shows that neither wing has a maximal stroke (see the
proof of Proposition~\ref{d7-cover}(i)).  By Definition~\ref{def:germ} and the convention after
Theorem~\ref{thm:exclusion}, we may use this numbering and colouring.

We first normalize the orders as in Lemma~\ref{lem:germ}.  After renormalization,
$Z$ and $W$ have a common order.  Choosing a power $t$ of the Puiseux parameter makes this
order $-2$.  Thus $t$ plays the role of $\epsilon$ in \cite[\S3]{AK}, and $\ord$ denotes
order in $t$.  The renormalizations $(z,w)\mapsto(\lambda z,\lambda^{-1}w)$ preserve the
diagram, the distance constraints, $I$, $U$ and every $r_{kl}^2=z_{kl}w_{kl}$.

To describe the two clusters at their different scales, we use the nested chart
with exponents $P,Q,R$ and coordinates $A_1,A_2,C,u_3,u_4$,
$B_3,B_4,D,v_1,v_2$ and $D_{kl}$ ($k<l$):
\begin{gather*}
  z_1=A_1,\quad z_2=A_2,\qquad w_3=t^{-Q}B_3,\quad w_4=t^{-Q}B_4,\\
  z_j=C+t^Pm_5u_j\ \ (j=3,4),\qquad z_5=C-t^P(m_3u_3+m_4u_4),\\
  w_k=t^{-Q}D+t^Rm_5v_k\ \ (k=1,2),\qquad w_5=t^{-Q}D-t^R(m_1v_1+m_2v_2),
\end{gather*}
and $\delta_{kl}=t^{e_{kl}}D_{kl}$, where $e_{kl}=-R/2$ on the upper wing, $Q/2$ on the crossed
pairs and $(Q-P)/2$ on the lower wing ($D$ is the scaled cluster centre).
Here $C$ and $t^{-Q}D$ are the centres of mass of the $z$-cluster $\{3,4,5\}$ and
the $w$-cluster $\{1,2,5\}$.  The exponents $P$ and $R$ describe the internal separations
of these clusters, while $-Q$ describes the common $w$-scale.  The choices of $e_{kl}$
make the two sides of all ten distance constraints $\delta_{kl}^2z_{kl}w_{kl}=1$ have the same order.
For the pairs $12$ and $34$, the distance constraints become
\[
  m_5D_{12}^2(A_2-A_1)(v_2-v_1)=1,\qquad m_5D_{34}^2(u_4-u_3)(B_4-B_3)=1.
\]
The \emph{open cone} is $P>Q>0$, $R>0$.

We retain the rotational freedom in this chart and use the equations of
\eqref{eq:system} preceding the gauge $z_{12}=w_{12}$.
The distance constraints give
$(A_2-A_1)(B_4-B_3)\ne0$.
The transformations $(z,w)\mapsto(\lambda z,\lambda^{-1}w)$ preserve these
equations, $I$ and $U$, so the inertia calculation can be carried out directly
in the chart.

\subsection{Orders of the germs and the chart}
We now determine the orders of all separations from the diagram rules.
This isolates the order range relevant to an unbounded four-product and shows that the
rescaled coordinates have finite limits.

\begin{proposition}[Representation in the nested chart]\label{d7-cover}
Let $\gamma$ be a Diagram~7 germ.  There are $a,b$ with $-2<a<1$ and $-2<b<1$ such that the
following hold.
\begin{romanlist}
  \item On the upper wing $\ord z_{kl}=a$ and $\ord w_{kl}=(4-a)/3$; on the lower wing
  $\ord w_{kl}=b$ and $\ord z_{kl}=(4-b)/3$; on the crossed pairs $\ord z_{kl}=a$ and
  $\ord w_{kl}=b$.  Moreover $\ord z_k\ge a$ and $\ord w_k\ge b$ for every $k$.
  \item After the rotation $(z,w)\mapsto(t^{-a}z,t^{a}w)$, $\gamma$ is a germ of the nested
  chart with exponents
  \[
    P=\frac{4-b}{3}-a,\qquad Q=-(a+b),\qquad R=\frac{4+2a}{3};
  \]
  all chart coordinates have nonnegative order, so they converge, and every $D_{kl}$ has order
  zero.  Moreover $R=\frac23(a+2)>0$ and $P-Q=\frac23(b+2)>0$, so $(P,Q,R)$ lies in the open
  cone if and only if $a+b<0$.
\end{romanlist}
\end{proposition}

\begin{proof}
(i) We use Fact~\ref{fact:rules}.  By Rule~2e \cite[p.~552]{AK} the three $z$-separations of a
wing have one order, and so do its three $w$-separations; let $a$ be the $z$-order of the upper
wing and $b$ the $w$-order of the lower wing.  As $\delta_{kl}^2z_{kl}w_{kl}=1$,
$\ord Z_{kl}=-(\ord z_{kl}+3\ord w_{kl})/2$ and $\ord W_{kl}=-(3\ord z_{kl}+\ord w_{kl})/2$;
since the strokes have order $-2$, this gives the other two wing orders.  The upper wing is an
isolated component of the $z$-diagram without $z$-circles, so it has no maximal $z$-stroke
(Rule~1e, p.~546), and along a $z$-stroke Estimate~2 (p.~548) gives $-2\le\ord z_{kl}\le1$, with
$-2$ exactly for a maximal $z$-edge and $1$ exactly for a $zw$-edge.  Hence $-2<a<1$, and dually
$-2<b<1$.  As $z_{13}=z_{15}+z_{53}$
with $\ord z_{53}=(4-b)/3>1>a$, and $w_{13}=w_{15}+w_{53}$ with $\ord w_{15}=(4-a)/3>1>b$, the
crossed orders are $a$ and $b$; likewise for $14$, $23$, $24$.  Finally
$\sum_km_kz_k=\sum_km_kw_k=0$ on $A_m$ (proof of Lemma~\ref{lem:distances}), so each $z_k$ is a
combination of $z$-differences, all of order at least $a$; likewise $\ord w_k\ge b$.

(ii) After the rotation the $z$-differences have order $0$ on the upper wing and the crossed
pairs and $P$ on the lower wing, the $w$-differences have order $R$ on the upper wing and $-Q$
elsewhere, and $\ord z_k\ge0$, $\ord w_k\ge-Q$.  Inverting the chart map, $A_k=z_k$,
$B_j=t^Qw_j$, $C$ and $t^{-Q}D$ are the centres of the two clusters, $m_5t^Pu_j=z_j-C$ and
$m_5t^Rv_k=w_k-t^{-Q}D$ are combinations of differences inside them, and
$D_{kl}=t^{-e_{kl}}\delta_{kl}$, where $\ord\delta_{kl}=-(\ord z_{kl}+\ord w_{kl})/2=e_{kl}$.  So
every coordinate has nonnegative order and $\ord D_{kl}=0$.
\end{proof}

Thus $a+b<0$ is equivalent to $P>Q>0$ and $R>0$.
These inequalities make each term of the inertia decomposition decay.
For $a+b\ge0$, we will instead apply the distance orders directly to the
four-products in Lemma~\ref{d7-bounded}.

\subsection{The inertia identity}
We express $I$ through the internal separations of the two clusters and their
crossed forces.  The chart then converts each summand into a positive power
of the parameter times a bounded function.
Write
\[
  F^z_k=z_k-\sum_{l\ne k}m_l\delta_{kl}^3(z_k-z_l),\qquad
  F^w_k=w_k-\sum_{l\ne k}m_l\delta_{kl}^3(w_k-w_l)
\]
for the coordinate rows of \eqref{eq:system} and $T_{kl}=\delta_{kl}^2z_{kl}w_{kl}-1$ for the
distance constraints.  Then $I-U=\sum_km_kw_kF^z_k+\sum_{k<l}m_km_l\delta_{kl}T_{kl}$, so $U=I$ on $A_m$
\cite[p.~541]{AK}.

\begin{lemma}[Inertia identity]\label{d7-identity}
Let
\[
  S_5=\sum_{j=3,4}\sum_{l=1,2}m_jm_l\delta_{jl}^3
  \bigl((z_j-z_5)(w_l-w_5)-(w_j-w_5)(z_l-z_5)\bigr).
\]
Then, as polynomials in the masses and in $(z,w,\delta)$,
\begin{multline*}
  I=\sum_{k=1,2}m_kz_k(w_k-w_5)+\sum_{j=3,4}m_jw_j(z_j-z_5)+S_5\\
  -\sum_{j=3,4}m_j\bigl((z_j-z_5)F^w_j-(w_j-w_5)F^z_j\bigr)+w_5\sum_{k=1}^5m_kF^z_k.
\end{multline*}
In particular $I=\sum_{k=1,2}m_kz_k(w_k-w_5)+\sum_{j=3,4}m_jw_j(z_j-z_5)+S_5$ on $A_m$.
\end{lemma}

\begin{proof}
Put $\hat z_k=z_k-z_5$ and $\hat w_k=w_k-w_5$.  By antisymmetry of the force terms,
$\sum_km_kF^z_k=\sum_km_kz_k$.  Since $z_j\hat w_j-w_j\hat z_j=z_5w_j-w_5z_j$,
\begin{align*}
  I&=\sum_km_kz_k\hat w_k+w_5\sum_km_kz_k\\
   &=\sum_{k=1,2}m_kz_k\hat w_k+\sum_{j=3,4}m_j\bigl(w_j\hat z_j+z_5w_j-w_5z_j\bigr)
    +w_5\sum_km_kF^z_k.
\end{align*}
For $j\in\{3,4\}$, using $z_j-z_l=\hat z_j-\hat z_l$ and $w_j-w_l=\hat w_j-\hat w_l$,
\[
  \hat z_jF^w_j-\hat w_jF^z_j=(w_5z_j-z_5w_j)
  +\sum_{l\ne j}m_l\delta_{jl}^3(\hat z_j\hat w_l-\hat w_j\hat z_l).
\]
In $\sum_{j=3,4}m_j(\hat z_jF^w_j-\hat w_jF^z_j)$ the terms with $l=5$ vanish, those with
$l\in\{3,4\}$ cancel in pairs, and those with $l\in\{1,2\}$ add up to $S_5$.  Hence
$\sum_{j=3,4}m_j(z_5w_j-w_5z_j)=S_5-\sum_{j=3,4}m_j(\hat z_jF^w_j-\hat w_jF^z_j)$, and
substituting this into the display gives the identity.
\end{proof}

The term $S_5$ is the torque about body~$5$ of the crossed forces on the lower wing; compare
\cite[(31)]{AK}.

\begin{lemma}\label{d7-vanish}
Let $\gamma$ be a Diagram~7 germ on $A_m$ with $a+b<0$ (Proposition~\ref{d7-cover}).  Then
$I(\gamma(t))\to0$ as $t\to0$.
More precisely, in the nested chart of Proposition~\ref{d7-cover}(ii) the three terms
$\sum_{k=1,2}m_kz_k(w_k-w_5)$, $\sum_{j=3,4}m_jw_j(z_j-z_5)$ and $S_5$ are sums of powers $t^e$
times polynomials in the chart coordinates, with exponents
\[
  e=R;\qquad e=P-Q;\qquad e\in\bigl\{\tfrac Q2,\ P+\tfrac Q2,\ \tfrac{3Q}2+R,\ P+\tfrac{3Q}2+R\bigr\},
\]
respectively.
\end{lemma}

\begin{proof}
On $A_m$ the last two sums in Lemma~\ref{d7-identity} vanish, and the three remaining terms are
invariant under the rotations, so we may evaluate them on the rotated germ in the chart.  There
\begin{gather*}
  w_k-w_5=t^R\bigl((m_k+m_5)v_k+m_{k'}v_{k'}\bigr),\qquad
  z_j-z_5=t^P\bigl((m_j+m_5)u_j+m_{j'}u_{j'}\bigr),\\
  w_j-w_5=t^{-Q}(B_j-D)+t^R(m_1v_1+m_2v_2),\\
  z_l-z_5=A_l-C+t^P(m_3u_3+m_4u_4),
\end{gather*}
where $\{k,k'\}=\{1,2\}$, $\{j,j'\}=\{3,4\}$ and $l\in\{1,2\}$, and
$\delta_{jl}^3=t^{3Q/2}D_{jl}^3$.  Substituting, together with $z_k=A_k$ and $w_j=t^{-Q}B_j$,
gives the stated exponents.  Since $a+b<0$, Proposition~\ref{d7-cover}(ii) puts $(P,Q,R)$ in the
open cone, so every exponent is positive, and all chart coordinates are bounded.  Hence the three
terms, and with them $I$, tend to zero.
\end{proof}

\subsection{Germs outside the open cone}
We complete the order analysis by treating $a+b\ge0$.
Every squared distance then has nonnegative order, and each four-product
contains at least one squared distance of positive order.

\begin{lemma}\label{d7-bounded}
Let $\gamma$ be a Diagram~7 germ with $a+b\ge0$ (Proposition~\ref{d7-cover}).  Then every
four-product tends to zero along $\gamma$.
\end{lemma}

\begin{proof}
By Proposition~\ref{d7-cover}(i), $r_{kl}^2=z_{kl}w_{kl}$ has order $(4+2a)/3>0$ on the upper
wing, $(4+2b)/3>0$ on the lower wing and $a+b\ge0$ on the crossed pairs.  The four-product
$p_{ij}$ contains the three squared sides of the triangle $\{k,l,m\}$ complementary to
$\{i,j\}$, and this triangle contains a wing pair: if $5\in\{k,l,m\}$, both pairs through~$5$
are wing pairs; otherwise $\{k,l,m\}\subset\{1,2,3,4\}$ contains $\{1,2\}$ or $\{3,4\}$ (the
crossed pairs form the bipartite graph $K_{2,2}$).  Hence
$\ord p_{ij}\ge\min\{(4+2a)/3,(4+2b)/3\}>0$.
\end{proof}

\subsection{Conclusion of the Diagram 7 argument}
We combine the two order estimates with the constancy of the inertia to
exclude an unbounded four-product on a curve through a real point.

\begin{proof}[Proof of Theorem~\ref{thm:d7}]
Let $\gamma$ be a Diagram~7 germ with an unbounded four-product on an irreducible curve
$\Gamma\subset A_m$, numbered as above, and let $a,b$ be as in Proposition~\ref{d7-cover}.  Some
four-product has negative order along $\gamma$ (Definition~\ref{def:germ}), so $a+b<0$ by
Lemma~\ref{d7-bounded}, and $I(\gamma(t))\to0$ by Lemma~\ref{d7-vanish}.  By
Lemma~\ref{lem:U-constant}, $I=U$ is constant on $\Gamma$, with value
$\lim_{t\to0}I(\gamma(t))=0$.  At a real point $w_k=\bar z_k$, so $I=\sum_km_k|z_k|^2$,
which is positive because $m_k>0$ and the distance constraint for the pair $12$ forces $z_1\ne z_2$.  Hence
$\Gamma$ contains no real point.
\end{proof}

\section{Diagram 8}\label{sec:d8}

We exclude Diagram~8 by comparing the force equations with the exact distance
constraints.  The limiting equations determine a mass relation; the first
nonzero terms reduce the admissible masses to a one-parameter family, where
two quadratic balances are incompatible.

Diagram~8 is the butterfly of \cite[\S7.4]{AK} with exactly one circled wing
(Fact~\ref{fact:fig11}).  We use the numbering of \cite[\S8.8]{AK}: the upper wing $\{1,2,5\}$
is an isolated triangle of $z$-edges without $z$-circles, the lower wing $\{3,4,5\}$ is a
triangle of $w$-edges with $w$-circles at $3$ and $4$ (vertex $5$ is never circled
\cite[\S7.4]{AK}, and Rules~1a, 1d place the circles of a circled wing at its other two
vertices), and the crossed pairs $13$, $14$, $23$, $24$ are not edges.
The other circling becomes this one after exchanging $z$ and $w$ and renumbering
$1\leftrightarrow3$, $2\leftrightarrow4$; both operations are symmetries of \eqref{eq:system}
(the second permutes the masses), so it suffices to treat this numbering for every
$m\in\Rp^5$.
The mass condition of the diagram is relation (40) of \cite[\S8.8]{AK},
\begin{equation}\label{eq:d8-40}
  m_5^2(m_1+m_2)^4=m_1^3m_2^3,
\end{equation}
as stated in \cite[p.~578]{AK}.  For the other circling it reads $m_5^2(m_3+m_4)^4=m_3^3m_4^3$, which is their (39).
The proof begins with a chart in which the lower wing has a finite limit.
Every Diagram~8 germ converges to the boundary of this chart
(Proposition~\ref{prop:d8-cover}), and its limiting equations imply
\eqref{eq:d8-40} (Proposition~\ref{prop:d8-first}).  We then use the
{exact distance equations on the upper wing and the coefficients of order nine
to restrict the possible germs to one curve of}
masses.  A second-order expansion excludes the remaining germs, including complex
ones (Theorem~\ref{thm:d8}).

We shall also use two symmetries to normalize the masses and coordinates.  Scaling the
masses by $\lambda>0$ and $(z,w,\delta)$ by
$(\lambda^{1/3},\lambda^{1/3},\lambda^{-1/3})$ maps $A_m$ onto $A_{\lambda m}$.
Every coordinate and every component of $\mathcal Z$, $\mathcal W$ is multiplied by a
nonzero constant, so the orders and diagrams are preserved (Lemma~\ref{lem:germ}).
The equations of \eqref{eq:system}, apart from $z_{12}=w_{12}$, are also invariant under
$z\mapsto cz$, $w\mapsto c^{-1}w$.  This changes the $z$-orders by $\ord c$ and the
$w$-orders by $-\ord c$, preserving the diagram \cite[p.~547]{AK}.

\subsection{A chart in distance ratios}
We normalize one separation in each wing and introduce distance ratios on
the upper wing.  In the resulting coordinates the force and distance equations
extend polynomially to the limiting configurations.
Let $\gamma$ be a Diagram~8 germ.  We rotate it by $z_{12}^{-1}$, so that $z_{12}=1$ along
$\gamma$.  The chart has two parameters $t,h$ and nineteen coordinates, the
positions and shapes $A_1,A_2,C,B_3,B_4,D,u,v,R_{12},R_{25},P,Q$ and the inverse distances
$D_{13},D_{14},D_{23},D_{24},D_{34},D_{35},D_{45}$; we write $x$ for all of them.  It is given by
\begin{gather*}
  z_1=A_1,\quad z_2=A_2,\quad z_3=C+m_5t^6u,\quad z_4=C+m_5t^6v,\\
  z_5=C-t^6(m_3u+m_4v),\quad w_3=t^{-6}B_3,\quad w_4=t^{-6}B_4,\\
  w_1=t^{-6}D+m_5t^4h^{-2}P,\quad w_2=t^{-6}D+m_5t^4h^{-2}Q,\\
  w_5=t^{-6}D-t^4h^{-2}(m_1P+m_2Q),\quad \delta_{15}=t^{-2}h,\\
  \delta_{12}=t^{-2}hR_{12},\quad \delta_{25}=t^{-2}hR_{25},\\
  \delta_{bc}=t^3D_{bc}\ \ (b\in\{1,2\},\ c\in\{3,4\}),\qquad
  \delta_{kl}=D_{kl}\ \ (k,l\in\{3,4,5\}).
\end{gather*}
{On the gauge slice defined below, we take a local sixth root $t$ of
$(w_4-w_3)^{-1}$.  Once this root is fixed and $th\ne0$, the inverse chart map is uniquely
determined: $C$ and $t^{-6}D$ are the centres of mass of $z_3,z_4,z_5$ and of}
$w_1,w_2,w_5$, $h=t^2\delta_{15}$, $R_{12}=\delta_{12}/\delta_{15}$,
$R_{25}=\delta_{25}/\delta_{15}$, $u=(z_3-C)/(m_5t^6)$, $P=h^2(w_1-t^{-6}D)/(m_5t^4)$, and
{the remaining coordinates are recovered in the same way.  Along a germ,
a finite ramification allows all these roots to be chosen consistently.}
The rows of the chart are the gauges $E_{20}:=A_2-A_1-1$ and $E_{21}:=B_4-B_3-1$ (that is,
$z_{12}=1$ and $w_4-w_3=t^{-6}$) and $E_0,\dots,E_{19}$, obtained by substituting the chart into
the twenty equations of \eqref{eq:system} other than $z_{12}=w_{12}$ and multiplying each by a
monomial in $t$ and $h^{1/2}$: $E_{2k-2}$ and $E_{2k-1}$ come from the $z$- and the $w$-equation
of body $k$, and $E_{10},\dots,E_{19}$ from the distance constraints of the pairs
$12,13,14,15,23,24,25,34,35,45$.
In the corner coordinates
\[
  a:=th^{-1/2},\qquad \eta:=t^2h^{1/2}\qquad(\text{so } t^3=a\eta,\ h^{3/2}=\eta\,a^{-2})
\]
we use the monomials
\[
  \mu:=a^6=t^6h^{-3},\qquad \nu:=a^2\eta^2=t^6,\qquad T:=a\eta=t^3,\qquad \eta^2=t^4h,
\]
and we write $X:=A_1-C$, $Y:=A_2-C$, $M_{12}:=m_1+m_2$, $M:=m_3+m_4+m_5$, and
$C_5:=C-\nu(m_3u+m_4v)$, which is $z_5$.

{The chart parameters may be complex.  In asymptotic estimates we therefore
interpret sums of scales as sums of absolute values: $O(\mu+\rho^2)$ means
$O(|\mu|+|\rho|^2)$, and $o(\mu+\rho^2)$ means $o(|\mu|+|\rho|^2)$.
Products inside $O$ and $o$ are interpreted in the same way.  Algebraic identities and
equations between leading coefficients retain their usual complex meaning.}

{\begin{lemma}\label{lem:d8-rows}
The chart equations have the following polynomial identities, valid for all masses.
\begin{romanlist}
\item Every $E_j$ is a polynomial in $a$, $\eta$ and $x$, and its restriction
$f_j:=E_j|_{a=\eta=0}$ to the boundary is
\begin{align*}
  f_0&=-\bigl(m_2(X-Y)R_{12}^3+m_5X\bigr), & f_8&=m_1X+m_2YR_{25}^3,\\
  f_2&=-\bigl(m_1(Y-X)R_{12}^3+m_5YR_{25}^3\bigr), & f_{20}&=A_2-A_1-1,\\
  f_{10}&=m_5(X-Y)(P-Q)R_{12}^2-1, & f_{13}&=X\bigl((m_1+m_5)P+m_2Q\bigr)-1,\\
  f_{16}&=YR_{25}^2\bigl(m_1P+(m_2+m_5)Q\bigr)-1, & f_4&=f_6=C,\quad f_1=f_3=D,\\
  f_{11}&=X(D-B_3)D_{13}^2-1, & f_{12}&=X(D-B_4)D_{14}^2-1,\\
  f_{14}&=Y(D-B_3)D_{23}^2-1, & f_{15}&=Y(D-B_4)D_{24}^2-1,\\
  f_5&=B_3+m_4(B_4-B_3)D_{34}^3-m_5(B_3-D)D_{35}^3, & f_{21}&=B_4-B_3-1,\\
  f_7&=B_4-m_3(B_4-B_3)D_{34}^3-m_5(B_4-D)D_{45}^3,\\
  f_9&=D+m_3(B_3-D)D_{35}^3+m_4(B_4-D)D_{45}^3,\\
  f_{17}&=m_5(B_3-B_4)(u-v)D_{34}^2-1,\\
  f_{18}&=(B_3-D)\bigl((m_3+m_5)u+m_4v\bigr)D_{35}^2-1,\\
  f_{19}&=(B_4-D)\bigl(m_3u+(m_4+m_5)v\bigr)D_{45}^2-1.
\end{align*}
\item $m_1E_0+m_2E_2+m_5E_8+a^6(m_3E_4+m_4E_6)=a^6(m_1A_1+m_2A_2+MC)$.
\item $E_1=D+\eta^2\beta_1+T^3\tau_1+\sum_{i\ge6,\,j\ge2}a^i\eta^j(\cdots)$, and $E_3$ is the
same with $\beta_3,\tau_3$, where
\begin{align*}
  \beta_1&=-m_5\bigl(m_2(P-Q)R_{12}^3+(m_1+m_5)P+m_2Q\bigr),\\
  \beta_3&=-m_5\bigl(m_1(Q-P)R_{12}^3+(m_1P+(m_2+m_5)Q)R_{25}^3\bigr),\\
  \tau_1&=m_3(B_3-D)D_{13}^3+m_4(B_4-D)D_{14}^3,\\
  \tau_3&=m_3(B_3-D)D_{23}^3+m_4(B_4-D)D_{24}^3;
\end{align*}
the monomials of $E_1-E_3$ are $\eta^2$, $T^3$, $a^6\eta^2$ and $a^9\eta^5$.
\item For $k\in\{0,2,8,10,13,16,20\}$, $E_k=f_k(\mathcal I)+\sum_{i\ge6}a^i\eta^j(\cdots)$, where
$f_k$ is read as a function of $(X,Y,R_{12},R_{25},P,Q)$ and
$\mathcal I:=(A_1-C_5,\,A_2-C_5,\,R_{12},\,R_{25},\,P,\,Q)$.
\item $E_4=C+\nu L_3+T^3F_3+T^5G_3$ exactly, where
$L_3=-m_5\bigl(m_4(u-v)D_{34}^3+((m_3+m_5)u+m_4v)D_{35}^3-u\bigr)$,
$F_3=-\bigl(m_1(C-A_1)D_{13}^3+m_2(C-A_2)D_{23}^3\bigr)$ and
$G_3=-m_5u\bigl(m_1D_{13}^3+m_2D_{23}^3\bigr)$.
\end{romanlist}
\end{lemma}}
{\begin{proof}
We write out the chart equations to identify both their limiting terms and the
monomials that enter the later order comparisons.
Put $I=\{1,2\}$, $J=\{3,4\}$, $P_1=P$, $P_2=Q$, $u_3=u$, $u_4=v$,
and set
\[
q=R_{12},\quad s=R_{25},\quad V=m_3u+m_4v,\quad S=m_1P+m_2Q,
\quad \gamma=\mu\eta^2=\nu a^4.
\]
The unscaled positions are
\[
z_i=A_i,\quad z_j=C+\nu m_5u_j,\quad z_5=C_5=C-\nu V,
\]
whereas their scaled $w$-coordinates are
\[
\nu w_i=D+\gamma m_5P_i,\qquad
\nu w_j=B_j,\qquad \nu w_5=D-\gamma S.
\]
The inverse-distance cubes are $\mu^{-1}q^3,\mu^{-1},\mu^{-1}s^3$
on $12,15,25$, respectively, $T^3D_{ij}^3$ on crossed pairs, and
$D_{kl}^3$ on the lower wing.
The chart equations are obtained by multiplying the $z$-equations of $1,2,5$
by $\mu$, leaving those of $3,4$ unscaled, and multiplying all $w$-equations by $\nu$.
Writing $X'=A_1-C_5$ and $Y'=A_2-C_5$, the resulting upper $z$-rows are
\begin{align*}
E_0={}&\mu A_1-m_2(A_1-A_2)q^3-m_5X'
 -\mu T^3\sum_{j\in J}m_j(A_1-C-\nu m_5u_j)D_{1j}^3,\\
E_2={}&\mu A_2-m_1(A_2-A_1)q^3-m_5Y's^3
 -\mu T^3\sum_{j\in J}m_j(A_2-C-\nu m_5u_j)D_{2j}^3,\\
E_8={}&\mu C_5+m_1X'+m_2Y's^3
 +\mu\nu\sum_{j\in J}m_j(V+m_5u_j)D_{j5}^3.
\end{align*}
If $j'$ denotes the other element of $J$, the lower $z$-rows are
\begin{align*}
E_{2j-2}={}&C+\nu m_5u_j
 -\nu m_{j'}m_5(u_j-u_{j'})D_{jj'}^3
 -\nu m_5(V+m_5u_j)D_{j5}^3\\
&-T^3\sum_{i\in I}m_i(C-A_i+\nu m_5u_j)D_{ij}^3.
\end{align*}
This formula at $j=3$ proves (v).

The upper $w$-rows are
\begin{align*}
E_1&=D+\eta^2\beta_1+T^3\tau_1
 +\gamma m_5P-\gamma T^3m_5P\sum_{j\in J}m_jD_{1j}^3,\\
E_3&=D+\eta^2\beta_3+T^3\tau_3
 +\gamma m_5Q-\gamma T^3m_5Q\sum_{j\in J}m_jD_{2j}^3.
\end{align*}
Thus (iii) follows, including its complete list of monomials.
The lower and central $w$-rows are
\begin{align*}
E_{2j-1}={}&B_j-m_{j'}(B_j-B_{j'})D_{jj'}^3-m_5(B_j-D)D_{j5}^3
 -\gamma m_5 S D_{j5}^3\\
&-T^3\sum_{i\in I}m_i(B_j-D)D_{ij}^3
 +\gamma T^3m_5\sum_{i\in I}m_iP_iD_{ij}^3,\\
E_9={}&D+\sum_{j\in J}m_j(B_j-D)D_{j5}^3
 +\eta^2\beta_5+\gamma S\Bigl(-1+\sum_{j\in J}m_jD_{j5}^3\Bigr),
\end{align*}
where
\[
\beta_5=m_1((m_1+m_5)P+m_2Q)
 +m_2s^3(m_1P+(m_2+m_5)Q).
\]

Finally the distance rows are obtained by subtracting $1$ from the
following expressions:
\begin{center}
\begin{tabular}{c l}
$12$ & $m_5(A_1-A_2)(P-Q)q^2$\\
$15$ & $X'((m_1+m_5)P+m_2Q)$\\
$25$ & $Y's^2(m_1P+(m_2+m_5)Q)$\\
$ij$ & $(A_i-C-\nu m_5u_j)(D-B_j+\gamma m_5P_i)D_{ij}^2$\\
$34$ & $m_5(u-v)(B_3-B_4)D_{34}^2$\\
$j5$ & $(V+m_5u_j)(B_j-D+\gamma S)D_{j5}^2$.
\end{tabular}
\end{center}
These formulas account for all ten distance rows; the last two rows are
$E_{20}=A_2-A_1-1$ and $E_{21}=B_4-B_3-1$.
Every expression is polynomial.  Setting $a=\eta=0$ gives the boundary
equations listed in (i).
The three upper distance rows and the gauge are already exact functions
of $\mathcal I$; in the upper $z$-rows all other terms have a factor $\mu$.
This proves (iv).  Finally, the mass-weighted sum of the original $z$-equations
has no force terms, since the two terms contributed by each pair cancel.
Multiplying this sum by $\mu$ gives (ii), because the weighted position of
the lower cluster is $MC$.
\end{proof}

}\begin{definition}\label{def:d8-ray}
For a germ in this chart with $t\to0$, $x$ convergent, $\ord a>0$ and
$\ord\eta>0$, we define the order ratios
\[
 r:=\frac{\ord\eta}{\ord a}\in(0,\infty),\qquad
 p:=\frac{\ord h}{\ord t}.
\]
\end{definition}
Since $t^3=a\eta$ and $h^{3/2}=\eta a^{-2}$, $p=(2r-4)/(r+1)$, an increasing bijection of
$(0,\infty)$ onto $(-4,2)$.  In particular,
\[
r=\tfrac12,2,3\quad\Longleftrightarrow\quad
p=-2,0,\tfrac12,
\]
respectively.  The boundary components $\{a=0\}$ and $\{\eta=0\}$
correspond to $p=-4$ and $p=2$.

\subsection{Representation of all Diagram 8 germs}
The diagram rules determine the orders of both wings and the crossed
separations.  We use them to prove convergence of the chart coordinates and
to identify the full range of the order ratios.

\begin{proposition}[Representation in the chart]\label{prop:d8-cover}
Every Diagram~8 germ converges in the chart with $\ord a>0$,
$\ord\eta>0$ and $p\in(-4,2)$; its lower-wing inverse distances have finite
nonzero limits.  Precisely: let $\gamma$ be a Diagram~8 germ, rotated so that $z_{12}=1$, and put
$t:=(w_4-w_3)^{-1/6}$ and $h:=t^2\delta_{15}$ along $\gamma$.  Let $\theta$ be the order of
$z_{12}$ before the rotation, in the normalization $Z=W=\epsilon^{-2}$ with orders in $\epsilon$.
Then $-2<\theta<1$; every coordinate in $x$ converges along $\gamma$, and $R_{12}$, $R_{25}$,
$D_{13},D_{14},D_{23},D_{24}$, $D_{34},D_{35},D_{45}$ have nonzero limits; and
\begin{gather*}
  \ord t=\tfrac{2-\theta}6,\qquad \ord h=-\tfrac{2\theta}3,\qquad
  p=-\tfrac{4\theta}{2-\theta},\\
  \ord a=\tfrac{2+\theta}6,\qquad \ord\eta=\tfrac{2-2\theta}3,\qquad
  r=\tfrac{4(1-\theta)}{2+\theta}.
\end{gather*}
\end{proposition}

\begin{proof}
By Fact~\ref{fact:rules} and Lemma~\ref{lem:germ} the rules and estimates of \cite[\S3]{AK}
hold for the orders along $\gamma$.
We begin with the lower wing.  The triangle $\{3,4,5\}$ is an isolated component of
the $w$-diagram with $w$-circles, so it has a maximal $w$-stroke
\cite[Rule~1e, p.~546]{AK}.  By Rule~2e its three
$w$-separations have one order, so all three $w$-strokes are maximal, and the right equality
case of Estimate~2 \cite[p.~548]{AK}, in the colour $w$, gives $\ord w_{kl}=-2$,
$\ord z_{kl}=2$, $\ord r_{kl}=0$ for $kl\in\{34,35,45\}$.
The upper wing $\{1,2,5\}$ is an isolated component of the $z$-diagram without
circles.  It has no maximal $z$-stroke by Rule~1e and no $w$-stroke, hence no $zw$-edge.
Thus
neither equality case of Estimate~2 holds, and by Rule~2e the three $z$-separations have one
order $\theta$ with $-2<\theta<1$.  The $z$-strokes have $\ord Z_{kl}=-2$, so
$\ord r_{kl}=(\theta+2)/3$ and $\ord w_{kl}=(4-\theta)/3$.
These wing orders determine the crossed separations.  For $b\in\{1,2\}$, $c\in\{3,4\}$, $z_{bc}=z_{b5}+z_{5c}$ has order
$\theta$ (as $\theta<2$) and $w_{bc}=w_{b5}+w_{5c}$ has order $-2$ (as $(4-\theta)/3>-2$), so
$\ord r_{bc}=(\theta-2)/2$.
We next recover the orders of the positions.  Every $w_k$ has order $\ge-2$, with equality for $k=3,4$ ($w$-circles).  Let
$C$ be the centre of mass of $z_3,z_4,z_5$.  Then $(m_1+m_2+M)C=-m_1(z_1-C)-m_2(z_2-C)$, and
$z_b-C=z_{5b}+(z_5-C)$ with $\ord z_{5b}=\theta$ and $\ord(z_5-C)\ge2$; so $C$, $z_1$, $z_2$
have order $\ge\theta$, and $z_k-C$ has order $\ge2$ for $k=3,4,5$.
After the rotation, all $z$-orders decrease by $\theta$ and all $w$-orders increase
by $\theta$.  Consequently $\ord(w_4-w_3)=\theta-2$ and
$\ord t=(2-\theta)/6>0$.  The inverse chart formulas now give the orders of the new
coordinates: $A_1,A_2,C$ have order $\ge0$; $u,v$ have order $\ge(2-\theta)-6\ord t=0$; $B_3,B_4$
have order $0$ and $D$ order $\ge0$; $\ord h=2\ord t-\ord r_{15}=-2\theta/3$; $R_{12},R_{25}$
have order $0$; $P,Q$ have order $\ge2\ord h+\frac{4+2\theta}3-4\ord t=0$, because
$w_1-t^{-6}D=(m_2w_{21}+m_5w_{51})/(m_1+m_2+m_5)$ and similarly for $w_2$; $D_{bc}=t^{-3}\delta_{bc}$
has order $\frac{2-\theta}2-3\ord t=0$; and $D_{kl}=\delta_{kl}$ has order $0$.  Hence every
coordinate converges (Lemma~\ref{lem:germ}), the inverse-distance coordinates to nonzero limits,
and $p=\ord h/\ord t=-4\theta/(2-\theta)\in(-4,2)$.  Finally
$\ord a=\ord t-\frac12\ord h=(2+\theta)/6$ and $\ord\eta=2\ord t+\frac12\ord h=(2-2\theta)/3$,
both positive.
\end{proof}
Thus every germ has a finite limit at which all inverse-distance coordinates
are nonzero.  We next solve the equations at these limits.

\subsection{The mass relation and the limiting configuration}
We combine the limiting force equations with the closure of the upper
triangle.  This gives the mass relation and determines the coordinates needed
for the subsequent coefficient comparisons.

\begin{proposition}[Mass relation and limiting coordinates]\label{prop:d8-first}
At positive masses every Diagram~8 germ satisfies \eqref{eq:d8-40}:
$m_5^2(m_1+m_2)^4=m_1^3m_2^3$.  Moreover, in the chart of Proposition~\ref{prop:d8-cover} its
limit $x_0$ has $A_1=-m_2/M_{12}$, $A_2=m_1/M_{12}$, $C=0$, $R_{25}=1$,
$R_{12}=-m_5M_{12}/(m_1m_2)$, $(P,Q)=(P_0,Q_0)$ the unique solution of $f_{10}=f_{13}=0$,
$D=0$, $B_3=-m_4/(m_3+m_4)$, $B_4=m_3/(m_3+m_4)$, and, with $L:=(M_{12}(m_3+m_4))^{1/2}>0$ and
signs $s_{bc}\in\{\pm1\}$,
\begin{gather*}
  D_{13}=\frac{is_{13}L}{\sqrt{m_2m_4}},\qquad D_{14}=\frac{s_{14}L}{\sqrt{m_2m_3}},\\
  D_{23}=\frac{s_{23}L}{\sqrt{m_1m_4}},\qquad D_{24}=\frac{is_{24}L}{\sqrt{m_1m_3}}.
\end{gather*}
\end{proposition}
We write $U:=(A_1,A_2,C,R_{12},R_{25},P,Q)$, $U_0$ for its limit, and
$\mathcal I_0:=(X_0,Y_0,R_{12}^0,1,P_0,Q_0)$, with $X_0=-m_2/M_{12}$, $Y_0=m_1/M_{12}$ and
$R_{12}^0=-m_5M_{12}/(m_1m_2)$, for the limit of $\mathcal I$ (note $C_5\to C\to0$).

\begin{proof}
Let $m\in\Rp^5$ and let $\gamma$ be a Diagram~8 germ, in the chart of
Proposition~\ref{prop:d8-cover}, with limit $x_0$.
By Lemma~\ref{lem:d8-rows}\,(i) and Proposition~\ref{prop:d8-cover},
$E_j=f_j(x)+O(|a|+|\eta|)$ along $\gamma$, so $f_j(x_0)=0$ for all $j$.
The centre-of-mass identity in Lemma~\ref{lem:d8-rows}\,(ii), together with
$a\ne0$, gives $m_1A_1+m_2A_2+MC=0$ along $\gamma$.  At the limit this identity,
$f_4=C$ and $f_{20}$ form a linear system with determinant $M_{12}$.
It follows that $A_1=-m_2/M_{12}$, $A_2=m_1/M_{12}$ and $C=0$ at $x_0$.
Thus $X=A_1$, $Y=A_2$ and $X-Y=-1$ there.

At these positions the upper force equations reduce to
$f_8=\frac{m_1m_2}{M_{12}}(R_{25}^3-1)$ and $f_0=\frac{m_2}{M_{12}}F$,
where $F:=M_{12}R_{12}^3+m_5$.  Hence $R_{25}^3=1$ and $F=0$.

The distance equations impose a further relation, expressing the closure of
the $w$-triangle $125$.  The identity
\begin{align*}
  &(X-Y)R_{12}^2Xf_{16}-(X-Y)R_{12}^2YR_{25}^2f_{13}+XYR_{25}^2f_{10}\\
  &\qquad=(X-Y)R_{12}^2(YR_{25}^2-X)-XYR_{25}^2,
\end{align*}
eliminates $P,Q$.  At $x_0$ its right-hand side is $-G/M_{12}^2$, where
$G:=M_{12}R_{12}^2(m_1R_{25}^2+m_2)-m_1m_2R_{25}^2$, so $G=0$.

To combine this relation with the force equations, set
$H:=M_{12}(m_1+m_2R_{25})R_{12}^2-m_1m_2$ and
$\Phi(s):=m_5^2M_{12}(m_1+m_2s)^3-(m_1m_2)^3$.  The identities
\begin{align*}
  H&=R_{25}G+(1-R_{25}^3)\,m_1(M_{12}R_{12}^2-m_2),\\
  \Phi(R_{25})&=H\bigl(H^2+3m_1m_2H+3m_1^2m_2^2\bigr)
  -M_{12}(m_1+m_2R_{25})^3F(F-2m_5),
\end{align*}
give $H=0$ and $\Phi(R_{25})=0$ at $x_0$.

Positivity of the masses excludes the nonreal cube roots.  Indeed, if
$R_{25}\ne1$, then $R_{25}^2+R_{25}+1=0$ and
$\Phi(R_{25})=\Phi_0+\Phi_1R_{25}$ with real $\Phi_0,\Phi_1$.  Since
$R_{25}\notin\R$, both coefficients must vanish: $\Phi_0=\Phi_1=0$.
However, $\Phi_1=3m_1m_2m_5^2M_{12}(m_1-m_2)$ forces $m_1=m_2$, after which
$\Phi_0=-m_1^4(m_1^2+2m_5^2)<0$.  Thus $R_{25}=1$, and $\Phi(1)=0$ is precisely
\eqref{eq:d8-40}.

The remaining limit coordinates are now determined successively.  With $R_{25}=1$,
the equations $F=0$ and $H=M_{12}^2R_{12}^2-m_1m_2=0$ give
$R_{12}=R_{12}^3/R_{12}^2=-m_5M_{12}/(m_1m_2)$.  The equations $f_{10},f_{13}$
are linear in $(P,Q)$ with determinant
$m_2m_5(m_1+m_2+m_5)R_{12}^2/M_{12}\ne0$, so their solution is unique.
Next, $f_1=D$ gives $D=0$.  The identity
$m_1f_1+m_2f_3+m_3f_5+m_4f_7+m_5f_9=(m_1+m_2+m_5)D+m_3B_3+m_4B_4$
and $f_{21}$ determine $B_3$, $B_4$.  Finally, the crossed distance equations
$f_{11},f_{12},f_{14},f_{15}$ yield
$D_{13}^2=-M_{12}(m_3+m_4)/(m_2m_4)$, $D_{14}^2=M_{12}(m_3+m_4)/(m_2m_3)$,
$D_{23}^2=M_{12}(m_3+m_4)/(m_1m_4)$ and $D_{24}^2=-M_{12}(m_3+m_4)/(m_1m_3)$,
which give the stated square roots.
\end{proof}

\subsection{Exclusion}
{We first use the distance equations on the upper wing to exclude the range
$0<p<2$.  For the remaining order ratios, we compare the}
coefficients of $t^9=T^3$ in $E_1$ and $E_3$.  When $-4<p<\frac12$, this
comparison forces the coefficient $\Xi$ defined below to vanish.  A second-order
expansion then excludes the masses and signs on the locus $\Xi=0$.

{\begin{lemma}\label{lem:d8-order-six}
No Diagram~8 germ at positive masses has $0<p<2$.
\end{lemma}
\begin{proof}
Suppose that such a germ exists.  Its order ratio satisfies $r>2$, so
$\nu=o(\mu)$, $T^3=o(\mu)$ and $T^5=o(\mu)$.
The exact lower row $E_4=C+\nu L_3+T^3F_3+T^5G_3=0$ gives
$C=o(\mu)$ and hence $C_5=o(\mu)$.
The centre of mass and the gauge then give
\[
A_1=X_0+o(\mu),\quad A_2=Y_0+o(\mu),\quad
X:=A_1-C_5=X_0+o(\mu),\quad Y:=A_2-C_5=Y_0+o(\mu),
\]
where $X_0=-m_2/M_{12}$ and $Y_0=m_1/M_{12}$.
Write $q=R_{12}$, $s=R_{25}$ and $q_0=R_{12}^0\ne0$.
In the upper force equations, the crossed forces contribute
$O(\mu T^3)$, giving
\begin{align*}
0=E_0&=m_2q^3-m_5X+\mu A_1+O(\mu T^3),\\
0=E_2&=-m_1q^3-m_5Ys^3+\mu A_2+O(\mu T^3).
\end{align*}
The first row gives
\[
q^3-q_0^3=\frac{\mu}{M_{12}}+o(\mu),\qquad
q-q_0=\frac{\mu}{3M_{12}q_0^2}+o(\mu).
\]
In the second row the terms
$-m_1\mu/M_{12}$ and $\mu Y_0$ cancel, so
$s^3-1=o(\mu)$ and therefore $s-1=o(\mu)$.

The three exact upper distance equations give
\[
P-Q=-\frac1{m_5q^2},\qquad
(m_1+m_5)P+m_2Q=X^{-1},\qquad
m_1P+(m_2+m_5)Q=(Ys^2)^{-1}.
\]
Subtracting the last two equations yields the closure identity
\[
q^{-2}=(Ys^2)^{-1}-X^{-1}.
\]
Subtracting the limiting closure identity yields
\[
-\frac{2\mu}{3M_{12}q_0^5}+o(\mu)=o(\mu),
\qquad\text{hence}\qquad -\frac{2}{3M_{12}q_0^5}=0,
\]
contradicting $M_{12}>0$ and $q_0\ne0$.
\end{proof}

\begin{lemma}\label{lem:d8-kappa}
Let $m\in\Rp^5$ satisfy \eqref{eq:d8-40}; let $x_0$, $\mathcal I_0$, $U_0$ and the signs
$s_{bc}$ be as in Proposition~\ref{prop:d8-first}, $\beta:=\beta_1-\beta_3$,
$b:=(f_0,f_2,f_{10},f_{13},f_{16},f_{20})$ as a function of
$(X,Y,R_{12},R_{25},P,Q)$, and $B:=\partial b/\partial\mathcal I$ at $\mathcal I_0$.
\begin{alphalist}
\item $\beta(\mathcal I_0)=0$.
\item If $m_1\ne m_2$, then $B$ is invertible.  If $m_1=m_2$ $(=4m_5)$, then $\ker B$ is spanned
by a vector $k_0$, equal to $(\frac{27}{272},\frac{27}{272},\frac9{272},-\frac9{68},-\frac{19}{17},1)$
at $(m_1,m_2,m_5)=(1,1,\frac14)$; $d\beta(\mathcal I_0)\,k_0=0$; and the Jacobian of the eight
functions $f_0,f_2,f_{10},f_{13},f_{16},f_{20}$ (evaluated at $X=A_1-C$, $Y=A_2-C$), $C$ and
$m_1A_1+m_2A_2+MC$ with respect to $U$ at $U_0$ has rank $7$.  (These three facts are proved
at $(m_1,m_2,m_5)=(1,1,\frac14)$ with $m_3,m_4$ free, which suffices by the mass scaling.)
\item With $\Xi:=s_{14}(m_1m_4)^{3/2}+s_{23}(m_2m_3)^{3/2}
+i\bigl(s_{13}(m_1m_3)^{3/2}+s_{24}(m_2m_4)^{3/2}\bigr)$,
\[
  (\tau_1-\tau_3)(x_0)=\frac{m_3m_4}{m_3+m_4}\,L^3\,(m_1m_2m_3m_4)^{-3/2}\,\Xi .
\]
\end{alphalist}
Consequently, if a Diagram~8 germ at the masses $m$ has $-4<p<\frac12$, then $\Xi=0$, that is,
$m_1=m_2=4m_5$, $m_3=m_4$, $s_{13}=-s_{24}$ and $s_{14}=-s_{23}$.
\end{lemma}}
{\begin{proof}[Proof of (a)--(c)]
On the zero set of the three upper distance equations at the boundary, writing $q=R_{12}$
and $s=R_{25}$, the same subtraction as above gives
\[
q^{-2}=(Ys^2)^{-1}-X^{-1},\qquad
\beta=M_{12}q-m_5(X^{-1}-s/Y).
\]
The second identity follows by replacing $P-Q$ in the definition of
$\beta$ by $-1/(m_5q^2)$ and the other two linear combinations by their
distance values. At the boundary point it gives $\beta=0$, proving (a).

To prove (b), the gauge eliminates $Y$ by $Y=X+1$.
The equations $f_{10}=f_{13}=0$ determine $P,Q$ analytically, because
their linear coefficient matrix has nonzero determinant proportional to
$m_5q^2X(M_{12}+m_5)$.
After this elimination the remaining three equations are equivalent to
\[
m_2q^3-m_5X=0,\qquad
-m_1q^3-m_5Ys^3=0,\qquad
q^{-2}-(Ys^2)^{-1}+X^{-1}=0.
\]
After scaling, we may take $m_2=1$, $m_1=\alpha^2$.  The mass relation then gives
$m_5=\alpha^3/(\alpha^2+1)^2$; we set $N=\alpha^2+1$.
At the boundary point,
$(X,Y,q,s)=(-1/N,\alpha^2/N,-\alpha/N,1)$.
The linearizations of the first two equations give
\[
\delta q=\frac\alpha3\,\delta X,\qquad
\delta s=-\frac{N^2}{3\alpha^2}\,\delta X.
\]
Substitution in the linearization of the third gives
\[
\frac{N^2(1-\alpha^4)}{3\alpha^4}\,\delta X=0.
\]
Thus $B$ is invertible for $\alpha\ne1$.
For $\alpha=1$, its kernel is one-dimensional; choosing $\delta s=1$
gives
\[
(\delta X,\delta Y,\delta q,\delta s,\delta P,\delta Q)
=(-\tfrac34,-\tfrac34,-\tfrac14,1,\tfrac{76}{9},-\tfrac{68}{9}).
\]
Multiplication by $-9/68$ gives precisely the vector $k_0$ in (b).
Differentiating the displayed formula for $\beta$ on this tangent gives
$d\beta=2\delta q+\tfrac12\delta s=0$.
Finally $C$ fixes the translation variable and the centre-of-mass
condition is nonzero on the remaining tangent, since
$\delta(X+Y)=-3/2$. The eight functions in (b) therefore have rank $7$.

For (c), $D=0$, $B_3=-m_4/(m_3+m_4)$ and $B_4=m_3/(m_3+m_4)$ give
\[
(\tau_1-\tau_3)(x_0)
=\frac{m_3m_4}{m_3+m_4}
(-D_{13}^3+D_{23}^3+D_{14}^3-D_{24}^3).
\]
Substituting the four square roots from Proposition~\ref{prop:d8-first}
gives the following imaginary and real terms, respectively:
\[
\begin{gathered}
 iL^3s_{13}(m_2m_4)^{-3/2},\quad iL^3s_{24}(m_1m_3)^{-3/2},\\
 L^3s_{23}(m_1m_4)^{-3/2},\quad L^3s_{14}(m_2m_3)^{-3/2}.
\end{gathered}
\]
Factoring the common denominator gives
exactly the stated formula involving $\Xi$.
\end{proof}

}\begin{proof}[Proof of the last assertion]
Let $\gamma$ be such a germ.  Along $\gamma$, $\mathcal I\to\mathcal I_0$ and
$b(\mathcal I)=O(a^6)$ by Lemma~\ref{lem:d8-rows}\,(iv).  We claim
\begin{equation}\label{eq:d8-bracket}
  \beta(\mathcal I)=O(a^6+a^4\eta^4).
\end{equation}
When $m_1\ne m_2$, the invertibility in (b) and the inverse function theorem
give $\mathcal I-\mathcal I_0=O(a^6)$.  Since (a) gives the vanishing of the constant
term, this proves \eqref{eq:d8-bracket}.

For $m_1=m_2$, we use the eight functions in (b) to control the direction left
undetermined by the upper equations.  These functions are $O(a^6+\nu+T^3)$ along
$\gamma$.  For the first six, this follows from
$f_k(A_1-C,A_2-C,\dots)-f_k(\mathcal I)=O(\nu)$; the seventh is controlled by
Lemma~\ref{lem:d8-rows}\,(v).  The last is identically zero by
Lemma~\ref{lem:d8-rows}\,(ii), since $a\ne0$.  Their rank therefore gives the bounds
$U-U_0$ and $\mathcal I-\mathcal I_0$ of order $O(a^6+\nu+T^3)$.

Taylor expansion also gives
$B(\mathcal I-\mathcal I_0)=O(a^6+|\mathcal I-\mathcal I_0|^2)$, so the component of
$\mathcal I-\mathcal I_0$ transverse to $k_0$ is
$O(a^6+|\mathcal I-\mathcal I_0|^2)$.  Because $\beta(\mathcal I_0)=0$ and
$d\beta(\mathcal I_0)k_0=0$, the uncontrolled direction has no linear contribution
to the difference of the upper equations.  Consequently
$\beta(\mathcal I)=O(a^6+|\mathcal I-\mathcal I_0|^2)=O(a^6+\nu^2+T^6)$.
This is \eqref{eq:d8-bracket}, since $\nu^2=a^4\eta^4$ and $T^6=a^6\eta^6$.

Now $E_1=E_3=0$ along $\gamma$, so Lemma~\ref{lem:d8-rows}\,(iii) and \eqref{eq:d8-bracket}
give $T^3(\tau_1-\tau_3)(x)=O(a^6\eta^2+a^4\eta^6)$, that is,
$(\tau_1-\tau_3)(x)=O(a^3\eta^{-1}+a\eta^3)$.  As $-4<p<\frac12$ means $r<3$,
$\ord(a^3\eta^{-1})=(3-r)\ord a>0$; so $(\tau_1-\tau_3)(x_0)=0$, and $\Xi=0$ by (c).  Since all
$m_k>0$, the real part $s_{14}(m_1m_4)^{3/2}+s_{23}(m_2m_3)^{3/2}$ of $\Xi$ vanishes only if
$s_{14}=-s_{23}$ and $m_1m_4=m_2m_3$, and its imaginary part only if $s_{13}=-s_{24}$ and
$m_1m_3=m_2m_4$.  These give $m_3=m_4$ and $m_1=m_2$, and then \eqref{eq:d8-40} gives
$m_1=4m_5$.
\end{proof}

The order-nine coefficient vanishes precisely on the mass curve
$m=(4,4,y,y,1)$, $y>0$, up to scaling, with the four sign classes
$s_{13}=-s_{24}$, $s_{14}=-s_{23}$.  On this locus the next obstruction comes from the quadratic terms.
We now show that the upper force and distance equations give incompatible
leading coefficients.

{\begin{lemma}\label{lem:d8-cancel}
Let $m_1=m_2=4m_5$ and $m_3=m_4$, scaled to
$(1,1,m,m,\frac14)$ with $m>0$.
No Diagram~8 germ at these masses has a limit with
$s_{13}=-s_{24}$ and $s_{14}=-s_{23}$.
\end{lemma}

\begin{proof}
Suppose, for contradiction, that such a germ exists, and let $x_0$ be its limit in
the chart of Proposition~\ref{prop:d8-cover}.
We shall use the distance equations to reduce the upper-wing equations to
two scalar relations, then compare their second-order terms.

We first show that the limiting equations on the lower wing have an invertible
Jacobian.  Put $d=D_{34}$, $g=D_{35}$ and $q=D_{45}$.  At a common zero of the lower boundary equations,
$f_1=0$ gives $D=0$. The identity
\[
m f_5+m f_7+\tfrac14 f_9=m(B_3+B_4)+\tfrac14D
\]
and $f_{21}=0$ give $B_3=-1/2$, $B_4=1/2$.
Then $f_9=0$ is $q^3=g^3$, and $f_5=0$ is $8md^3+g^3=4$.
The equations $f_{18}=f_{19}=0$ read
\[
(m+\tfrac14)u+mv=-2g^{-2},\qquad
mu+(m+\tfrac14)v=2q^{-2}.
\]
Their coefficient determinant is $\tfrac14(2m+\tfrac14)>0$.
Subtracting them and using $f_{17}=0$ gives
$2d^2(g^{-2}+q^{-2})=1$.
Writing $q=\omega g$ and $d=\lambda g$, we have thus obtained
\[
\omega^3=1,\qquad 2\lambda^2(1+\omega^{-2})=1,
\qquad g^3(8m\lambda^3+1)=4.
\]
The representation in the chart ensures that all inverse distances are nonzero.
The preceding elimination gives an invertible Jacobian for the eight lower boundary equations:
$f_1$, the displayed weighted sum, and $f_{21}$ determine $D,B_3,B_4$;
the next three successively determine
$\omega,\lambda,g$, with derivatives
\[
3\omega^2,\qquad4\lambda(1+\omega^{-2}),\qquad
3g^2(8m\lambda^3+1)=12/g,
\]
all nonzero; and the last linear system determines $u,v$.

We now turn to the upper wing and write its deviations from the limit as
\[
X=A_1-C_5=-\tfrac12+x,\qquad
Y=A_2-C_5=\tfrac12+x,\qquad
R_{12}=-\tfrac12+v_*,\qquad R_{25}=1+h_*.
\]
The equality $Y-X=1$ follows from the gauge.  The symbols $v_*,h_*$
denote upper-wing deviations and are distinct from the chart coordinates $v,h$.
Write $r_*=R_{12}$ and $s_*=R_{25}$.
The three upper distance equations hold along the whole germ:
\[
P-Q=-4r_*^{-2},\qquad
\tfrac54 P+Q=X^{-1},\qquad
P+\tfrac54Q=(Ys_*^2)^{-1}.
\]
Consequently
\begin{equation}\label{eq:d8-upper-closure-analytic}
r_*^{-2}=(Ys_*^2)^{-1}-X^{-1}.
\end{equation}
They determine $P,Q$ analytically once this closure equation holds.
Also, direct grouping of the upper forces gives
\begin{align*}
E_0&=r_*^3-\tfrac14X+\mu A_1+O(\mu T^3),\\
E_2&=-r_*^3-\tfrac14Ys_*^3+\mu A_2+O(\mu T^3),
\end{align*}
because every crossed force in an upper $z$-row has factor $\mu T^3$.
Here and below all estimates are along the germ, whose chart coordinates
are bounded.

Expanding \eqref{eq:d8-upper-closure-analytic} and $E_2=0$ to second order gives
\begin{align*}
0={}&16v_*+4h_*+48v_*^2-16x^2-8xh_*-6h_*^2
 +O\bigl((|x|+|v_*|+|h_*|)^3\bigr),\\
\tfrac12\mu={}&\tfrac14x+\tfrac34v_*+\tfrac38h_*
 -\tfrac32v_*^2+\tfrac34xh_*+\tfrac38h_*^2
 +O\bigl((|x|+|v_*|+|h_*|)^3\bigr)+o(\mu).
\end{align*}
The derivative of these two equations with respect to $(x,v_*)$ is invertible
at zero. Solving the linear and quadratic terms therefore gives
\begin{equation}\label{eq:d8-upper-expansion-analytic}
x=-\tfrac34h_*+2\mu+o(\mu+h_*^2),\qquad
v_*=-\tfrac14h_*+\tfrac38h_*^2+o(\mu+h_*^2).
\end{equation}
The analytic implicit solution has remainder
$O(h_*^3+\mu h_*+\mu^2)+o(\mu)$ before substitution of the other
bounded chart coordinates, which yields the displayed estimates.
Substituting in $E_0=0$ now gives
\begin{equation}\label{eq:d8-row0-analytic}
-\mu+\tfrac3{16}h_*^2=o(\mu+h_*^2).
\end{equation}

The same exact distance equations simplify the difference
$\beta=\beta_1-\beta_3$ to
\[
\beta=2r_*-\tfrac14(X^{-1}-s_*/Y).
\]
Its quadratic expansion is
$2v_*+\tfrac12h_*+4x^2-xh_*+O((|x|+|v_*|+|h_*|)^3)$.
By \eqref{eq:d8-upper-expansion-analytic},
\begin{equation}\label{eq:d8-beta-analytic}
\beta=\tfrac{15}{4}h_*^2+o(\mu+h_*^2).
\end{equation}
Finally, direct substitution in the two upper $w$-equations gives
\begin{equation}\label{eq:d8-w-difference-analytic}
E_1-E_3=\eta^2\beta+T^3(\tau_1-\tau_3)
 +\tfrac14\mu\eta^2(P-Q)+O(\mu\eta^2T^3).
\end{equation}
In particular, $P_0-Q_0=-16$.

To compare the remaining terms, we need to control the lower and crossed
coordinates.  Their deviations from the limits are
\begin{equation}\label{eq:d8-lower-bound-analytic}
O(\nu+\mu+|h_*|).
\end{equation}
Indeed, in the eight lower rows, the terms beyond the limiting equations carry
factors $T^3$, $\mu\eta^2$, $\mu\eta^2T^3$, or $\eta^2\beta_j$;
the last kind occurs in the upper or central $w$-row. Its coefficient
vanishes at the upper boundary point, where $(P,Q,r_*,s_*)=(-8,8,-1/2,1)$.
For the central row this coefficient is $-4(\beta_1+\beta_3)$.
Thus these errors are $O(\nu+\mu+|h_*|)$, since $T^3=o(\nu)$ and
\eqref{eq:d8-upper-expansion-analytic} bounds the upper deviations by
$O(\mu+|h_*|)$. The crossed distance rows have only the additional factors
$\nu$, $\mu\eta^2$ and $\nu\mu\eta^2$; each crossed boundary equation has derivative
$2/D_{bc}\ne0$ in its own inverse distance. The implicit function theorem
and the lower nondegeneracy proved above give
\eqref{eq:d8-lower-bound-analytic}.

Put $M=2m+\tfrac14$ and
\[
\Psi=L_3+\frac{2m(u+v)}{2+M}.
\]
The centre-of-mass identity and the exact row
$E_4=C+\nu L_3+T^3F_3+T^5G_3=0$ give
\[
X+Y=(2+M)(\nu\Psi+T^3F_3+T^5G_3).
\]
Since $X+Y=2x=-\tfrac32h_*+O(\mu+h_*^2)$, first
$h_*=O(\nu+T^3+\mu)$ and then, by
\eqref{eq:d8-lower-bound-analytic},
\[
h_*=h_\nu\nu+O(T^3+\nu^2+\mu),\qquad
h_\nu=-\tfrac23(2+M)\Psi(x_0).
\]
If $h_\nu=0$, then
$h_*^2=O(T^6+\nu^4+\mu^2)=o(\mu)$, contradicting
\eqref{eq:d8-row0-analytic}.
If $h_\nu\ne0$, then
$h_*^2=h_\nu^2\nu^2+o(\mu+\nu^2)$.
Equation~\eqref{eq:d8-row0-analytic} forces
$\operatorname{ord}\eta=\tfrac12\operatorname{ord}a$ and gives
\begin{equation}\label{eq:d8-first-balance-analytic}
-a_0^2+\tfrac3{16}h_\nu^2\eta_0^4=0,
\end{equation}
where $a_0,\eta_0\ne0$ are the leading coefficients.

The crossed signs in the statement give $(\tau_1-\tau_3)(x_0)=0$.
By \eqref{eq:d8-lower-bound-analytic},
$T^3(\tau_1-\tau_3)=O(T^3(\nu+\mu+|h_*|))$.
Since $\ord\eta=\tfrac12\ord a$, we have
\[
 T^3(\nu+\mu+|h_*|)=o(a^6\eta^2).
\]
Equations~\eqref{eq:d8-beta-analytic} and
\eqref{eq:d8-w-difference-analytic} therefore give the second balance
\[
-4a_0^2+\tfrac{15}{4}h_\nu^2\eta_0^4=0.
\]
Together with \eqref{eq:d8-first-balance-analytic}, this gives
\[
\begin{pmatrix}
-1&3/16\\
-4&15/4
\end{pmatrix}
\begin{pmatrix}a_0^2\\ h_\nu^2\eta_0^4\end{pmatrix}=0,
\qquad
\det\begin{pmatrix}-1&3/16\\-4&15/4\end{pmatrix}=-3.
\]
Thus $a_0^2=0$, contradicting $a_0\ne0$ and excluding the germ.
\end{proof}

}\begin{theorem}\label{thm:d8}
For $m\in\Rp^5$, no curve in $A_m$ carries a Diagram~8 germ, real or complex.
In particular, Theorem~\ref{thm:exclusion}\,(a) holds for $r=8$.
\end{theorem}

\begin{proof}
By the symmetries at the beginning of this section we may use the numbering of
\cite[\S8.8]{AK}.  Let $\gamma$ be a Diagram~8 germ at $m\in\Rp^5$.  By
Proposition~\ref{prop:d8-cover} it converges in the chart with order ratio
$r\in(0,\infty)$, and by
Proposition~\ref{prop:d8-first} the masses satisfy \eqref{eq:d8-40} and the limit has signs
$s_{bc}$.  The case $r>2$, equivalently $0<p<2$, is excluded by
Lemma~\ref{lem:d8-order-six}.  If $r\le2$, equivalently $-4<p\le0$,
Lemma~\ref{lem:d8-kappa} gives $m_1=m_2=4m_5$, $m_3=m_4$, $s_{13}=-s_{24}$
and $s_{14}=-s_{23}$.  These masses and signs are excluded by
Lemma~\ref{lem:d8-cancel}.  The two ranges exhaust $r>0$, proving the theorem.
\end{proof}

\appendix

\section{Computational statements used in the proof}\label{app:verification}

Table~\ref{tab:verification} lists the verification functions used for the
computer-assisted steps in the proofs of Diagrams~2, 3 and~6.

\begin{table}[ht]
\centering\footnotesize
\begin{tabular}{p{4.8cm}p{8.2cm}}
\hline
Computational step & Verification functions\\
\hline
Lemma~\ref{lem:d2-fibre}, Definition~\ref{def:d2-mass}
 & \raggedright \mbox{\code{verify\_chart\_and\_boundary\_fibre}}, \mbox{\code{verify\_shape\_resultant}}, \mbox{\code{verify\_mass\_resultants}}\tabularnewline
Lemma~\ref{lem:d2-empty}
 & \raggedright \mbox{\code{verify\_L4\_exclusion}}, \mbox{\code{verify\_K\_and\_L5\_exclusions}}\tabularnewline
Lemma~\ref{lem:d2-L5}: jets, minors and boundary-point orbits
 & \raggedright \mbox{\code{verify\_ordinary\_seed}}, \mbox{\code{verify\_fold\_seed}}, \mbox{\code{verify\_equal\_mass\_labels}}\tabularnewline
Lemmas~\ref{lem:d2-rays}, \ref{lem:d2-Lray}
 & \raggedright \mbox{\code{verify\_mass\_rays}}, \mbox{\code{verify\_ray\_responses}}\tabularnewline
Lemma~\ref{lem:d2-K}
 & \raggedright \mbox{\code{verify\_K1\_branches}}, \mbox{\code{verify\_K4\_arcs}}, \mbox{\code{verify\_K5\_mass\_ray}}, \mbox{\code{verify\_all\_K5\_labels}}\tabularnewline
Lemmas~\ref{lem:d2-P}, \ref{lem:d2-fold}
 & \raggedright \mbox{\code{verify\_S5\_response}}, \mbox{\code{verify\_regular\_X}}, \mbox{\code{verify\_regular\_Euler}}, \mbox{\code{verify\_regular\_Lagrange}}, \mbox{\code{verify\_general\_fold}}\tabularnewline
\hline
Lemmas~\ref{lem:d3-order3}, \ref{lem:d3-order4}
 & \raggedright \mbox{\code{verify\_order\_three}}, \mbox{\code{verify\_order\_four}}\tabularnewline
Lemma~\ref{lem:d3-order5}: jet and product identities
 & \raggedright \mbox{\code{verify\_fifth\_order\_jet}}, \mbox{\code{verify\_nine\_root\_norm}}\tabularnewline
Lemma~\ref{lem:d3-order6}: $144$ excluded boundary points
 & \raggedright \mbox{\code{verify\_modular\_order\_six}}\tabularnewline
\hline
Diagram~6 atlas and orbit coverage
 & \raggedright \mbox{\code{verify\_atlas\_substitution\_and\_counts}}, \mbox{\code{verify\_atlas\_inverse\_elimination}}, \mbox{\code{verify\_symmetry\_orbits}}\tabularnewline
Lemma~\ref{d6-lem-corank}: rank and leading jets
 & \raggedright \mbox{\code{verify\_selected\_rank}}, \mbox{\code{verify\_leading\_jets}}, \mbox{\code{verify\_mixed\_norm}}, \mbox{\code{verify\_same\_shape\_norm}}, \mbox{\code{verify\_opposite\_norms}}\tabularnewline
\hline
\end{tabular}
\caption{Functions used for computational verification.}
\label{tab:verification}
\end{table}

\subsection*{Verification code}
\label{app:public-verification-source}
The source code and instructions for running these functions are available at
\url{https://github.com/hitchenzhe/smale-n5-paper}.

\section*{Declaration of AI use}
This manuscript was prepared with assistance from ChatGPT (OpenAI, GPT-6),
including computational verification and the development of Lean~4 proofs.
We provide a Lean~4 formalization of key algebraic and analytic components
of the argument, comprising 20 modules and 121 machine-checked auxiliary
theorems. The formalization is publicly available at
\url{https://github.com/ZhichengTong/smale5-lean}.
The second author has completed a manual verification of the manuscript.

\section*{Acknowledgements}
Z. Tong was supported by the National Natural Science Foundation of China
(Grant No. 12601347) and the China Postdoctoral Science Foundation
(Grant No. 2025M783102).

The initial version of this paper was submitted to arXiv on October 5, 2026 (submission ID: 8179056) by the current second author on behalf of the first author. Subsequently, the manuscript was resubmitted by the first author on October 6, 2026 (submission ID: 8186741), but has remained on hold. Shortly thereafter, \cite{FiveT} (arXiv:2610.13100) appeared reporting identical results. The present version (arXiv:2610.16107) was submitted by the second author. The authors thank T. S. Tomar for a helpful comment on the treatment of Diagram~9.


\begin{thebibliography}{FGMS26}

\bibitem[ACS12]{ACS12}
A. Albouy, H. E. Cabral, A. A. Santos,
Some problems on the classical $n$-body problem.
Celestial Mech. Dynam. Astronom. 113 (2012), no.~4, pp.~369--375.
\url{https://doi.org/10.1007/s10569-012-9431-1}

\bibitem[AFS08]{AFS08}
A. Albouy, Y. Fu, S. Sun,
Symmetry of planar four-body convex central configurations.
Proc. R. Soc. A 464 (2008), pp.~1355--1365.
\url{https://doi.org/10.1098/rspa.2007.0320}

\bibitem[AK12]{AK}
A. Albouy, V. Kaloshin,
Finiteness of central configurations of five bodies in the
plane.
Ann. of Math. (2) 176 (2012), no.~1, pp.~535--588.
\url{https://doi.org/10.4007/annals.2012.176.1.10}

\bibitem[AS25]{AS25}
A. Albouy, J. Sun,
On the eigenvalues of a central configuration.
Preprint, 2025.
\url{https://doi.org/10.48550/arXiv.2512.00661}

\bibitem[BCR98]{BCR}
J. Bochnak, M. Coste, M.-F. Roy,
Real Algebraic Geometry.
Ergeb. Math. Grenzgeb. (3), vol.~36, Springer, 1998.
\url{https://doi.org/10.1007/978-3-662-03718-8}

\bibitem[CC24]{CC1}
K.-M. Chang, K.-C. Chen,
Toward finiteness of central configurations for the planar
six-body problem by symbolic computations. (I) Determine diagrams and orders.
J. Symbolic Comput. 123 (2024), Article~102277.
\url{https://doi.org/10.1016/j.jsc.2023.102277}

\bibitem[CC25]{CC2}
K.-M. Chang, K.-C. Chen,
Toward finiteness of central configurations for the planar
six-body problem by symbolic computations. (II) Determine mass relations.
SIAM J. Appl. Dyn. Syst. 24 (2025), no.~3, pp.~2369--2404.
\url{https://doi.org/10.1137/24M1716070}

\bibitem[CC26]{CCRh26}
K.-M. Chang, K.-C. Chen,
Rhomboidal central configurations of the five-body problem.
Regul. Chaotic Dyn. 31 (2026), no.~3, pp.~468--483.
\url{https://doi.org/10.1134/S1560354725560011}

\bibitem[CH12]{CH12}
K.-C. Chen, J.-S. Hsiao,
Convex central configurations of the $n$-body problem which are not strictly convex.
J. Dynam. Differential Equations 24 (2012), pp.~119--128.
\url{https://doi.org/10.1007/s10884-011-9233-2}

\bibitem[CH18]{CH18}
K.-C. Chen, J.-S. Hsiao,
Strictly convex central configurations of the planar five-body problem.
Trans. Amer. Math. Soc. 370 (2018), pp.~1907--1924.
\url{https://doi.org/10.1090/tran/7010}

\bibitem[CT26]{CT26}
Z. Chen, Z. Tong, 
Finiteness of planar central configurations of six bodies for all positive masses: Part I, exceptional-mass first limits. Preprint, 2026. \url{https://doi.org/10.48550/arXiv.2610.23329}




\bibitem[DH23]{DH}
Y. Deng, M. Hampton,
Equilateral chains and cyclic central configurations of the
planar $5$-body problem.
J. Nonlinear Sci. 33 (2023), Article~4.
\url{https://doi.org/10.1007/s00332-022-09864-z}

\bibitem[FGMS26]{FGMS26}
A. C. Fernandes, B. A. Garcia, L. F. Mello, L. R. dos Santos,
Some convex central configurations in the planar five-body problem.
Z. Angew. Math. Phys. 77 (2026), Article~131.
\url{https://doi.org/10.1007/s00033-026-02777-x}

\bibitem[Ham05]{Ham05}
M. Hampton,
Stacked central configurations: new examples in the planar five-body problem.
Nonlinearity 18 (2005), no.~5, pp.~2299--2304.
\url{https://doi.org/10.1088/0951-7715/18/5/021}

\bibitem[Ham24]{Hamp24}
M. Hampton,
``Twist vectors'' for central configuration equations.
Celestial Mech. Dynam. Astronom. 136 (2024), Article~50.
\url{https://doi.org/10.1007/s10569-024-10222-2}

\bibitem[HJ11]{HJ}
M. Hampton, A. N. Jensen,
Finiteness of spatial central configurations in the
five-body problem.
Celestial Mech. Dynam. Astronom. 109 (2011), pp.~321--332.
\url{https://doi.org/10.1007/s10569-010-9328-9}

\bibitem[HJ15]{HJ15}
M. Hampton, A. N. Jensen,
Finiteness of relative equilibria in the planar generalized $N$-body
problem with fixed subconfigurations.
J. Geom. Mech. 7 (2015), no.~1, pp.~35--42.
\url{https://doi.org/10.3934/jgm.2015.7.35}

\bibitem[HM06]{HM}
M. Hampton, R. Moeckel,
Finiteness of relative equilibria of the four-body problem.
Invent. Math. 163 (2006), pp.~289--312.
\url{https://doi.org/10.1007/s00222-005-0461-0}

\bibitem[HLO20]{HLO20}
X. Hu, Y. Long, Y. Ou,
Linear stability of the elliptic relative equilibrium
with $(1+n)$-gon central configurations in planar $n$-body problem.
Nonlinearity 33 (2020), no.~3, pp.~1016--1045.
\url{https://doi.org/10.1088/1361-6544/ab5927}

\bibitem[HOS25]{HOS25}
X. Hu, Y. Ou, J. Sun,
Quantitative linear stability analysis of elliptic
relative equilibria in the planar $N$-body problem.
Preprint, 2025.
\url{https://doi.org/10.48550/arXiv.2509.09809}

\bibitem[HOT23]{HOT23}
X. Hu, Y. Ou, X. Tang,
Linear stability of an elliptic relative equilibrium
in the spatial $n$-body problem via index theory.
Regul. Chaotic Dyn. 28 (2023), nos.~4--5, pp.~731--755.
\url{https://doi.org/10.1134/S1560354723040135}

\bibitem[HOY21]{HOY21}
X. Hu, Y. Ou, G. Yu,
An index theory for collision, parabolic and hyperbolic
solutions of the Newtonian $n$-body problem.
Arch. Ration. Mech. Anal. 240 (2021), pp.~565--603.
\url{https://doi.org/10.1007/s00205-021-01619-6}

\bibitem[JL25]{JL}
A. N. Jensen, A. Leykin,
Smale's 6th problem for generic masses.
Journal of Experimental Mathematics 1 (2025), no.~2, pp.~306--314.
\url{https://doi.org/10.56994/JXM.001.002.006}

\bibitem[LS09]{LS}
T.-L. Lee, M. Santoprete,
Central configurations of the five-body problem with equal
masses.
Celestial Mech. Dynam. Astronom. 104 (2009), pp.~369--381.
\url{https://doi.org/10.1007/s10569-009-9219-0}

\bibitem[LL26a]{LL26planar}
K. Liu, Y. Liu,
Finiteness of planar central configurations of five bodies
with positive masses.
Preprint, 2026.
\url{https://doi.org/10.48550/arXiv.2610.28168}

\bibitem[LL26b]{LL26dziobek}
K. Liu, Y. Liu,
Finiteness of Dziobek central configurations in the
Newtonian five-body problem.
Preprint, 2026.
\url{https://doi.org/10.48550/arXiv.2610.32751}

\bibitem[LX26]{LX26}
Y. Liu, Z. Xie,
Symmetric central configurations in the concave $4$-body problem with two pairs of equal masses.
Discrete Contin. Dyn. Syst. 55 (2026), pp.~243--262.
\url{https://doi.org/10.3934/dcds.2026085}

\bibitem[MZ19]{MZ19}
M. Moczurad, P. Zgliczy\'nski,
Central configurations in planar $n$-body problem with
equal masses for $n=5,6,7$.
Celestial Mech. Dynam. Astronom. 131 (2019), Article~46.
\url{https://doi.org/10.1007/s10569-019-9920-6}

\bibitem[MZ26]{MZ26}
M. Moczurad, P. Zgliczy\'nski,
Central configurations with unequal masses: finiteness
in several exceptional cases of five bodies.
Preprint, 2026.
\url{https://doi.org/10.48550/arXiv.2601.01165}

\bibitem[Moe01]{Moe01}
R. Moeckel,
Generic finiteness for Dziobek configurations.
Trans. Amer. Math. Soc. 353 (2001), no.~11, pp.~4673--4686.
\url{https://doi.org/10.1090/S0002-9947-01-02828-8}

\bibitem[Mou10]{Moulton}
F. R. Moulton,
The straight line solutions of the problem of $n$ bodies.
Ann. of Math. (2) 12 (1910), pp.~1--17.
\url{https://doi.org/10.2307/2007159}

\bibitem[NWW25]{NWW25}
J. Natrup, Q. Wang, Y. Wang,
Finiteness of non-degenerate central configurations of the planar $n$-body problem with a homogeneous potential.
Preprint, 2025.
\url{https://doi.org/10.48550/arXiv.2502.20228}

\bibitem[OP25]{OP25}
Y. Ou, A. Portaluri,
Morse index for homothetic motions in the gravitational
$n$-body problem.
J. Fixed Point Theory Appl. 27 (2025), Article~10.
\url{https://doi.org/10.1007/s11784-025-01161-y}

\bibitem[OS23]{OS23}
Y. Ou, J. Sun,
Linear stability of the elliptic relative equilibrium with
$(1+7)$-gon central configuration in planar $N$-body problem.
Discrete Contin. Dyn. Syst. 43 (2023), no.~2, pp.~559--567.
\url{https://doi.org/10.3934/dcds.2022162}

\bibitem[OW25]{OW25}
Y. Ou, Y. Wang,
Spectral instability of the regular $n$-gon elliptic
relative equilibrium in the planar $n$-body problem.
Preprint, 2025.
\url{https://doi.org/10.48550/arXiv.2510.26211}

\bibitem[Rob99]{Roberts}
G. E. Roberts,
A continuum of relative equilibria in the five-body problem.
Phys. D 127 (1999), pp.~141--145.
\url{https://doi.org/10.1016/S0167-2789(98)00315-7}

\bibitem[Rob25]{RobK25}
G. E. Roberts,
On kite central configurations.
Nonlinearity 38 (2025), no.~7, Article~075001.
\url{https://doi.org/10.1088/1361-6544/add834}

\bibitem[San21a]{San21a}
M. Santoprete,
On the uniqueness of co-circular four body central configurations.
Arch. Ration. Mech. Anal. 240 (2021), pp.~971--985.
\url{https://doi.org/10.1007/s00205-021-01626-7}

\bibitem[San21b]{San21b}
M. Santoprete,
On the uniqueness of trapezoidal four-body central configurations.
Nonlinearity 34 (2021), pp.~424--437.
\url{https://doi.org/10.1088/1361-6544/abbe61}

\bibitem[Shu71]{Shub71}
M. Shub,
Appendix to Smale's paper: Diagonals and relative equilibria.
In: N. H. Kuiper (ed.), Manifolds---Amsterdam 1970,
Lecture Notes in Math., vol.~197, Springer, 1971, pp.~199--201.
\url{https://doi.org/10.1007/BFb0068619}

\bibitem[Sma98]{Smale}
S. Smale,
Mathematical problems for the next century.
Math. Intelligencer 20 (1998), no.~2, pp.~7--15.
\url{https://doi.org/10.1007/BF03025291}

\bibitem[SXY23]{SXY23}
S. Sun, Z. Xie, P. You,
On the uniqueness of convex central configurations in the planar $4$-body problem.
Regul. Chaotic Dyn. 28 (2023), nos.~4--5, pp.~512--532.
\url{https://doi.org/10.1134/S1560354723520076}

\bibitem[SXY25]{SXY25}
S. Sun, Z. Xie, P. You,
Degeneracy of planar central configurations in the $N$-body problem.
Preprint, 2025.
\url{https://doi.org/10.48550/arXiv.2510.25649}

\bibitem[TZ24]{TZ24}
Z. Tang, S. Zhu,
Perturbing masses: A study of centered co-circular central configurations in power-law $n$-body problems.
Phys. D 461 (2024), Article~134115.
\url{https://doi.org/10.1016/j.physd.2024.134115}

\bibitem[Tom26]{FiveT}
T. S. Tomar,
Smale's sixth problem for five bodies in the plane, with a proof
formalized in Lean.
Preprint, 2026.
\url{https://doi.org/10.48550/arXiv.2610.13100}

\bibitem[Win41]{Wintner}
A. Wintner,
The Analytical Foundations of Celestial Mechanics.
Princeton Math. Ser., vol.~5, Princeton University Press, 1941.

\bibitem[Xia91]{Xia91}
Z. Xia,
Central configurations with many small masses.
J. Differential Equations 91 (1991), no.~1, pp.~168--179.
\url{https://doi.org/10.1016/0022-0396(91)90137-X}

\bibitem[YZ26]{YZ}
X. Yu, S. Zhu,
On the finiteness of four-body central configurations.
Ann. Sc. Norm. Super. Pisa Cl. Sci. (5), to appear.
\url{https://doi.org/10.2422/2036-2145.202501_008}
\end{thebibliography}
\end{document}